\documentclass[11pt]{amsart}

\usepackage{geometry}
\usepackage{todonotes}
\newcommand\han[1]{\todo[color=yellow!40]{#1}} 

\usepackage{amssymb, adjustbox, enumerate, amsbsy, stmaryrd}
\usepackage{amsmath}
\usepackage[mathscr]{eucal}
\usepackage{amsthm}
\usepackage{bm}
\usepackage{graphicx}
\usepackage{tikz-cd}
\usetikzlibrary{matrix,arrows,decorations.pathmorphing}
\usepackage{amscd}
\usepackage{epic}
\usepackage{appendix}
\usepackage[all,color]{xy}
\usepackage{enumitem}
\usepackage{bbm}

\usepackage{color}
\usepackage{mathtools}
\usepackage{hyperref}
\usepackage{comment}
\hypersetup{
	colorlinks=true,
	linkcolor=blue,
	citecolor=red,
	filecolor=blue,
	urlcolor=blue
}

\theoremstyle{plain}
\newtheorem{theorem}{Theorem}[section]
\newtheorem{lemma}[theorem]{Lemma}
\newtheorem{proposition}[theorem]{Proposition}
\newtheorem{corollary}[theorem]{Corollary}
\newtheorem{conjecture}[theorem]{Conjecture}
\newtheorem{question}[theorem]{Question}
\newtheorem{cor}[theorem]{Corollary}
\newtheorem{thm}[theorem]{Theorem}
\newtheorem{prop}[theorem]{Proposition}
\newtheorem{lem}[theorem]{Lemma}
\newtheorem{conj}[theorem]{Conjecture}

\theoremstyle{definition}
\newtheorem{example}[theorem]{Example}
\newtheorem{definition}[theorem]{Definition}

\newtheorem{defn}[theorem]{Definition}

\theoremstyle{remark}
\newtheorem{remark}[theorem]{Remark}

\newcommand{\id}{\mathrm{id}}

\newcommand{\cC}{\mathcal{C}}
\newcommand{\cD}{\mathcal{D}}
\newcommand{\cF}{\mathcal{F}}
\newcommand{\cG}{\mathcal{G}}
\newcommand{\cE}{\mathcal{E}}

\newcommand{\cO}{\mathcal{O}}
\newcommand{\cZ}{\mathcal{Z}}

\newcommand{\cX}{\mathcal{X}}
\newcommand{\cY}{\mathcal{Y}}
\newcommand{\cT}{\mathcal{T}}
\newcommand{\cP}{\mathcal{P}}

\newcommand{\bA}{\mathbb{A}}
\newcommand{\bC}{\mathbb{C}}
\newcommand{\bP}{\mathbb{P}}
\newcommand{\bQ}{\mathbb{Q}}
\newcommand{\Qq}{\mathbb{Q}}

\newcommand{\bR}{\mathbb{R}}
\newcommand{\Rr}{\mathbb{R}}
\newcommand{\bZ}{\mathbb{Z}}
\newcommand{\Zz}{\mathbb{Z}}

\newcommand{\fD}{\mathfrak{D}}

\newcommand{\fa}{\mathfrak{a}}
\newcommand{\fb}{\mathfrak{b}}
\newcommand{\fm}{\mathfrak{m}}

\newcommand{\mult}{\mathrm{mult}}

\newcommand{\bxb}{B\subset \mathcal{X}\to B}

\newcommand{\vol}{\mathrm{vol}}
\newcommand{\ord}{\mathrm{ord}}

\newcommand{\lct}{\mathrm{lct}}
\newcommand{\Val}{\mathrm{Val}}
\newcommand{\nvol}{\widehat{\mathrm{vol}}}

\newcommand{\Diff}{\mathrm{Diff}}
\newcommand{\Supp}{\mathrm{Supp}\,}
\newcommand{\Spec}{\mathrm{Spec}}

\newcommand{\mld}{\mathrm{mld}}
\newcommand{\gr}{\mathrm{gr}}
\newcommand{\supj}{^{(j)}}
\newcommand{\supl}{^{(l)}}
\newcommand{\supn}{^{(n)}}

\newcommand{\Exc}{\mathrm{Exc}}

\newcommand{\hvol}{\widehat{\mathrm{vol}}}
\newcommand{\Vol}{\mathrm{Vol}}
\newcommand{\loc}{\mathrm{loc}}
\newcommand{\QM}{\mathrm{QM}}
\newcommand{\bk}{\mathbbm{k}}
\newcommand{\bG}{\mathbb{G}}
\newcommand{\e}{\mathrm{e}}

\newcommand{\cS}{\mathcal{S}}

\newcommand{\red}{\mathrm{red}}
\newcommand{\bxanb}{B\subset\cX^{\rm an}\to B}
\newcommand{\hX}{\widehat{X}}
\newcommand{\hR}{\widehat{R}}
\newcommand{\hx}{\hat{x}}
\newcommand{\hY}{\widehat{Y}}
\newcommand{\hmu}{\hat{\mu}}
\newcommand{\hS}{\widehat{S}}
\newcommand{\hGamma}{\widehat{\Gamma}}
\newcommand{\hv}{\hat{v}}
\newcommand{\Aut}{\mathrm{Aut}}

\newcommand{\tS}{\widetilde{S}}
\newcommand{\bK}{\mathbb{K}}
\newcommand{\tc}{\mathrm{tc}}
\newcommand{\bmu}{\bm{\mu}}
\newcommand{\hZ}{\widehat{Z}}
\newcommand{\hz}{\hat{z}}
\newcommand{\obK}{\overline{\mathbb{K}}}
\newcommand{\bN}{\mathbb{N}}
\newcommand{\Proj}{\mathrm{Proj}}
\newcommand{\hW}{\widehat{W}}

\numberwithin{equation}{section}
\newcommand\numberthis{\addtocounter{equation}{1}\tag{\theequation}}

\begin{document}

	\title[ACC for local volumes and boundedness of singularities]{ACC for local volumes and boundedness of singularities} 
	\date{\today}
	
		\author[J. ~Han]{Jingjun Han}
\address{Shanghai Center for Mathematical Sciences, Fudan University, Jiangwan Campus, Shanghai, 200438, China}
\email{hanjingjun@fudan.edu.cn}
\address{Department of Mathematics, The University of Uath, Salt Lake City, UT 84112, USA}
\email{jhan@math.utah.edu}
\address{Mathematical Sciences Research Institute, Berkeley, CA 94720, USA}
\email{jhan@msri.org}
	
	\author[Y. ~Liu]{Yuchen Liu}
	\address{
		Department of Mathematics, Northwestern University, Evanston, IL 60208, USA
	}
	\email{yuchenl@northwestern.edu}
	
	\author[L.~Qi]{Lu Qi}
	\address{
		Department of Mathematics, Princeton University,
		Princeton, NJ 08544, USA
	}
	\email{luq@princeton.edu}

	\begin{abstract}
	The ACC conjecture for local volumes predicts that the set of local volumes of klt singularities $x\in (X,\Delta)$ satisfies the ACC if the coefficients of $\Delta$ belong to a DCC set. 
	In this paper, we prove the ACC conjecture for local volumes under the assumption that the ambient germ is analytically bounded. We introduce another related conjecture, which predicts the existence of $\delta$-plt blow-ups of a klt singularity whose local volume has a positive lower bound.
	We show that the latter conjecture also holds when the ambient germ is analytically bounded. Moreover, we prove that both conjectures hold in dimension 2 as well as for 3-dimensional terminal singularities.\\

	\end{abstract}

	   
		\maketitle
	

	\tableofcontents
	
	\medskip

	\section{Introduction}
	

	
	Kawamata log terminal (klt) singularities form an important class of singularities which emerges from the study of the Minimal Model Program (MMP) (see e.g. \cite{BCHM10}). It becomes clear now that klt singularities appear naturally in other contexts: they form the right class of singularities of K-semistable or K\"ahler-Einstein Fano varieties 
	(see \cite{Oda13, LX14, DS14, BBEGZ19, CDS15, Tia15, BBJ15, LTW19} etc.); they share common properties with global Fano varieties, e.g. their (algebraic) fundamental groups are finite (see \cite{Xu14, GKP16, Bra20} etc.), and they always admit plt blow-ups whose exceptional divisors, known as \emph{Koll\'ar components}, are klt (log) Fano varieties (see \cite{Sho96, Pro00, Kud01, Xu14, LX16} etc.). 
	
	Recently, the study of the local volume of klt singularities, first introduced by C. Li in \cite{Li18}, has attracted lots of attention. Let us recall the definition below. Let $x\in (X,\Delta)$ be an $n$-dimensional klt singularity over an algebraically closed field of characteristic $0$. For any real valuation $v$ of $K(X)$ centered at $x$, its \emph{normalized volume} is defined as 
	\[
	\hvol_{(X,\Delta),x}(v):=A_{(X,\Delta)}(v)^n\cdot\vol(v),
	\]
	where $A_{(X,\Delta)}(v)$ is the log discrepancy of $v$ according to \cite{JM12, BdFFU15}, and $\vol(v)$ is the volume of $v$ according to \cite{ELS03}. The \emph{local volume} of the klt singularity $x\in (X,\Delta)$ is defined as 
	\[
	\hvol(x,X,\Delta):=\min_{v} \hvol_{(X,\Delta),x}(v),
	\]
	where the existence of a normalized volume minimizer was shown by Blum \cite{Blu18a}. Such a minimizer is always quasi-monomial by Xu \cite{Xu19} and unique up to rescaling by Xu and Zhuang \cite{XZ20}. The main purpose of Li's invention of the normalized volume functional was to establish a local K-stability theory for klt singularities. More precisely, according to the Stable Degeneration Conjecture \cite{Li18, LX18}, the $\hvol$-minimizer is expected to have a finitely generated graded algebra, which  degenerates $x\in (X,\Delta)$ to a K-semistable log Fano cone singularity. For an extensive discussion of progress on this conjecture, we refer to the survey article \cite{LLX18}.
	
	The local volume of a klt singularity is an important invariant which reflects essential geometric information and has deep connection to K-stability. It is shown by Li and Xu \cite{LX16} that a divisorial valuation minimizes $\hvol$ if and only if it comes from a K-semistable Koll\'ar component. For a quotient singularity $o\in \bA^n/G$, we know that $\hvol(o, \bA^n/G)=n^n/|G|$ by \cite[Example 7.1]{LX16}. Moreover, such a multiplicative formula holds for any finite crepant Galois morphism between klt singularities (known as the finite degree formula) by the recent work of Xu and Zhuang \cite[Theorem 1.3]{XZ20}. It is shown by the second author and Xu in \cite[Appendix A]{LX19} that $\hvol(x,X,\Delta)\leq n^n$ for any klt singularity $x\in (X,\Delta)$, where the equality holds if and only if $x\in X\setminus\Supp(\Delta)$ is smooth. 
	By works of Blum, the second author, and Xu \cite{BL18a,Xu19}, in a $\bQ$-Gorenstein family of klt singularities, the local volume of fibers is a lower semi-continuous and constructible function on the base. This leads to a proof of the openness of K-semistability \cite{Xu19} (for a different proof, see \cite{BLX19}). For a K-semistable log Fano pair, the local volume of any singularity is bounded from below by the global volume up to a constant \cite{Fuj18, Liu18, LL19}. Such an estimate is crucial in the study of explicit K-moduli spaces (see e.g. \cite{SS17, LX19, GMGS18, ADL19, ADL20, Liu20}). 
	Compared with the minimal log discrepancy (mld),  there is an inequality $\hvol(x,X,\Delta)< n^n\cdot \mld(x,X,\Delta)$ from \cite[Theorem 6.13]{LLX18}. A differential geometric interpretation of the local volume goes as follows: when $x\in X$ arises from a Gromov-Hausdorff limit of K\"ahler-Einstein Fano manifolds, Li and Xu \cite[Corollary 3.7]{LX18} showed that the local volume of $x\in X$ is the same as the volume density of its metric tangent cone up to a constant scalar (see also \cite{HS17, SS17}).
	
In this paper, we explore the relation between local volumes and the boundedness of singularities. Motivated by the finite degree formula which yields an effective upper bound of the order of the local fundamental group of a klt singularity in terms of its local volume (see \cite[Corollary 1.4]{XZ20}) and other phenomena from differential geometry (see e.g. \cite[Section 5.1]{SS17}), we expect that the existence of a positive lower bound for local volumes guarantees certain boundedness property on singularities. In addition, our expectation is closely related to the ACC conjecture on local volumes as local volumes of a bounded family of singularities take finitely many values  by \cite{Xu19}. 

	
	
	Below, we split our discussion into two parts. The first part treats the conjecture on discreteness and the ACC property for local volumes. The second part is focused on the conjecture which predicts the existence of $\delta$-plt blow-ups when the local volumes have a positive lower bound. Note  that the latter conjecture combined with \cite{HLM19} would imply that klt singularities whose local volumes have a positive lower bound are log bounded up to special degeneration.
	Our main results confirm the above conjectures for singularities $x\in (X,\Delta)$ in three cases: when $x\in X$ analytically belongs to a $\Qq$-Gorenstein bounded family, when the dimension is $2$, or when $x\in X$ is $3$-dimensional terminal and $\Delta=0$. We note that although the statements are divided into two parts, their proofs share the same strategy. 
	
	\subsection{ACC and discreteness of local volumes}
	
	In this subsection, we address the following folklore conjecture on the discreteness and the ACC for local volumes. 
	Note that part (1) was first stated in \cite[Question 6.12]{LLX18} (see also \cite[Question 4.3]{LX19}), and part (2) has appeared in \cite[Conjecture 8.4]{HLS19} as a natural extension of part (1). 
	
	\begin{conjecture}\label{conj:ACC for nvol}
	    Let $n$ be a positive integer and $I\subset[0,1]$ a subset. Consider the set of local volumes 
	\[
		\mathrm{Vol}^{\loc}_{n,I}\coloneqq\left\{\nvol(x,X,\Delta)\left| \begin{array}{l}x\in (X,\Delta) \text{\ is $n$-dimensional klt, where }\Delta=\sum_{i=1}^m a_i\Delta_i,\\
	    \text{$a_i\in I$ for any $i$, and each $\Delta_i\ge0$ is a Weil divisor.} 
		\end{array}\right.
		\right\}.
	\] 
		\begin{enumerate}
			\item If $I$ is finite, then $\Vol^{\loc}_{n,I}$ has $0$ as its only accumulation point.
			\item If $I$ satisfies the DCC, then 		    $\Vol^{\loc}_{n,I}$ satisfies the ACC.
		\end{enumerate} 
	\end{conjecture}
	
	We note that the first author, J. Liu, and Shokurov proved Conjecture \ref{conj:ACC for nvol} for exceptional singularities \cite[Theorem 8.5]{HLS19}. We also remark that a special case of part (1) that $n^n$ is not an accumulation point of $\Vol_{n,I}^{\rm loc}$ with $I=\{0\}$ is a weaker version of the ODP Gap Conjecture \cite[Conjecture 5.5]{SS17} which was verified in dimension at most $3$, see \cite{LL19, LX19}.
	
	Our first main result states that if $(x\in X^{\rm an})\in (B\subset \cX^{\rm an}\to B)$, that is, the ambient germ $x\in X$ analytically belongs to a $\Qq$-Gorenstein bounded family $(\bxb)$ (see Definition \ref{defn:family of singularities}), then the set of local volumes $\{\hvol(x,X,\Delta)\}$ satisfies the conclusion of Conjecture \ref{conj:ACC for nvol}. In particular, Theorem \ref{thm:discreteACC} implies that Conjecture \ref{conj:ACC for nvol} holds when $x\in X$ is a smooth germ.  
	
	\begin{theorem}\label{thm:discreteACC}
		Let $n$ be a positive integer and $I\subset [0,1]$ a subset. Let $B\subset \cX\to B$ be a $\bQ$-Gorenstein family of $n$-dimensional klt singularities. Consider the set of local volumes
		\[   
		\Vol_{B\subset\cX\to B, I}:=  \left\{ \nvol(x,X,\Delta)\left| \begin{array}{l}
		(x\in X^{\rm an})\in (B\subset \cX^{\rm an}\to B),~ x\in(X,\Delta) \text{ is klt},\\
		\text{where }\Delta=\sum_{i=1}^m a_i\Delta_i,~\text{$a_i\in I$ for any $i$,} \text{ and}\\
		\text{each $\Delta_i\ge0$ is a $\bQ$-Cartier Weil divisor.} 
		\end{array}\right.
		\right\}.
		\]
		\begin{enumerate}
			\item If $I$ is finite, then $\Vol_{B\subset\cX\to B, I}$ has no non-zero accumulation point. 
			\item If $I$ satisfies the DCC, then $\Vol_{B\subset\cX\to B, I}$ satisfies the ACC.
		\end{enumerate}
	\end{theorem}

	
	

If $x\in (X,\Delta)$ belongs to a log bounded family and $\Delta$ has finite rational coefficients, then Xu \cite[Theorem 1.3]{Xu19} proved that their local volumes belong to a finite set. We remark that Theorem \ref{thm:discreteACC} does not assume the boundedness of $\Supp\Delta$ and allows (DCC) real coefficients.

	The following theorem confirms Conjecture \ref{conj:ACC for nvol} in dimension $2$.
	\begin{theorem}\label{thm:surface ACC}
		Conjecture \ref{conj:ACC for nvol} holds when $n=2$.
	\end{theorem}
	
	We also show that the local volumes of $3$-dimensional terminal singularities without boundary divisors are discrete away from $0$. Note that these singularities (even the Gorenstein ones) are not analytically bounded (see e.g. \cite{Mor85, Rei87} or \cite[\S 5.3]{KM98}), and their local volumes (even the Gorenstein ones) can converge to $0$ (see e.g. \cite[Example 4.2]{LX19}).
	
	\begin{theorem}\label{thm:3fold-terminal-vol}
	The set of local volumes 
	\[
	\Vol_{3}^{\mathrm{term}}\coloneqq\{\nvol(x,X)\mid x\in X \text{\ is $3$-dimensional terminal}\}
	\]
	has $0$ as its only accumulation point.
	\end{theorem}
	
	Note that even if we assume the Stable Degeneration Conjecture \cite[Conjecture 1.2]{LX18} is true, Conjecture \ref{conj:ACC for nvol} is still open in dimension $n\geq 3$. This is essentially due to the lack of a boundedness result for K-semistable Fano cone singularities whose local volumes have a lower bound. To compare, the corresponding global boundedness result was proved by Jiang \cite{Jia17} and Xu-Zhuang \cite{XZ20} based on the BAB Conjecture proven by Birkar \cite{Bir16} and Batyrev's Conjecture proven by Hacon-McKernan-Xu \cite{HMX14} respectively. For related discussions, see Conjecture \ref{conj:bddspecial}, Question \ref{que:weil-index}, or \cite[Example 4.4]{LX19}.
	
	\subsection{Local volumes and boundedness of singularities} 
	
In this subsection, we study the relationship between local volumes and certain boundedness condition on singularities. We expect the following two classes of singularities are equivalent:

	\begin{equation}\label{eqn: categories of singularities}
	    \left\{\begin{array}{c}
	       x\in(X,\Delta) \text{ is $\epsilon_1$-lc, and admits a } \\ 
	       \delta\text{-plt blow-up for some fixed } 
	         \epsilon_1,\delta>0
	    \end{array}\right\}\simeq 
	    \left\{\begin{array}{c}
	   \nvol(x,X,\Delta)>\epsilon \\
	          \text{ for some fixed } \epsilon>0
	    \end{array}\right\}.
    \end{equation}
   We remark that it is expected in \cite{HLS19} and proved in \cite[Theorems 1.1 and 4.1]{HLM19} that the first class of singularities in \eqref{eqn: categories of singularities} belongs to a bounded family up to  special degeneration (See Section \ref{sec:spdeg} for the definition of special degenerations).

	\medskip
	
We first show that the local volumes of $n$-dimensional $\epsilon_1$-lc singularities with $\delta$-plt blow-ups have a positive lower bound depending only on $n$, $\epsilon_1$ and $\delta$, which confirms one direction of our expectation in \eqref{eqn: categories of singularities}.
	
	\begin{theorem}\label{thm:bddness for nvol}
		Let $n\ge 2$ be a positive integer and $\delta, \epsilon_1$ positive real numbers. Then there exists a positive real number $\epsilon$ depending only on $n,\epsilon_1$ and $\delta$ satisfying the following. 
		
		If $x\in (X,\Delta)$ is an $n$-dimensional klt singularity, such that
		\begin{enumerate}
			\item  $\mld(x,X,\Delta)\geq \epsilon_1$, and
			\item  $x\in(X,\Delta)$ admits a $\delta$-plt blow-up,
		\end{enumerate}
		then $\nvol(x,X,\Delta)\geq \epsilon$.
	\end{theorem}
	
	For the converse direction in \eqref{eqn: categories of singularities}, we propose the following conjecture.

	\begin{conjecture}\label{conj:existence of delta-plt blow-up}
		Let $n\ge 2$ be a positive integer and $\eta,\epsilon$ positive real numbers. Then there exists a positive real number $\delta$ depending only on $n,\eta$ and $\epsilon$ satisfying the following. 
		If $x\in (X,\Delta=\sum_{i=1}^m a_i\Delta_i)$ is an $n$-dimensional klt singularity such that		\begin{enumerate}
			\item $a_i>\eta$ for any $i$, 
			\item each $\Delta_i\geq 0$ is a Weil divisor, and
			\item $\nvol(x,X,\Delta)>\epsilon$,
		\end{enumerate}
		then $x\in(X,\Delta)$ admits a $\delta$-plt blow-up.
	\end{conjecture}
	
	We prove that the statement of Conjecture \ref{conj:existence of delta-plt blow-up} is true if $x\in X$ analytically belongs to a $\Qq$-Gorenstein bounded family.
	
	\begin{theorem}\label{thm:existence of delta-plt blow-up}
		Let $n\ge 2$ be a positive integer, $\eta,\epsilon $ positive real numbers, and $B\subset \cX\to B$ an $\bQ$-Gorenstein family of $n$-dimensional klt singularities. Then there exists a positive real number $\delta$ depending only on $n,\eta,\epsilon$ and $B\subset \cX\to B$ satisfying the following.
		
		If $x\in (X,\Delta=\sum_{i=1}^m a_i\Delta_i)$ is an $n$-dimensional klt singularity such that
		\begin{enumerate}
		\item  $(x\in X^{\rm an})\in (B\subset \cX^{\rm an}\to B)$,
			\item $a_i>\eta$ for any $i$, 
				\item each $\Delta_i\geq 0$ is a $\bQ$-Cartier Weil divisor, and
			\item $\nvol(x,X,\Delta)>\epsilon$,
		\end{enumerate}
		then $x\in(X,\Delta)$ admits a $\delta$-plt blow-up.
	\end{theorem}
	
	We note that Theorem \ref{thm:existence of delta-plt blow-up} fails to hold without assuming condition (2), that is, the existence of a positive lower bound on the nonzero coefficients, see Example \ref{eg:coeff no lower bound}.


    Similar to Theorems \ref{thm:surface ACC} and \ref{thm:3fold-terminal-vol}, we also confirm Conjecture \ref{conj:existence of delta-plt blow-up} in dimension $2$ and for $3$-dimensional terminal singularities without boundary divisors. 
	 
	

	\begin{theorem}\label{thm:delta-plt-blowup-lowdim}
		Conjecture \ref{conj:existence of delta-plt blow-up} holds in the following two situations.
		\begin{enumerate}
		    \item $n=2$.
		    \item $n=3$, $\Delta=0$, and $x\in X$ is terminal.
		\end{enumerate}
	\end{theorem}
	

	


		An immediate consequence of Theorem \ref{thm:existence of delta-plt blow-up} and \cite[Theorem 1.3]{HLS19} is that under the conditions of Theorem \ref{thm:existence of delta-plt blow-up}, the ACC conjecture for minimal log discrepancies holds. Recall that the ACC conjecture for minimal log discrepancies is closely related to the termination of flips \cite{Sho04} and is still open in dimension at least $3$ even when $x\in X$ is fixed. For other recent progress on minimal log discrepancies, we refer the readers to \cite{JLiu18,Kaw18,Jia19,JLX19,Mor20}.
	

	\begin{corollary}\label{cor: accmld when nv bdd}
		Let $n$ be a positive integer, $I\subset [0,1]$ a set which satisfies the DCC, $\epsilon$ a positive real number, and $B\subset \cX\to B$ a $\bQ$-Gorenstein family of $n$-dimensional klt singularities. Then the set 
		
		$$ \left\{\mld(x,X,\Delta)\left| \begin{array}{l} (x\in X^{\rm an})\in (B\subset \cX^{\rm an}\to B),~ x\in(X,\Delta) \text{ is klt},\\
		\text{where }\Delta=\sum _{i=1}^m a_i\Delta_i,~\text{$a_i\in I$ for any $i$},~\text{each $\Delta_i\ge 0$}\\
		\text{is a $\bQ$-Cartier Weil divisor},~\text{and $\nvol(x,X,\Delta)>\epsilon$.}	\end{array}\right.
		\right\}$$
			satisfies the ACC.
			
	Moreover, if $I$ is a finite set, then the only possible accumulation point of the above set is $0$.  
	\end{corollary}
	
		Theorem \ref{thm:boundedness of global index} answers a folkfore question on the boundedness of the Cartier index of any $\Qq$-Cartier Weil divisor in a log bounded family (see \cite[Question 3.31]{HLS19}). We refer readers to \cite[Theorem 1.10]{GKP16}, \cite[Lemma 2.24]{Bir19}, and \cite[Lemma 7.14]{CH20} for some partial results. Our approach to show Theorem \ref{thm:boundedness of global index} is based on Theorem \ref{thm:existence of delta-plt blow-up} and \cite[Theorem 1.2]{Bir18}. 
		
	\begin{theorem}\label{thm:boundedness of global index} Let $\epsilon$ be a positive real number. Suppose $\cC\coloneqq\{(X,\Delta)\}$ is a set of $\epsilon$-lc projective pairs that belongs to a log bounded family $\cP$. Then there exists a positive integer $N$ which only depends on $\cP$ and $\epsilon$ satisfying the following.
	
	Let $(X,\Delta)\in \cC$, and $D$ a $\Qq$-Cartier Weil divisor on $X$. Then $ND$ is Cartier.
		\end{theorem}
	
	
\noindent
\textbf{Sketch of proofs}. We first sketch the proofs of Theorems \ref{thm:discreteACC} and \ref{thm:existence of delta-plt blow-up}. For simplicity, in both theorems, we assume that $x\in X$ is fixed, the coefficients of $\Delta$ belong to a rational finite set, and $\nvol(x,X,\Delta)$ has a positive lower bound. By the boundedness of Cartier index of any $\bQ$-Cartier Weil divisor on $X$, we may further assume that each $\Delta_i$ is Cartier. Our idea is to reduce both theorems to the case when $\Supp\Delta$ belongs to a bounded family, and then we may apply the constructibility of local volumes in a log bounded family proved by Xu \cite[Theorem 1.3]{Xu19}, and the existence of ``good'' $\delta$-plt blow-ups in a log bounded family (see Theorem \ref{thm: family kc with nv bdd}). The reduction follows from two steps. In step 1, we show that there exists a positive integer $k$ depending only on positive lower bounds of both $\hvol(x,X,\Delta)$ and $\lct(X,\Delta;\Delta)$, such that if $\Delta^k$ is 
a $k$-th truncation of $\Delta$, then $\nvol(x,X,\Delta)=\nvol(x,X,\Delta^k)$ (see Theorem \ref{thm:truncation}). Moreover, we show that any ``good'' $\delta$-plt blow-up of $x\in (X,\Delta^k)$ is also a $\delta$-plt blow-up of $x\in (X,\Delta)$ (see Proposition \ref{prop: truncation keep nv(v)}). Our argument is inspired by generic limit constructions from \cite{Kol08, dFEM10, dFEM11} and a truncation argument in \cite{Xu19} based on Li's properness estimate \cite{Li18}. In step 2, we establish an inequality $c\cdot \lct(X,\Delta;\Delta)\geq \hvol(x,X,\Delta)$ where $c$ is a positive constant depending only on $x\in X$ (see Theorem \ref{thm:estimate on lct}). This shows that the constant $k$ from step 1 can be chosen to depend only on the positive lower bound of $\hvol(x,X,\Delta)$, so we get the boundedness of $\Delta^k$. Here a ``good'' $\delta$-plt blow-up means that $\nvol_{(X,\Delta),x}(\ord_S)$ is bounded from above where $S$ is the induced Koll\'ar component. 


It is worthwhile to mention that many results for local volumes were only proved for $\Qq$-divisors $\Delta$ in previous literature, and some key ingredients in their proofs including the existence of monotonic $n$-complement \cite[Theorem 1.8]{Bir19} fail for $\Rr$-divisors. Thus one technical difficulty in our paper is to generalize these results to the case where $\Delta$ is an $\Rr$-divisor and the coefficient set $I$ is not finite. To resolve this issue, we generalize \cite[Main Theorem]{Blu18a} and \cite[Theorems 1.2 and 1.3]{Xu19} from $\bQ$-divisors to $\bR$-divisors (see Section 3) and prove a Lipschitz type estimate on local volumes (see Theorem \ref{thm:uniform lipschitz}). Another technical difficulty is that we need to treat analytic boundaries and analytically bounded families, so that together with Theorems \ref{thm:discreteACC}, \ref{thm:existence of delta-plt blow-up}, and classification results, we can prove Conjectures \ref{conj:ACC for nvol} and \ref{conj:existence of delta-plt blow-up} in dimension 2 as well as for 3-dimensional terminal singularities.

\medskip

\noindent \textbf{Acknowledgement.} This work began when the first author visited the second author at Yale University in October 2018. The first author would like to thank their hospitality. Part of this work was done when the third author visited MIT in the fall semester of 2018 and Johns Hopkins University several times during 2018--2020. The third author would like to thank their hospitality.

We would like to thank Chen Jiang, Jihao Liu, Mircea Musta\c{t}\u{a}, Zhiyu Tian, Xiaowei Wang, Chenyang Xu, Qizheng Yin and Ziquan Zhuang for helpful discussions. We also thank Harold Blum and Guodu Chen for useful comments. YL is partially supported by the NSF Grant No. DMS-2148266 (transferred from the NSF Grant No.
DMS-2001317). JH was supported by a grant from the Simons Foundation (Grant Number 814268, MSRI). LQ would like to thank his advisor Chenyang Xu for constant support, encouragement, and numerous inspiring conversations. 
	
	\section{Preliminaries}
	\subsection{Pairs and singularities}
	Throughout this paper, we work over an algebraically closed field $\bk$ of characteristic $0$ unless it is specified.
	
	We adopt the standard notation and definitions in \cite{KM98}, and will freely use them.
	
	\begin{definition}[Pairs and singularities]\label{defn: positivity}
		A pair $(X,\Delta)$ consists of a normal quasi-projective variety $X$ and an $\Rr$-divisor $\Delta\ge0$ such that $K_X+\Delta$ is $\Rr$-Cartier. Moreover, if the coefficients of $\Delta$ are $\leq 1$, then $\Delta$ is called a boundary of $X$. If moreover $\Delta$ has $\bQ$-coefficients, then we say that $(X,\Delta)$ is a $\bQ$-pair.
		
		Let $E$ be a prime divisor on $X$ and $D$ an $\mathbb R$-divisor on $X$. We define $\mult_E D$ to be the multiplicity of $E$ along $D$. Let $\phi:W\to X$
		be any log resolution of $(X,\Delta)$ and let
		$$K_W+\Delta_W:=\phi^{*}(K_X+\Delta).$$
		The \emph{log discrepancy} of a prime divisor $E$ on $W$ with respect to $(X,\Delta)$ is defined as 
		\[
		A_{(X,\Delta)}(E):=1-\mult_{E}\Delta_W.
		\]
		For any positive real number $\epsilon$, we say that $(X,\Delta)$ is lc (resp. klt, $\epsilon$-lc, $\epsilon$-klt) if $A_{(X,\Delta)}(E)\ge0$ (resp. $>0$, $\ge \epsilon$, $>\epsilon$) for every log resolution $\phi:W\to X$ as above and every prime divisor $E$ on $W$. We say that $(X,\Delta)$ is plt (resp. $\epsilon$-plt) if $A_{(X,\Delta)}(E)>0$ (resp. $>\epsilon$) for any exceptional prime divisor $E$ over $X$. Note that a prime divisor $E$ over $X$ is simply a prime divisor $E$ on some log resolution $W$ of $X$. The center of $E$ on $X$ (denoted by $c_X(E)$) is the scheme theoretic point $\phi(\eta)\in X$ where $\eta$ is the generic point of $E$.
		
		A \emph{singularity} $x\in (X,\Delta)$ consists of a pair $(X,\Delta)$ and a closed point $x\in X$. The singularity $x\in (X,\Delta)$ is called an lc (resp. a klt, an $\epsilon$-lc) singularity if there exists an open neighborhood $U$ of $x$ in $X$ such that $(U,\Delta|_U)$ is lc (resp. klt, $\epsilon$-lc). The \emph{minimal log discrepancy} of an lc singularity $x\in (X,\Delta)$ is defined as 
		$$\mld(x,X,\Delta):=\min\{A_{(X,\Delta)}(E)\mid E \text{ is a prime divisor} \text{ over } X \text{ with } c_{X}(E)=x\}.$$
		The singularity $x\in (X,\Delta)$ is called $\epsilon$-lc if $\mld(x,X,\Delta)\ge\epsilon$.
	\end{definition}

	\begin{definition}[Log canonical thresholds]
Let $x\in (X,\Delta)$ be an lc singularity and let $D$ be an effective $\mathbb{R}$-Cartier $\Rr$-divisor. The \emph{log canonical threshold} of $D$ with respect to $x\in (X,\Delta)$ is
\begin{equation*}
\lct_{x}(X,\Delta;D)\coloneqq {\rm sup}\{t \in\mathbb{R}\mid x\in (X,\Delta+tD) {\rm~is~log~canonical}\}.
\end{equation*}
For convenience, we will denote $\lct_{x}(X,\Delta;D)$ by $\lct(X,\Delta;D)$ if $x$ is clear from the context. Similarly, we may define the log canonical threshold $\lct(X,\Delta;\fa)$ (resp. $\lct(X,\Delta;\fa_\bullet)$) of an ideal $\fa$ (resp. a graded sequence of ideals $\fa_\bullet$) with respect to $x\in (X,\Delta)$, see, for example, \cite[Definition 3.4.1]{Blu18b}.
\end{definition}
	
		
	
Next we give some estimates on order functions. 
	
	\begin{definition}
	Let $X$ be a normal variety, $x\in X$ a closed point, and $\fm_{X,x}$ the maximal ideal of the local ring $\mathcal{O}_{X,x}$ at $x$. The \emph{order function} $\ord_x: \mathcal{O}_{X,x}\to \Zz_{\ge0} \cup\{\infty\}$ is defined by
	\[
        \ord_{x}(f):=\sup \left\{j \geq 0 \mid f \in     \fm_{X,x}^{j}\right\}.
    \]
    This is a valuation if $x$ is a smooth point, but not in general. Let $\Delta=\mathrm{div}(f)$ be an effective Cartier divisor, where $f\in \mathcal{O}_{X,x}$, we define $\ord_{x}(\Delta)\coloneqq\ord_{x}(f)$. We remark that $\ord_{x}(\Delta)$ is well-defined, that is, $\ord_{x}(\Delta)$ is independent on the choice of $f$.

	\end{definition}
	
	\begin{proposition}\label{prop: order of a klt boundary}
		Let $x\in (X,\Delta)$ be a klt singularity of dimension $n$. 
		\begin{enumerate}
			\item  $\lct(X,\Delta;\fm_{X,x})\leq n$, where $\fm_{X,x}\subseteq\cO_{X,x}$ is the maximal ideal of $x$.
			\item  	Suppose that $x\in (X,\Delta+c\Delta_0)$ is a klt singularity for some positive real number $c$ and Cartier divisor $\Delta_0=\mathrm{div}(f)$, where $f\in \cO_{X}$. Then $\ord_{x}(f)<\frac{n}{c}$.
		\end{enumerate}
	\end{proposition}
	\begin{proof}
		(1) By lower-semi-continuity of log canonical thresholds in a family (Lemma \ref{lem: lctlowersc}), there exists a closed smooth point $x'\in X$, such that $n= \lct_{x'}(X;\fm_{X,x'})\ge \lct_{x'}(X,\Delta;\fm_{X,x'}) \ge \lct_{x}(X,\Delta;\fm_{X,x})$, where $\fm_{X,x'}\subseteq\cO_{X,x'}$ is the maximal ideal of $x'$.
		
		(2) Let $a\coloneqq\ord_{x}(f)$. Then $f\in \fm_{X,x}^{a}$, and $(X,\Delta+c\fm_{X,x}^{a})$ is klt. By (1), $ca< n$. Hence $\ord_{x}(f)< \frac{n}{c}$. 
	\end{proof}

	We also need the sub-additivity of log canonical thresholds \cite[Corollary 2]{JM08}.
	
	\begin{proposition}\label{prop:sub-additivity for lct}
		Let $x\in (X,\Delta)$ be an lc singularity where $X$ is $\bQ$-Gorenstein. For any ideal sheaves $\fa,\fb$ on $X$ whose cosupports contain $x$, we have
		\[
		\lct(X,\Delta;\fa+\fb)\le \lct(X,\Delta;\fa)+\lct(X,\Delta;\fb).
		\]
	\end{proposition}
	\begin{proof}
	 The proposition follows from \cite[Corollary 2]{JM08}.
	\end{proof}

\begin{defn}[Bounded families]\label{defn: bdd}
	A \emph{couple} consists of a normal projective variety $X$ and a divisor $D$ on $X$ such that $D$ is reduced. Two couples $(X,D)$ and $(X',D')$ are \emph{isomorphic} if there exists an isomorphism $X\rightarrow X'$ mapping $D$ onto $D'$.
	
	A set $\mathcal{P}$ of couples is \emph{bounded} if there exist finitely many projective morphisms $V^i\rightarrow T^i$ of varieties and reduced divisors $C^i$ on $V^i$ such that for each $(X,D)\in\mathcal{P}$, there exists $i$ and a closed point $t\in T^i$, such that the couples $(X,D)$ and $(V^i_t,C^i_t)$ are isomorphic, where $V^i_t$ and $C^i_t$ are the fibers over $t$ of the morphisms $V^i\rightarrow T^i$ and $C^i\rightarrow T^i$, respectively. 
	
	A set $\mathcal{C}$ of projective pairs $(X,B)$ is said to be \emph{log bounded} if the corresponding set of couples $\{(X,\Supp B)\}$ is bounded. A set of projective varieties $X$ is said to be \emph{bounded} if the corresponding set of couples $\{(X,0)\}$ is bounded. A log bounded (respectively bounded) set is also called a \emph{log bounded family} (respectively \emph{bounded family}).
\end{defn}

	\subsection{Normalized volumes of valuations}
	
	
	In this section we give the definition of normalized volumes of valuations from \cite{Li18}. Note that our definition slightly generalizes Li's definition as we treat $\Rr$-pairs. Throughout this section, we denote by $X$ a normal variety.
	
	\subsubsection{Valuations}
	A \emph{valuation} $v$ of  $K(X)$ is a function $v:K(X)^{\times}\to \bR$ satisfying the following conditions:
	\begin{itemize}
		\item $v(fg)=v(f)+v(g)$;
		\item $v(f+g)\geq \min\{v(f),v(g)\}$;
		\item $v(c)=0$ for $c\in \bk^{\times}$.
	\end{itemize}
	We also set $v(0)=+\infty$. Every valuation $v$ of $K(X)$ gives rise to a valuation ring $\cO_v:=\{f\in K(X)\mid v(f)\geq 0\}$. The value group of $v$ is the (abelian) subgroup $\Gamma_{v}:=v(K(X)^{\times})$ of $\Rr$. 
	
	Let $\xi\in X$ be a scheme-theoretic point.
	We say a valuation $v$ of $K(X)$ is \emph{centered at} $\xi=c_X(v)$ if its valuation ring $\cO_v$ dominates $\cO_{X,\xi}$ as local rings. We denote by $\Val_X$ the set of all valuations of $K(X)$ admitting a center on $X$. We denote by $\Val_{X,\xi}$ the subset of $\Val_X$ consisting of valuations centered at $\xi$. Note that the center of a valuation is unique if it exists due to separatedness of $X$.
	
	For $v\in \Val_X$ and a non-zero ideal sheaf $\fa\subset \cO_X$, we define 
	\[
	v(\fa):=\min\{v(f)\mid f\in \fa\cdot \cO_{X,\xi}\textrm{ where }\xi=c_X(v)\}.
	\]
	We endow $\Val_X$ with the weakest topology such that for any non-zero ideal sheaf $\fa\subset \cO_X$, the map $\Val_X\to \bR_{\geq 0}$ defined as  $v\mapsto v(\fa)$ is continuous.
	
	Given a valuation $v\in \Val_{X,\xi}$ and a real number $p$, we define the \emph{valuation ideal sheaf} $\fa_p(v)$ as $\fa_p(v)(U):=\{f\in \cO_{X}(U)\mid v(f)\geq p\}$. It is clear that the cosupport of $\fa_p(v)$ is $\overline{\{\xi\}}$ for $p>0$. In particular, if $\xi=x$ is a closed point on $X$ and $v\in \Val_{X,x}$, then $\fa_p(v)$ is an $\fm_x$-primary ideal for $p>0$.
	
	Let $\mu:Y\to X$ be a birational morphism from a normal variety $Y$. Hence $\mu^*:K(X)\to K(Y)$ is an isomorphism.
	Let $E\subset Y$ be a prime divisor. Then $E$ induces a valuation $\ord_E$ of $K(X)$ by assigning each rational function $f\in K(X)$ to the order of vanishing of $\mu^* f$ along $E$. A valuation $v\in \Val_X$ is a \emph{divisorial valuation} if $v=\lambda\cdot\ord_E$ for some prime divisor $E$ over $X$ and some $\lambda\in \bR_{>0}.$
	
	Let $(Y,D)$ be a \emph{log smooth model over} $X$, that is, $\mu: Y\to X$ is a proper birational morphism from a smooth variety $Y$, the divisor $D$ is reduced simple normal crossing on $Y$, and $\mu$ is an isomorphism on $Y\setminus\Supp(D)$. 
	Let $\bm{y}=(y_1,\cdots, y_r)$ be a system of algebraic coordinates at a scheme-theoretic point $\eta\in Y$. We assume that each divisor $(y_i=0)$ near $\eta$ is equal to an irreducible component of $D$. Let  $\bm{\alpha}=(\alpha_1,\cdots, \alpha_r)\in \bR_{\geq 0}^r$ be a vector. We define a valuation $v_{\bm{\alpha}}$ as follows. Since by Cohen's structure theorem we have $\widehat{\cO_{Y,\eta}}\cong \kappa(\eta)\llbracket y_1,y_2,\cdots , y_r\rrbracket$, any function $f\in \cO_{Y,\eta}$ has a Taylor expansion $
	f=\sum_{\bm{\beta}\in \bZ_{\geq 0}^r} c_{\bm{\beta}} \bm{y}^{\bm{\beta}}$, where $\bm{y}^{\bm{\beta}}:=\prod_{i=1}^r y_i^{\beta_i}$ and
		$c_{\bm{\beta}}\in \widehat{\cO_{Y,\eta}}$ is either $0$ or a unit.
	Then we define  $v_{\bm{\alpha}}(f):= \min\{\langle \bm{\alpha},\bm{\beta}\rangle\mid c_{\bm{\beta}}\neq 0\}$, where $\langle \bm{\alpha},\bm{\beta}\rangle:= \sum_{i=1}^r \alpha_i\beta_i$.
	A valuation $v\in \Val_X$ is \emph{quasi-monomial} if $v=v_{\bm{\alpha}}$ for some log smooth model $(Y,D)$ over $X$, a system of algebraic coordinates $\bm{y}$ at $\eta\in Y$, and $\bm{\alpha}\in \bR_{\geq 0}^r$. For a fixed log smooth model $(Y,D)$ over $X$ and $\eta\in Y$, we denote $\QM_\eta(Y,D)$ to be the collection of all quasi-monomial valuations $v_{\bm{\alpha}}$ that can be described as above at the point $\eta\in Y$. We define $\QM(Y,D):=\cup_{\eta}\QM_\eta(Y,D)$ where $\eta$ runs through all generic points of intersections of some irreducible components of $D$.

	\subsubsection{Log discrepancy}
	Let $\Delta$ be an effective $\Rr$-divisor on $X$ such that $K_X+\Delta$ is $\Rr$-Cartier, i.e. $(X,\Delta)$ is a pair. In this subsection, we define log discrepancy $A_{(X,\Delta)}(v)$ of a valuation $v\in \Val_{X}$ with respect to $(X,\Delta)$ following \cite{JM12, BdFFU15}. Note that a log smooth pair $(Y,D)$ is said to \emph{dominate} $(X,\Delta)$ if $(Y,D)$ is a log smooth model over $X$ and $\mu^{-1}(\Supp(\Delta))\subset \Supp(D)$.
	
	
	\begin{defn}  Let $v$ be a valuation of $K(X)$.
		\begin{enumerate}
			\item If $v=\lambda\cdot \ord_E$ is divisorial where $E\subset Y\xrightarrow{\mu} X$ is a prime divisor over $X$, then we define the \emph{log discrepancy of $v$ with respect to} $(X,\Delta)$ as
			\[
			A_{(X,\Delta)}(v):=\lambda\cdot A_{(X,\Delta)}(E) = \lambda(1 + \mult_E(K_Y-\mu^*(K_X+\Delta))).
			\]
			\item If $v=v_{\bm{\alpha}}$ is a quasi-monomial valuation that can be described at the point $\eta\in Y$ with respect to a log smooth model $(Y,D=\sum_{i=1}^{l}D_i)$ dominating $(X,\Delta)$ such that $D_i=(y_i=0)$ near $\eta$ for $1\leq i\leq r\leq l$, then we define the \emph{log discrepancy of $v$ with respect to} $(X,\Delta)$ as
			\[
			A_{(X,\Delta)}(v):= \sum_{i=1}^r \alpha_i\cdot A_{(X,\Delta)}(D_i).
			\]
			\item It was shown in \cite{JM12} that there exists a retraction map $r_{Y,D} : \Val_X \to \QM(Y, D)$ for any log smooth model $(Y, D)$ dominating $(X,\Delta)$, such that it induces a homeomorphism $\Val_X \xrightarrow{\cong} \varprojlim_{(Y,D)} \QM(Y, D)$. For any valuation $v \in \Val_X$, we define the \emph{log discrepancy of $v$ with respect to} $(X,\Delta)$ as
		\[
		A_{(X,\Delta)}(v):= \sup_{(Y,D)} A_{(X,\Delta)}(r_{Y,D}(v)) \in \bR\cup \{+\infty\},
		\]
		where the supremum is taken over all log smooth pairs $(Y,D)$ dominating $(X,\Delta)$. It is possible that $A_{X,\Delta}(v) = +\infty$  for some $v \in \Val_X$, see e.g. \cite[Remark 5.12]{JM12}.

		\end{enumerate}
		
	\end{defn}

	We collect some useful lemmata which are easy consequences of \cite{JM12} (see e.g. \cite[Lemma 5.3 and Remark 5.6]{JM12}).
	\begin{lem}
		The pair $(X,\Delta)$ is klt (resp. lc) if and only if for any non-trivial valuation $v\in \Val_{X}$ we have $A_{(X,\Delta)}(v)>0$ (resp. $\geq 0$).
	\end{lem}
	
	
	\begin{lem}\label{lem:A-pull-back}
		Let $(X,\Delta)$ and $(X',\Delta')$ be two pairs together with a proper birational morphism $\phi:X'\to X$. Then for any $v\in \Val_X$ we have
		\[
		A_{(X',\Delta')}(v)=A_{(X,\Delta)}(v)-v((K_{X'}+\Delta')-\phi^*(K_X+\Delta)).
		\]
	\end{lem}
	

	\subsubsection{Normalized volumes}
	
	In this subsection, we recall the definition of normalized volumes of Li \cite{Li18} for an $n$-dimensional klt singularity $x\in (X,\Delta)$. First we recall the definition of the volume of a valuation from \cite{ELS03}.
	
	\begin{defn}
		For a valuation $v\in \Val_{X,x}$, we define the \emph{volume} of $v$ by
		\[
		\vol_{X,x}(v):=\lim_{m\to+\infty}\frac{\ell(\cO_{X,x}/\fa_{m}(v))}{m^n/n!}.
		\]
		Here $\ell(\cdot)$ denotes the length of an Artinian module.
	\end{defn}
	
	We define $\Val_{X,x}^\circ:=\{v\in\Val_{X,x}\mid A_{(X,\Delta)}(v)<+\infty\}$. Note that this definition is independent of the choice of $\Delta$ by Lemma \ref{lem:A-pull-back}.
	
	\begin{defn}
		For a valuation $v\in \Val_{X,x}$, we define the \emph{normalized volume of $v$ with respect to} $x\in (X,\Delta)$ as
		\[
		\hvol_{(X,\Delta),x}(v):=\begin{cases}
		A_{(X,\Delta)}(v)^n\cdot \vol_{X,x}(v) & \textrm{ if } v\in \Val_{X,x}^\circ,\\
		+\infty & \textrm{ if }v\not\in \Val_{X,x}^\circ.
		\end{cases}
		\]
		The \emph{local volume} of a klt singularity $x\in (X,\Delta)$ is defined as
		\[
		\hvol(x,X,\Delta):=\inf_{v\in \Val_{X,x}} \hvol_{(X,\Delta),x}(v).
		\]
		When $\Delta$ is a $\Qq$-divisor, the existence of a $\hvol$-minimizer is proven by Blum \cite[Main Theorem]{Blu18a} when $\bk$ is uncountable, and by Xu \cite[Remark 3.8]{Xu19} in general. Such a minimizer is always quasi-monomial by \cite[Theorem 1.2]{Xu19} and unique up to rescaling by \cite[Theorem 1.1]{XZ20}. We will prove that both \cite[Main Theorem]{Blu18a} and \cite[Theorem 1.2]{Xu19} hold for any $\Rr$-divisor $\Delta\ge0$ and any algebraically closed field $\bk$; see Theorem \ref{thm:quasi-monomial}. Meanwhile, the proof of uniqueness of $\hvol$-minimizers from \cite{XZ20} can be easily generalized to $\bR$-divisors $\Delta$ (see Theorem \ref{thm:uniqueness}). By convention, we set $\hvol(x,X,\Delta')=0$ for a pair $(X,\Delta')$ that is not klt at $x$.
	\end{defn}

	The following theorem provides useful estimates on local volumes. It is a combination of \cite[Corollary 3.4]{Li18}, \cite[Theorem 1.6]{LX19}, and \cite[Theorem 6.13]{LLX18}. 
	
	\begin{thm}[loc. cit.]\label{thm:nvol is bounded by n^n}
		Let $x\in(X,\Delta)$ be an $n$-dimensional klt singularity. Then 
		\[
		0<\nvol(x,X,\Delta)\leq n^n\cdot \min\{1,\mld(x,X,\Delta)\}.
		\]
	\end{thm}
	
	The following lemma from \cite{Liu18} provides an alternative characterization of local volumes in terms of log canonical thresholds and multiplicities. A proof in the $\bQ$-pair case is provided in \cite[Proof of Theorem 2.6]{LLX18}. 
	
	\begin{lem}[{\cite[Theorem 27]{Liu18}}]\label{lem:nv=lcte}
		Let $x\in(X,\Delta)$ be an $n$-dimensional klt singularity.
		Then
		\[
		\hvol(x,X,\Delta)=\inf_{\fa\colon \fm_x\textrm{-primary}} \lct(X,\Delta;\fa)^n\cdot \e(\fa)=\inf_{\fa_\bullet\colon \fm_x\textrm{-primary}} \lct(X,\Delta;\fa_\bullet)^n\cdot \e(\fa_\bullet),
		\]
		where $\e(\fa)$ is the Hilbert-Samuel multiplicity of $\fa$, and $\e(\fa_\bullet)={\displaystyle\lim_{m\to+\infty}}\frac{\e(\fa_m)}{m^n}$.
	\end{lem}
	
	We note that although Theorem \ref{thm:nvol is bounded by n^n} and Lemma \ref{lem:nv=lcte} were originally proven for $\bQ$-pairs, their proofs generalize to the pair case with little change.

	The following properness and Izumi type estimates from \cite{Li18} are important in the study of normalized volumes. Note that although the original statements in \cite{Li18} assume $\Delta=0$, Li's proof generalizes easily to the pair setting by taking a log resolution of $(X,\Delta)$. We provide a proof for readers' convenience. For a family version, see Lemma \ref{lem:family Izumi}. 
	
	\begin{lem}[{\cite[Theorems 1.1 and 1.2]{Li18}}]\label{lem:izumi}
	Let $x\in (X,\Delta)$ be a klt singularity. Denote $\fm:=\fm_{X,x}$ the maximal ideal at $x$. Then there exist positive real numbers $C_1,C_2$ depending only on $x\in (X,\Delta)$ such that for any $f\in \cO_{X,x}$ and any $v\in \Val_{X,x}$, we have
	\begin{enumerate}
	    \item(Properness estimate)
	    \[
	        \nvol_{(X,\Delta),x}(v)\ge C_1\frac{A_{(X,\Delta)}(v)}{v(\fm)}.
	    \]
	    \item(Izumi type estimate)
	    \[
	        v(\fm)\ord_x(f)\leq v(f)\leq C_2 A_{(X,\Delta)}(v) \ord_x(f).
	    \]
	\end{enumerate}
	
	\end{lem}
	
	\begin{proof}
	 We first prove part (2), i.e. the Izumi type estimate.
	 The first inequality is obvious. For the second inequality, we choose a log resolution $\mu: X'\to (X,\Delta)$ with $K_{X'}+\Delta'=\mu^*(K_X+\Delta)$. Since $(X,\Delta)$ is klt, there exists $\epsilon>0$ such that $\Delta'\leq (1-\epsilon)\Delta'_{\red}$. Since $\Delta'_{\red}$ is simple normal crossing, we know that $(X',\Delta'_{\red})$ is lc. Hence by Lemma \ref{lem:A-pull-back} we have
	 \begin{align*}
	 A_{(X,\Delta)}(v)=A_{X'}(v)-v(\Delta')& \geq A_{X'}(v)-(1-\epsilon)v(\Delta'_{\red})\\ &=\epsilon A_{X'}(v) +(1-\epsilon) A_{(X', \Delta'_{\red})} (v) \geq \epsilon A_{X'}(v).
	 \end{align*}
	 Let $\xi\in X'$ be the center of $v$ on $X'$.  By Izumi's inequality in the smooth case (see \cite[Proposition 5.1]{JM12}), for any $f\in \cO_{X,x}$ we have
	 \[
	 v(f)=v(\mu^* f)\leq A_{X'}(v) \ord_{\xi}(\mu^*f) \leq \epsilon^{-1} A_{(X,\Delta)}(v) \ord_{\xi}(\mu^*f).
	 \]
	 By Izumi's linear complementary inequality (see \cite[Theorem 3.2]{Li18}), there exists $a_2\geq 1$ depending only on $x\in X$ and $\mu$ such that $\ord_{\xi}(\mu^* f)\leq a_2 \ord_x(f)$. Hence (2) is proved by taking $C_2=\epsilon^{-1} a_2$. 
	 
	 Now (1) follows from (2) and \cite[Theorem 1.3]{Li18}.
	\end{proof}

	
We will also need the finite degree formula for normalized volumes which is conjectured by the second author and Xu \cite[Conjecture 4.1]{LX19} and proved by Xu-Zhuang \cite{XZ20}. Note that although the result was originally stated for $\bQ$-divisors, the proof of Xu and Zhuang can be easily generalized to $\bR$-divisors as it is a consequence of the uniqueness of minimizers (see Theorem \ref{thm:uniqueness}). 

\begin{theorem}[Finite degree formula, {cf. \cite[Theorem 1.3]{XZ20}}]\label{thm: finite degree formula} Let $y \in\left(Y, \Delta_{Y}\right)$ and $x \in(X, \Delta)$ be two klt singularities. 
Let $f:\left(y \in\left(Y, \Delta_{Y}\right)\right) \rightarrow\left(x \in(X, \Delta)\right)$ be a finite Galois morphism such that $f(y)=x$, and $K_{Y}+\Delta_{Y}=f^{*}(K_{X}+\Delta)$. Then
\[
\hvol(x, X, \Delta) \cdot \deg(f)=\hvol(y, Y, \Delta_{Y}).
\]
\end{theorem}	

We also include an easy but useful lemma.

	\begin{lemma}\label{lem:kill boundary}
	Let  $x\in(X,\Delta)$ be an $n$-dimensional klt singularity where $\Delta$ is $\Rr$-Cartier. Assume that $\lct(X,\Delta;\Delta)\geq\gamma$ for some  $\gamma>0$, then for any $v\in\Val_{X,x}$ we have
		$$A_{(X,\Delta)}(v)\geq\left(\frac{\gamma}{1+\gamma}\right)A_X(v),\quad \textrm{ and}\quad
		\nvol_{(X,\Delta),x}(v)\geq \left(\frac{\gamma}{1+\gamma}\right)^{n}\nvol_{X,x}(v).
		$$		
		
	\end{lemma}
	
	\begin{proof}
		This follows from the inequality $A_{(X,(1+\gamma)\Delta)}(v)=A_{X}(v)-(1+\gamma)v(\Delta)\ge 0$.
	\end{proof}
	
	\subsection{Koll\'ar components}
	
	\begin{definition}\label{def:kc}
		Let $x\in(X,\Delta)$ be a klt singularity. If a projective birational morphism $\mu:Y\to X$ from a normal variety $Y$ satisfies the following properties:
		\begin{enumerate}
			\item $\mu$ is isomorphic over $X\backslash\{x\}$,
			\item $\mu^{-1}(x)$ is an irreducible exceptional divisor $S$,
			\item $(Y,S+\mu^{-1}_*\Delta)$ is plt near $S$, and
			\item $-S$ is an $\mu$-ample $\bQ$-Cartier divisor,
		\end{enumerate}
		then we call $\mu$ a \emph{plt blow-up} of $x\in(X,\Delta)$ and $S$ a \emph{Koll\'ar component} of $x\in (X,\Delta)$. Moreover, if for a positive real number $\delta$ we have
		\begin{enumerate}
			\item[(3')] $(Y,S+\mu^{-1}_*\Delta)$ is $\delta$-plt near $S$,
		\end{enumerate}
		then we call $\mu$ a $\delta$-plt blow-up and $S$ a $\delta$-Koll\'ar component of $x\in (X,\Delta)$.
	\end{definition}
	
	\begin{proposition}[{\cite[Lemma 2.13]{LX16}}]\label{prop:finite deg formula}
		Let $\sigma:(x'\in (X',\Delta'))\to (x\in (X,\Delta))$ be a finite morphism between klt singularities such that $\sigma(x')=x$, and $\sigma^*(K_X+\Delta)=K_{X'}+\Delta'$. If $\mu:Y\to X$ is a plt blow-up of $x\in (X,\Delta)$ with the Koll\'ar component $S$, then 
		\begin{enumerate}
			\item $Y\times_X X'\to X'$ induces a Koll\'ar component $S'$ of $x'\in (X',\Delta')$, and
			\[
			\deg(\sigma)\cdot\nvol_{(X,\Delta),x}(\ord_S)=\nvol_{(X',\Delta'),x'}(\ord_{S'}).
			\]
			\item If in addition $\sigma$ is a Galois quotient morphism of a finite subgroup $G<\Aut(x'\in (X', \Delta'))$, then every $G$-invariant Koll\'ar component $S'$ over $x'\in (X',\Delta')$ arises as a pull-back of a Koll\'ar component $S$ over $x\in (X,\Delta)$. 
		\end{enumerate}
		
	\end{proposition}

	The following lemma is well-known to experts (see e.g. \cite[Proof of Theorem 1.3]{HX09}, \cite[Lemmata 3.7 and 3.8]{LX16}, \cite[Corollary 3.5]{Fuj19b}, or \cite[Lemma 4.8]{Zhu20}). 
	
	\begin{lem}\label{lem:kc-compute-lct}
	    Let $x\in (X,\Delta)$ be a klt singularity. Let $\fa$ be an ideal sheaf on $X$ cosupported at $x$. Then there exists a Koll\'ar component $S$ computing $\lct(X,\Delta;\fa)$. 
	\end{lem}

	The following result generalizes \cite[Theorem 1.3]{LX16} to $\bR$-divisors.
		\begin{theorem}\label{thm:lx-kc-minimizing} 
	Let $x\in (X,\Delta)$ be a klt singularity. Then 
	\begin{enumerate}
	    \item 
	$\hvol(x,X,\Delta)=\inf_{S}\hvol_{(X,\Delta),x}(\ord_S)$,
	where $S$ runs over all Koll\'ar components over $x\in (X,\Delta)$, and
	\item if $v_*\in \Val_{X,x}$ minimizes $\hvol_{(X,\Delta),x}$, then there exists a sequence of Koll\'ar components $\{S_k\}$  and positive numbers $b_k$ such that 
	\[
	\lim_{k\to+\infty} b_k\cdot\ord_{S_k} = v_* \textrm{ in }\Val_{X,x}\quad \textrm{and}\quad \lim_{i\to+\infty}\hvol_{(X,\Delta),x}(\ord_{S_k})=\hvol(x,X,\Delta).
	\]
	\end{enumerate}
	\end{theorem}
	
	\begin{proof}
	(1) The direction ``$\leq $'' is obvious. Thus it suffices to show that for any positive real number $\epsilon$, there exists a Koll\'ar component $S$ over $x\in (X,\Delta)$ such that $\hvol_{(X,\Delta),x}(\ord_S)\leq \hvol(x,X,\Delta)+\epsilon$. By Lemma \ref{lem:nv=lcte}, there exists an ideal sheaf $\fa$ on $X$ cosupported at $x$ such that $\lct(X,\Delta;\fa)^n\cdot \e(\fa)\leq \hvol(x,X,\Delta)+\epsilon$. By Lemma \ref{lem:kc-compute-lct}, there exists a Koll\'ar component $S$ computing $\lct(X,\Delta;\fa)$. Hence we have 
	\[
	\hvol_{(X,\Delta),x}(\ord_S)\leq \lct(X,\Delta;\fa)^n\cdot \e(\fa)\leq \hvol(x,X,\Delta)+\epsilon,
	\]
	where the first inequality follows from \cite[Lemma 26]{Liu18}. 
	
	The proof of part (2) is the same as that of \cite[Theorem 1.3]{LX16}, and we omit it.
	\end{proof}
	
	\begin{theorem}[{\cite[Theorem 1.2]{LX16}}]\label{thm:minimizer is K-ss}
		Let $x\in(X,\Delta)$ be a klt singularity where $\Delta\ge0$ is a $\bQ$-divisor. Then a divisorial valuation $\ord_S$ is a minimizer of $\nvol_{(X,\Delta),x}$ if and only if $S$ is a Koll\'ar component of $x\in(X,\Delta)$ and $(S,\Delta_S)$ is $K$-semistable, where $\mu:Y\to X$ is the corresponding plt blow-up of $x\in (X,\Delta)$, and	$\Delta_S$ is the different divisor of $K_Y+\mu_{*}^{-1}\Delta+S$ on $S$. 
	\end{theorem}

	\subsection{Analytically isomorphic singularities}
	
	\begin{defn}
		We say two singularities $(x\in X)$ and $(x'\in X')$ are analytically isomorphic (denoted by $(x\in X^{\rm an})\cong (x'\in X'^{\rm an})$) if we have an isomorphism $\widehat{\cO_{X,x}}\cong \widehat{\cO_{X',x'}}$ of $\bk$-algebras.
	\end{defn}
	
	Here we use the notion ``analytically isomorphic'' as ``formally isomorphic'' in literature, although the former notion (over $\bC$) usually refers to isomorphic as complex analytic germs. Note that a famous result of Artin \cite[Corollary 2.6]{Art69} shows that formally isomorphic singularities have isomorphic \'etale neighborhoods, hence over $\bC$ the two notions are equivalent.

	We will use the following Proposition without citing it frequently. 
	\begin{prop}\label{prop: Q-Gorestein}
		Assume that $(x\in X)$ and $(x'\in X')$ are analytically isomorphic singularities. Then  $(x\in X)$ is $\bQ$-Gorenstein if and only if $(x'\in X')$ is $\bQ$-Gorenstein. Moreover, the Cartier index of $K_X$ near $x$ is the same as the Cartier index of $K_{X'}$ near $x'$. 
	\end{prop}
	
	\begin{proof}
		Denote $R:=\cO_{X,x}$ and $R':=\cO_{X',x'}$. Let $\widehat{R}$ and $\widehat{R'}$ be their completions. Then we have an isomorphism $\widehat{R}\cong\widehat{R'}$. 
		Since both dimension and depth are preserved under completion, we know that $R$ is Cohen-Macaulay if and only if $\widehat{R}\cong \widehat{R'}$ is Cohen-Macaulay,  if and only if $R'$ is Cohen-Macaulay. For a finite $R$-module $M$ and $m\in\bZ_{>0}$, we denote $M^{[m]}:=(M^{\otimes m})^{**}$.
		Thus it suffices to show that $\omega_R^{[m]}$ is free if and only if $\omega_{R'}^{[m]}$ is free for $m\in\bZ_{>0}$. Here $\omega_{A}$ denotes the canonical module of a Cohen-Macaulay ring $A$. By \cite[Theorem 3.3.5]{BH93}, we know that $\omega_{\widehat{R}}\cong \omega_R\otimes_R \widehat{R}$. Since $R\hookrightarrow\widehat{R}$ and $R'\hookrightarrow \widehat{R'}$ are faithfully flat, we know that 
		\[
		\omega_{\widehat{R}}^{[m]}\cong \omega_R^{[m]}\otimes_R \widehat{R}, \quad \textrm{and}\quad 			\omega_{\widehat{R'}}^{[m]}\cong \omega_{R'}^{[m]}\otimes_{R'} \widehat{R'}.
		\]
		Hence $\omega_R^{[m]}$ is free if and only if  $\omega_{\widehat{R}}^{[m]}\cong\omega_{\widehat{R'}}^{[m]}$ is free, if and only if  $\omega_{R'}^{[m]}$ is free.
	\end{proof}
	
	Recall that for a klt singularity $x\in X$, the space $\Val_{X,x}^{\circ}$ consists of valuations $v\in \Val_{X,x}$ satisfying $A_X(v)<+\infty$.
	
	\begin{prop}\label{prop:an-iso}
	Assume that $(x\in X)$ and $(x'\in X')$ are analytically isomorphic singularities where $(x\in X)$ is klt. Then  $(x'\in X')$ is also klt. Moreover, there exists a bijection $\phi: \Val_{X,x}^\circ\to \Val_{X',x'}^\circ$ such that the following statements hold for any $v\in \Val_{X,x}^\circ$.
	\begin{enumerate}
	    \item We have $A_X(v)=A_{X'}(\phi(v))$.
	    \item We have $\gr_{v}\cO_{X,x}\cong \gr_{\phi(v)}\cO_{X',x'}$ as graded rings. In particular, $\vol_{X,x}(v)=\vol_{X',x'}(\phi(v))$.
	    \item We have 
	    $\hvol_{X,x}(v)=\hvol_{X',x'}(\phi(v))$ and  $\hvol(x,X)=\hvol(x',X')$.
	    \item If $v=\ord_S$ is a Koll\'ar component $S$ of $(x\in X)$, then $\phi(v)=\ord_{S'}$ is a Koll\'ar component $S'$ of $(x\in X')$, and $(S,\Gamma)\cong (S',\Gamma')$ where $\Gamma$ and $\Gamma'$ are different divisors. 
	\end{enumerate}
	\end{prop}
	
	\begin{proof}
	For simplicity, denote $(R,\fm):=(\cO_{X,x},\fm_{X,x})$ and $(R',\fm'):=(\cO_{X',x'},\fm_{X',x'})$. Let $(\widehat{R},\widehat{\fm})$ and $(\widehat{R'},\widehat{\fm'})$ be the completion of $(R,\fm)$ and $(R',\fm')$ respectively.
	Since $x\in X$ is klt, by \cite[Proposition 2.11(1)]{dFEM11} we know that $\Spec\,\widehat{R}$ is klt in the sense of \cite[Page 226]{dFEM11}. Hence $x'\in X'$ is also klt by \cite[Proposition 2.11(1)]{dFEM11} and the isomorphism $\Spec\,\widehat{R}\cong \Spec\,\widehat{R'}$.
	
	Next we construct the bijection $\phi$. By \cite[Corollary 5.11]{JM12}, any valuation $v\in \Val_{X,x}^\circ$ has a unique extension $\hat{v}$ to $\Spec \widehat{R}$. Note that although \cite[Corollary 5.11]{JM12} has the assumption that $R$ is regular, the same argument goes through for any klt singularity $x\in X$ by replacing the Izumi inequality \cite[Proposition 5.10]{JM12} with Lemma \ref{lem:izumi}. Denote by $\psi: \widehat{R}\xrightarrow{\cong} \widehat{R'}$ the isomorphism. Then we may define $\phi(v):= (\psi_* \hat{v})|_{R'} \in \Val_{X',x'}$.
	
	
	Let $\pi:W\to X$ be a log resolution of $X$. Denote by $\hX:= \Spec ~\hR$ and $\widehat{X'}:=\Spec~\widehat{R'}$. Let $\hW:=W\times_X \hX$ with $\hat{\pi}: \hW\to \hX$. By \cite[Proposition A.14]{dFEM11}, we have $\hat{\pi}^* K_{W/X}=K_{\hW/\hX}$. By \cite[Proposition 5.13]{JM12}, we have that $A_W(v)=A_{\hW}(\hv)$. Thus by Lemma \ref{lem:A-pull-back} we have
	\begin{equation}\label{eq:completion-1}
	A_{X}(v)= A_W(v) + v(K_{W/X}) = A_{\hW}(\hv) + \hv(K_{\hW/\hX}).
	\end{equation}
	Since $\hX\cong \hX'$ by assumption, we know that $\hW\to \hX'$ is a log resolution in the sense of \cite{Tem18}.
	Let $\pi':W'\to X'$ be a log resolution of $X'$. Denote by $\widehat{W'}:= W'\times_{X'}\widehat{X'}$. Thus $\widehat{W'}\to \widehat{X'}$ is also a log resolution. Thus by \cite[Remark 5.6]{JM12} and the above arguments, we have
	\begin{align*}\label{eq:completion-2}
	A_{\hW}(\hv) + \hv(K_{\hW/\hX}) & = A_{\widehat{W'}} (\psi_* \hv) + \psi_* \hv(K_{\widehat{W'}/\widehat{X'}})\numberthis\\ & = A_{W'}(\phi(v)) + \phi(v) (K_{W'/X'})=A_{X'}(\phi(v)).
	\end{align*}
	Combining \eqref{eq:completion-1} and \eqref{eq:completion-2}, we get $A_{X}(v)=A_{X'}(\phi(v))$. 
	Hence $\phi$ takes value in $\Val_{X',x'}^{\circ}$. Similarly we can define $\phi^{-1}$ which implies that $\phi: \Val_{X,x}^{\circ}\to \Val_{X',x'}^\circ$ is a bijection. In addition, we have shown part (1). 
	
	For part (2), we first show that $\fa_p(v)\cdot \hR=\fa_p(\hv)$ for any $p\in \bR_{\geq 0}$. Since $\hv$ is an extension of $v$, we have $\fa_p(v)\cdot \hR\subset\fa_p(\hv)$. On the other hand, 
	suppose $f\in \fa_p(\hv)\setminus \{0\}$, then let $m\in \bN$ be an integer such that $m\cdot v(\fm)> \hv(f)$. Since $\hv(\widehat{\fm}^m)=\hv(\fm^m\cdot \hR)=v(\fm^m)=mv(\fm)$, we know that $\hv(\widehat{\fm}^m)>\hv(f)\geq p$. Choose $g\in R$ such that $f-g\in \widehat{\fm}^m$, then $v(g)=\hv(f)\geq p$. Thus we have $g\in \fa_p(v)$ and $\fm^m\subset \fa_p(v)$ which implies $f\in (g)+\widehat{\fm}^m\subset \fa_p(v)\cdot \hR$. As a result, we have $\fa_p(\hv)\subset \fa_p(v)\cdot \hR$ which implies $\fa_p(v)\cdot \hR=\fa_p(\hv)$. 
	
	Since all valuation ideals $\fa_p(v)$ of $v$  are $\fm$-primary, we have $\fa_p(v)/\fa_{>p}(v)\cong \fa_p(\hv)/\fa_{>p}(\hv)$ for any $p\in \bR_{\geq 0}$. Thus we have  $\gr_v R\cong \gr_{\hat{v}} \widehat{R}$ as graded rings. Apply similar arguments to $\phi(v)$ and $\widehat{\phi(v)}=\psi_* \hv$, we get $\gr_{\phi(v)} R'\cong \gr_{\psi_*\hat{v}}\widehat{R'}$. Since $\psi:\hR\to \widehat{R'}$ is an isomorphism, we get 
	\[
	\gr_v R\cong \gr_{\hat{v}} \widehat{R}\cong \gr_{\psi_*\hat{v}}\widehat{R'}\cong \gr_{\phi(v)} R'.
	\]
	From the isomorphism $\gr_v R\cong \gr_{\phi(v)} R'$, we know that $\ell(R/\fa_p(v))=\ell(R'/\fa_p(\phi(v)))$ for any $p\in \bR_{\geq 0}$. Thus the volumes of $v$ and $\phi(v)$ are equal. This finishes the proof of part (2).
	
	Part (3) is a consequence of parts (1) and (2). 
	
	For part (4), suppose $v=\ord_S$ for a Koll\'ar component $S$ over $(x\in X)$. Since $v$ and $\phi(v)$ have isomorphic associated graded algebras by part (2), we know that the value group of $\phi(v)$ is the same as that of $v$, which is $\bZ$. Let $Y':= \Proj_{X'} \oplus_{m\in \bZ_{\geq 0}} \fa_m(\phi(v))$, where the finite generation of this graded algebra follows from the finite generation of $\gr_{\phi(v)} R'$ (see e.g. \cite[Lemma 32]{Liu18}). Clearly $Y'$ is normal as any valuation ideal sequence is integrally closed. Denote by $\mu':Y'\to X'$ the projection morphism, then $\mu'$ is isomorphic over $X'\setminus\{x'\}$ as  $\fa_m(\phi(v))$ is $\fm'$-primary. Let $S':=\Proj ~\gr_{\phi(v)} R'$ as a closed subscheme of $Y'$. Then by construction we know that $\Supp S'=\mu'^{-1}(x')$. By \cite[Section 2.4]{LX16} and part (2), we know that $S\cong \Proj~ \gr_v R \cong \Proj~ \gr_{\phi(v)} R'= S'$. Hence $S'$ is the only prime $\mu'$-exceptional divisor on $Y'$. Let $k\in \bZ_{>0}$ be an integer such that $\fa_{km}(\phi(v))=\fa_k(\phi(v))^m$ for any $m\in \bN$. Then we know that $\cO_{Y'}(k)$ is Cartier ample over $X'$, which implies that $\cO_{Y'}(k)\cong \cO_{Y'}(-qS')$ for some $q\in \bZ_{>0}$. It is clear that $\fa_{km}(\phi(v))=\mu'_* \cO_Y(km)$, thus we have
	$\fa_{km}(\phi(v)) = \mu'_* \cO_Y(-qmE)= \fa_{km} (\frac{k}{q}\ord_{S'})$ for any $m\in \bN$. Thus we have $\phi(v)=\frac{k}{q}\ord_{S'}$, which implies that $k=q$ and $\phi(v)=\ord_{S'}$ by comparing their value groups. In particular, we have $\fa_m(\phi(v))=\mu'_* \cO_{Y'}(-mS')$ for any $m\in \bN$. 
	
	Let $o$ and $o'$ be the cone vertices of $\Spec~\gr_{v}R$ and $\Spec~\gr_{\phi(v)}R'$ respectively.
	By \cite[Section 2.4]{LX16} we know that $o\in \Spec~\gr_{v}R$ is a klt singularity carring a $\bG_m$-action induced by the grading of $\gr_{v}R$, such that $(\Spec~\gr_v R)\setminus \{o\}$ is a Seifert $\bG_m$-bundle over $(S,\Gamma)$ in the sense of \cite{Kol04}. Thus by part (2) $\Spec~ \gr_{\phi(v)}R'$ is also a klt singularity with a $\bG_m$-action induced by the grading of $\gr_{\phi(v)}R'$. By \cite[Proof of Lemma 2.21(1)]{LWX18}, we know that $\mu':Y'\to X'$ provides a Koll\'ar component $S'$ with different divisor $\Gamma'$, such that $(\Spec~ \gr_{\phi(v)}R')\setminus \{o'\}$ is a Seifert $\bG_m$-bundle over $(S', \Gamma')$. Since $\gr_v R\cong \gr_{\phi(v)}R'$ as graded rings by part (2), we know that $(S,\Gamma)\cong (S', \Gamma')$ as $\bG_m$-quotients of isomorphic Serfert $\bG_m$-bundles. The proof is finished.
	\end{proof}
	

	\subsection{Family of singularities}
	\begin{definition}[{\cite{BL18a,Xu19}}]\label{defn:family of singularities}
		We call $B\subset (\cX,\cD)\to B$ a $\bQ$-Gorenstein (resp. an $\Rr$-Gorenstein) family of ($n$-dimensional) klt singularities over a (possibly disconnected) normal base $B$, if 
		\begin{enumerate}
			\item $\cX$ is normal and flat over $B$, 
			
			
			\item $K_{\cX/B}+\cD$ is $\bQ$-Cartier (resp. $\Rr$-Cartier),
			
			\item for any closed point $b\in B$, $\cX_b$ is connected, normal, and not contained in $\Supp(\cD)$,
			
			\item there is a section $B\subset \cX$, and 
			
			\item $b\in (\cX_b,\cD_b)$ is klt (of dimension $n$) for any closed point $b\in B$, where $\cD_b$ is the (cycle theoretic) restriction of $\cD$ over $b\in B$.
		\end{enumerate}
		
		Let $x\in X$ be a normal variety $X$ with a closed point $x$.
		Let $\bxb$ be a $\bQ$-Gorenstein family of klt singularities.
		We denote by $(x\in X)\in (B\subset \cX\to B)$, if there exists a closed point $b\in B$, a neighborhood $U$ of $x\in X$, and a neighborhood $U_b$ of $b\in \cX_b$, such that $(x\in U)$ is isomorphic to $(b\in U_b)$. We denote by $(x\in X^{\rm an})\in (B\subset \cX^{\rm an}\to B)$, if there exists a closed point $b\in B$ such that $\widehat{\cO_{X,x}}\cong \widehat{\cO_{\cX_b,b}}$ as $\bk$-algebras.
	\end{definition}
	\begin{remark}\label{rem: basechage q-gorestein family sing}
		Let $B'\to B$ be any morphism from a normal scheme $B'$ of finite type over $\bk$, the base change $B'\subset (\cX',\cD')=(\cX ,\cD)\times_{B} B'\to B'$ is a $\Qq$-Gorenstein (resp. $\Rr$-Gorenstein) family of klt singularities over $B'$, and $K_{\cX'/B'}+\cD'= g^{*}(K_{\cX/B}+\cD)$, where $g : \cX'\to \cX$ is the base change of $B'\to B$, see \cite[Proposition 8]{BL18a}.
	\end{remark}
	
	\begin{definition}
	Let $B\subset (\cX,\cD)\to B$ be an $\Rr$-Gorenstein family of klt singularities over a normal base $B$. We say a birational morphism $\mu:(\cY,\cE)\to (\cX,\cD)$ is a \emph{fiberwise log resolution} of $B\subset (\cX,\cD)\to B$ where $\cE$ is the sum of the strict transform of $\cD$ and the reduced exceptional divisor of $\cY\to \cX$, if 
	\begin{enumerate}
	    \item for each closed point $b\in B$, $(\cY_b,\cE_b)\to (\cX_b,\cD_b)$ is a log resolution,
	    \item any stratum of $(\cY,\cE)$, that is a component of the intersection $\cap \cE_i$ for components $\cE_i$ of $\cE$, has geometric irreducible fibers over $B$, and
	    \item for any exceptional prime divisor $\cF$ of $\mu$, the center of $\cF$ on $\cX$ is the section $B\subset \cX$ if and only if the center of $\cF_b$ on $\cX_b$ is $b\in \cX_b$ for some closed point $b\in B$.
	\end{enumerate}
		\end{definition}
	\begin{remark}
         For any $\Rr$-Gorenstein family of klt singularities $B\subset (\cX,\cD)\to B$ over a normal base $B$, by \cite[Definition-Lemma 2.8]{Xu19}, possibly stratifying the base $B$ into a disjoint union of finitely many constructible subsets and taking finite \'{e}tale coverings, we may assume that there exists a decomposition $B=\bigsqcup_{\alpha} B_{\alpha}$ into irreducible smooth strata $B_{\alpha}$ such that for each $\alpha$,  $(\cX\times_{B}B_{\alpha},\cD\times B_{\alpha})$ admits a fiberwise log resolution $\mu_{\alpha}$. In particular, exists a positive real number $\epsilon$, such that $b\in (\cX_b,\cD_b)$ is $\epsilon$-lc for any closed point $b\in B$.
	\end{remark}
	
	The next lemma shows that log canonical thresholds in $\Rr$-Gorenstein families are constructible and lower semi-continuous. For $\bQ$-Gorenstein families it is stated in \cite[Proposition 10]{BL18a} (see also \cite[Corollary 2.10]{Amb16}). We omit the proof since it is the same with \cite{Amb16}.
	
	\begin{lem}\label{lem: lctlowersc}
		Let $(\cX,\cD) \to B$ be an $\Rr$-Gorenstein family of klt singularities over a normal base $B$. Let $\fa$ be an ideal sheaf on $\cX$. Then 
		\begin{enumerate}
			\item The function $b\mapsto \lct(\cX_b,\cD_b;\fa_b)$ on $B$ is constructible;
			\item If in addition $V(\fa)$ is proper over $B$, then $b\mapsto \lct(\cX_b,\cD_b;\fa_b)$ on $B$ is lower semi-continuous with respect to the Zariski topology.
		\end{enumerate}
	\end{lem}
	
	The following lemma states a well known result on the klt locus in a family. See \cite[Corollary 2.10]{Amb16} for a similar statement. We omit the proof here because it follows from arguments similar to those in \cite{Amb16}.
	\begin{lemma}\label{lem: klt locus open}
		Let $B\subset (\cX,\cD)\to B$ be an $\bR$-Gorenstein family of klt singularities over a normal base $B$, and $\cE$ an effective $\Rr$-Cartier $\Rr$-divisor on $\cX$ such that $\Supp(\cE)$ does not contain any fiber $\cX_b$. Then 
		$$\{b\in B\mid (\cX_b,\cD_b+\cE_b) \text{ is klt near }b\in \cX_b\}$$
		is a Zariski open subset of $B$.
	\end{lemma}

	The following result is a variation of \cite[Theorems 20 and 21]{BL18a}, which is the generalization of \cite[Theorem 1.1 and 1.2]{Li18} to the case of $\bQ$-Gorenstein families of singularities.
	
	\begin{lemma}\label{lem:family Izumi}
  Let $B\subset \cX \to B$ be a $\bQ$-Gorenstein family of klt singularites over a normal base $B$. Then there exist positive constants $C_1,C_2$ depending only on $B\subset \cX \to B$ such that the following holds.
	
	
	If  a klt singularity $x\in X$ satisfies that $(x\in X^{\rm an})\in (B\subset \cX^{\rm an}\to B)$, then for any valuation $v\in\Val_{X,x}$ and any $f\in \cO_{X,x}$, we have
		
		\begin{enumerate}
			\item(properness estimate)
			\begin{equation*}
			\nvol_{X,x}(v)\geq C_1\frac{A_{X}(v)}{v(\fm_{X,x})}.
			\end{equation*}

			\item(Izumi type estimate)
			\begin{equation*}
			v(\fm_{X,x})\ord_x(f)\leq v(f)\leq C_2 A_{X}(v)\ord_x(f).
			\end{equation*}
		\end{enumerate}
	\end{lemma}
	
	\begin{proof}
	 Let $b\in B$ be the closed point such that $(x\in X^{\rm an})\cong (b\in \cX_{b}^{\rm an})$. By \cite[Theorems 20 and 21]{BL18a} there exists positive constants $C_1$ and $C_2$ depending only on $B\subset \cX\to B$ such that both (1) and (2) hold for the klt singularity $b\in \cX_{b}$. We claim that the same constants $C_1$ and $C_2$ work for $x\in X$ as well. We may assume that $A_X(v)<+\infty$ since otherwise the statements are trivial. By Proposition \ref{prop:an-iso}, any $v\in \Val_{X,x}^\circ$ corresponds to a unique valuation $v'\in \Val_{\cX_b,b}^\circ$ such that $A_X(v)=A_{\cX_b}(v')$ and $\hvol_{X,x}(v)=\hvol_{\cX_b,b}(v')$. Denote $\fm:=\fm_{X,x}$ and $\fm':=\fm_{\cX_b,b}$. Since all valuation ideals of $v$ (resp. $v'$) are $\fm$-primary (resp. $\fm'$-primary), we know that $v(\fm)=\hat{v}(\widehat{\fm})= \hat{v'}(\widehat{\fm'})=v'(\fm')$. Hence (1) is proven. For (2), notice that this is equivalent to $\fa_{C_2 A_X(v) k}(v)\subset \fm^{k}$. This is true since similar statement for $v'$ holds and both valuation ideals are $\fm$-primary or $\fm'$-primary. The proof is finished.
	\end{proof}
	
		
		
	
		\subsection{Family of Koll\'ar components}
	
	\begin{defn}
	Let $B\subset (\cX,\cD)\xrightarrow{\pi} B$ be an $\bR$-Gorenstein family of klt singularities over a normal irreducible base $B$. A proper birational map $\mu:\cY\to \cX$ is said to provide \emph{a flat family of Koll\'ar components $\cS$ over $(\cX,\cD)$ centered at $B$} if the following conditions hold.
	\begin{itemize}
	    \item $\cY$ is normal, $\mu$ is an isomorphism over $\cX\setminus B$, and $\cS=\Exc(\mu)$ is a prime divisor on $\cY$ with $\mu(\cS)=B$.
	    \item $\pi\circ\mu:\cY\to B$ is flat with normal connected fibers.
	    \item $\cS$ does not contain any fiber of $\pi\circ\mu$.
	    \item $-\cS$ is $\bQ$-Cartier and $\mu$-ample.
	    \item For any closed point $b\in B$, the pair $(\cY_b, \cS_b+(\mu_*^{-1}\cD)|_{\cY_b})$ is plt near $\cS_b$. In other words, $\mu_b:\cY_b\to \cX_b$ provides a Koll\'ar component $\cS_b$ over $b\in (\cX_b,\cD_b)$.
	\end{itemize}
	
	Suppose that $B$ is normal reducible. We say that $\mu:\cY\to \cX$ provides a flat family of Koll\'ar components if for each irreducible component $B_i$ of $B$, the restriction $\mu_i:\cY\times_B B_i\to \cX\times_B B_i$ of $\mu$ over $B_i$ provides a flat family of Koll\'ar components.
	\end{defn}
	
	\begin{proposition}
	Let $B\subset (\cX,\cD)\to B$ be an $\bR$-Gorenstein family of klt singularities over a normal base. Let $\mu:\cY\to \cX$ be a proper birational map providing a flat family of Koll\'ar components $\cS$ over $(\cX,\cD)$ centered at $B$. Let $\Gamma$ be the different divisor of $(\cY,\cS+\mu_*^{-1}\cD)$ along $\cS$. Then $\mu|_{\cS}: (\cS,\Gamma)\to B$ is an $\bR$-Gorenstein family of log Fano pairs. 
	\end{proposition}
	
	\begin{theorem}\label{thm: family kc with nv bdd}
		Let $B\subset (\cX,\cD)\to B$ be an $\bR$-Gorenstein family of klt singularities over a normal base. Then there exist a positive real number $\delta$, a quasi-finite surjective morphism $B'\to B$ from a normal scheme $B'$, and a proper birational morphism $\cY'\to \cX'$ which provides a flat family of Koll\'ar components $\cS'$ over $(\cX',\cD'):=(\cX,\cD)\times_B B'$ centered at $B'$ satisfying the following. 
		
		For any closed point $b'\in B'$,
		\begin{enumerate}
		    \item  $\hvol_{(\cX'_{b'},\cD'_{b'}), b'}(\cS'_{b'})\leq n^n+1$, and
		    \item $\cS'_{b'}$ is a $\delta$-Koll\'ar component of $b'\in (\cX_{b'},\cD_{b'})$.
		\end{enumerate}
	\end{theorem}
	\begin{proof}
	First of all, we may assume that $B$ is irreducible. By Noetherian induction, it suffices to find an  open immersion $B'\hookrightarrow B$ such that the statement of the theorem holds. For simplicity, we assume that $B$ is smooth. Let $\eta\in B$ be the generic point with residue field $\bK:=\kappa(\eta)=\bk(B)$. By Lemma \ref{lem:non-closed-field-kc}, 
		there exists a plt blow-up $\mu_{\eta}: \cY_{\eta}\to \cX_{\eta}$ of $\eta\in (\cX_{\eta}, \cD_{\eta})$ with the Koll\'ar component $\cS_{\eta}$, such that 
	\[
	\hvol_{(\cX_{\eta}, \cD_{\eta}), \eta}(\cS_{\eta})\leq n^n+1.
	\]
	Let $f_{\eta}:\cZ_{\eta}\to \cY_{\eta}$ be a log resolution of $(\cY_{\eta},\mu_{\eta_*}^{-1}\cD_{\eta}+\cS_{\eta})$. We may extend $\mu_\eta: \cY_\eta\to \cX_\eta$ to a dense open subset $B'\subset B$ as a projective birational morphism $\mu':\cY'\to \cX'$ where $\cX':=\cX\times_B B'$, such that $\cY'\backslash \cS'\to \cX'\backslash B'$ is an isomorphism, the center of $\cS'$ on $\cX'$ is $B'$, and $\cS'$ is $\Qq$-Cartier.  
	Since $\cY_{\eta}$ is normal, by Lemma \ref{lem:normal in family is open}, possibly shrinking $B'$ to an open subset, we may assume that the fiber $\cY'_{b'}$ is normal for any closed point $b'\in B'$, and $\cY'$ is normal, and that $f_{{\eta}}$ can be extended to a morphism $f':\cZ'\to \cY'$ between families, such that $f'$ is a log resolution of $(\cY',{\mu'}_{*}^{-1}\cD'+\cS')$. 
	By \cite[Definition-Lemma 2.8]{Xu19}, possibly shrinking $B'$ and replacing $B'$ with a finite \'{e}tale covering, we may assume that $f'$ is a fiberwise log resolution of $(\cY',{\mu'}_{*}^{-1}\cD'+\cS')$.
	In particular, $(\cY',{\mu'}_{*}^{-1}\cD'+\cS')$ is plt near $\cS'$.
	Moreover, since both ampleness and flatness are open properties in a family, possibly shrinking $B'$ to an open subset again, we may further assume that $-\cS'$ is ample over $\cX'$, and $\cS'$ is flat over $B'$. Hence $\cY'\to \cX'$ provides a flat family of $\delta$-Koll\'ar components for some positive real number $\delta$. By \cite[Lemma 2.11]{LX16}, for any closed point $b'\in B'$, $\vol_{\cX_{b'},{b'}}(\ord_{S_{b'}})=\vol(\cS'_{b'},-\cS'_{b'}|_{\cS'_{b'}})$. We have $A_{(\cX'_{b'},\cD'_{b'})}(\cS'_{b'})=A_{(\cX_{\eta},\cD_{\eta})}(\cS_{\eta})$ is a constant function of closed points $b'\in B'$. Since $-\cS'|_{\cS'}$ is ample over $B'$, by the invariance of the Hilbert polynomial in the flat family $\cS'\to B'$ (cf. \cite[\S 3, Theorem 9.9]{GTM52}), 
	$$\vol_{\cX'_{b'},{b'}}(\ord_{\cS'_{b'}})=\vol(\cS'_{b'},(-\cS'|_{\cS'})_{b'})=\vol_{\cX_{\eta},{\eta}}(\ord_{\cS_{\eta}})$$ is a constant function for any closed point $b'\in B'$. Hence 
	$$\nvol_{(\cX'_{b'},\cD'_{b'}),b'}(\ord_{\cS'_{b'}})=	\hvol_{(\cX_{\eta}, \cD_{\eta}), \eta}(\cS_{\eta})\leq n^n+1$$
	for any closed point $b'\in B'$. 
\end{proof}

	\begin{lemma}\label{lem:non-closed-field-kc}
	Let $(X,\Delta)$ be an $n$-dimensional klt pair over a field $\bK$ of characteristic $0$. Let $x\in X$ be a $\bK$-rational point. Then $\hvol(x,X,\Delta)\leq n^n$. Moreover, for any $\epsilon>0$ there exists a Koll\'ar component $S$ over $x\in (X,\Delta)$ such that $\hvol_{(X,\Delta),x}(\ord_S)\leq \hvol(x,X,\Delta)+\epsilon$.
	\end{lemma}
	
	\begin{proof}
	 Let $(R,\fm):=(\cO_{X,x},\fm_{X,x})$. Let $\obK$ be the algebraic closure of $\bK$. Denote by $(x_{\obK}\in (X_{\obK},\Delta_{\obK})):=(x\in (X,\Delta)\times_{\bK}\obK$. By Theorem \ref{thm:uniqueness}, there exists a unique $\hvol$-minimizer $v_{\obK}\in \Val_{X_{\obK},x_{\obK}}$ up to rescaling. Hence $v_{\obK}$ is invariant under the action of $\mathrm{Gal}(\obK/\bK)$. In particular, there exists $v\in \Val_{X,x}$ such that $v_{\obK}$ is the natural extension of $v$, that is, $\fa_m(v_{\obK})=\fa_m(v)\otimes_{\bK}\obK$. It is clear that 
	 \[
	 \hvol(x,X,\Delta)\leq \hvol_{(X,\Delta),x}(v)=\hvol_{(X_{\obK},\Delta_{\obK}),x_{\obK}}(v_{\obK})=\hvol(x_{\obK},X_{\obK}, \Delta_{\obK}).
	 \]
	 On the other hand, for any $\fm$-primary ideal $\fa\subset R$ we have $\lct(X,\Delta;\fa)=\lct(X_{\obK},\Delta_{\obK};\fa_{\obK})$ and $\e(\fa)=\e(\fa_{\obK})$ where $\fa_{\obK}:=\fa\times_{\bK}\obK$. Thus we have $\hvol(x,X,\Delta)\geq \hvol(x_{\obK},X_{\obK}, \Delta_{\obK})$ by Lemma \ref{lem:nv=lcte}. Thus by Theorem \ref{thm:nvol is bounded by n^n} we have 
	 \[
	 \hvol(x,X,\Delta)=\hvol(x_{\obK},X_{\obK}, \Delta_{\obK})\leq n^n.
	 \]
	 For the second statement, we have $\lct(X,\Delta;\fa_m(v))^n\cdot \e(\fa_m(v))\leq \hvol(x,X,\Delta)+\epsilon$ for any $m\gg 1$. Then by \cite[Lemma 4.8]{Zhu20}, there exists a Koll\'ar component $S_m$ over $x\in (X,\Delta)$ computing $\lct(X,\Delta;\fa_m(v))$. Therefore, for $m\gg 1$ we have 
	 	\[
	\hvol_{(X,\Delta),x}(\ord_{S_m})\leq \lct(X,\Delta;\fa_m(v))^n\cdot \e(\fa_m(v))\leq \hvol(x,X,\Delta)+\epsilon,
	\]
	where the first inequality follows from \cite[Lemma 26]{Liu18}. 
	\end{proof}
	
	 \begin{lemma}[{\cite[IV Proposition 11.3.13, Theorem 12.2.4]{EGA}}]\label{lem:normal in family is open}
   	Let $f:X\to Y$ be a flat morphism between varieties. Then 
   	$$\{ y\in Y\mid X_y \text{ is geometrically normal over }\kappa(y) \}$$
   	is open in $Y$. Moreover, if $f$ is faithfully flat and all the fibers of $f$ are normal, then $X$ is normal. 
   \end{lemma}

	\begin{corollary}\label{cor:delta-plt-fields}
	Assume that $\bk$ is an algebraically closed subfield of $\bC$. Let $x\in (X,\Delta)$ be a klt singularity over $\bk$. Denote $(x_{\bC}\in (X_{\bC}, \Delta_{\bC})):=(x\in (X,\Delta))\times_{\bk}\bC$. If $(x_{\bC}\in (X_{\bC}, \Delta_{\bC}))$ admits a $\delta$-plt blow-up, then so does $x\in (X,\Delta)$.
	\end{corollary}
	
	\begin{proof}
	Let $\mu_{\bC}: Y_{\bC}\to X_{\bC}$ be the $\delta$-plt blow up of $x_{\bC}\in (X_{\bC}, \Delta_{\bC})$. We can find an intermediate subfield $\bk\subset \bK\subset \bC$ such that $\bK$ is a finitely generated field extension of $\bk$, and $\mu_{\bC}$ is defined over $\bK$ which we denote $\mu_{\bK}: Y_{\bK}\to X_{\bK}$. Let $B$ be a smooth variety over $\bk$ such that its function field $\bk(B)$ is isomorphic to $\bK$. Hence by similar arguments to the proof of Theorem \ref{thm: family kc with nv bdd},  after possibly shrinking $B$, there is a proper birational map $\mu: \cY\to X\times B$ such that  $\mu$ provides a flat family of Koll\'ar components $\cS$ over $(X\times B, \Delta\times B)\to B$ centered at $x\times B$, and restricting $\mu$ to the generic fiber over $B$ yields $\mu_{\bK}$. By assumption, we know that $(Y_{\bK}, \cS_{\bK}+ \Delta_{\bK})$ is $\delta$-plt. After further shrinking $B$ such that there exists a fiberwise log resolution of $(\cY, \cS+\mu_*^{-1}\Delta\times B)$,  we have that $(\cY_b, \cS_b+(\mu_b)_*^{-1}\Delta_b)$ is $\delta$-plt for a general closed point $b\in B$. Thus the proof is finished. 
	\end{proof}

	\section{Minimizing valuations for pairs with real coefficients}
	The purpose of this section is to generalize \cite[Main Theorem]{Blu18a} and \cite[Theorems 1.2 and 1.3]{Xu19} to the setting of any $\Rr$-Cartier $\Rr$-divisor $K_X+\Delta$. We remark that in \cite{Xu19}, one needs the existence of monotonic $n$-complements \cite[Theorem 1.8]{Bir19}, which only holds for $\Qq$-Cartier $\Qq$-divisors $K_X+\Delta$ in general (cf. \cite[Example 5.1]{HLS19}).
	
	\subsection{Existence and quasi-monomialness of a minimizing valuation}
	
	A folklore principle is that we may recover properties of the $\Rr$-Cartier $\Rr$-divisor $K_{X}+\Delta$ from corresponding properties of some $\bQ$-Cartier $\bQ$-divisors $K_X+\Delta'$ provided that those $\Delta'$'s are very close to the given $\Rr$-divisor $\Delta$ in the rational envelope of $\Delta$. 

    Here we will use Lemma \ref{lem:nv decomposable R-complements} to construct desired $\Qq$-divisors $\Delta'$'s. Lemma \ref{lem:nv decomposable R-complements} is a special case of \cite[Theorem 5.6]{HLS19} and \cite[Theorem 1.6]{Nak16} which could be regarded as a generalization of the conjecture on accumulation points of log canonical thresholds due to Koll\'ar \cite[Theorem 1.11]{HMX14}. We will use it frequently in the rest of this section. Recall that we say $V\subseteq\Rr^m$ is the \emph{rational envelope} of $\bm{a}\in\Rr^m$ if $V$ is the smallest affine subspace containing $\bm{a}$ which is defined over the rationals.
	\begin{lemma}[{\cite[Theorem 5.6]{HLS19}}]\label{lem:nv decomposable R-complements}
		Fix a positive integer $n$ and a point $\bm{a}=(a_1,\ldots,a_m)\in\Rr^m$. Then there exist positive real numbers $t_i$, and rational points $\bm{a}_i=(a_i^{1},\ldots,a_i^{m})\in\Qq^m$ in the rational envelope of $\bm{a}$ for $1\le i\le l$ depending only on $n$ and $\bm{a}$, such that $\sum_{i=1}^l t_i=1$, $\sum_{i=1}^l t_i\bm{a}_i=\bm{a}$, and the following holds.
		
		Let $x\in (X,\Delta\coloneqq \sum_{j=1}^m a_j\Delta_j)$ be a klt singularity of dimension $n$ and $S$ any Koll\'ar component of $x\in (X,\Delta)$, such that $\Delta_j\ge0$ is a Weil divisor for any $1\le j\le m$. Then $\Supp\Delta_{(i)}=\Supp\Delta$, $x\in (X,\Delta_{(i)})$ is klt, and $S$ is a Koll\'ar component of $x\in (X,\Delta_{(i)})$ for any $1\le i\le l$, where $\Delta_{(i)}:=\sum_{j=1}^m a_i^j\Delta_j$. 
	\end{lemma}
	
	The following lemma will be applied to generalize \cite[Main Theorem]{Blu18a} and \cite[Theorem 1.2]{Xu19}.

\begin{lem}\label{lem: peturb klt singularity with a seq of KC}
	Let $x\in(X,\Delta)$ be a klt singularity, and $\{S_j\}_{j=1}^{\infty}$ a sequence of Koll\'ar components of $x\in (X,\Delta)$ such that $\lim_{j\to +\infty}\nvol_{(X,\Delta),x}(\ord_{S_j})\le n^n$. Then possibly passing to a subsequence of $\{S_j\}_{j=1}^{\infty}$, there exist a positive real number $a\in[\frac{1}{2},1]$, and a $\Qq$-divisor $\Delta'$ on $X$, such that
	\begin{enumerate}
	\item $\Supp \Delta=\Supp \Delta'$ and $x\in (X,\Delta')$ is klt,
	\item $\{S_j\}_{j=1}^{\infty}$ is a sequence of Koll\'ar components of $x\in (X,\Delta')$,
	    \item $\lim_{j\to +\infty}\frac{A_{(X,\Delta')}(S_j)}{A_{(X,\Delta)}(S_j)}=a$, and
	    \item $\nvol_{(X,\Delta'),x}(\ord_{S_j})<n^n+1$ for any $j$.
	\end{enumerate}
\end{lem}

\begin{proof}
	Possibly passing to a subsequence, we may assume that $\nvol_{(X,\Delta),x}(\ord_{S_j})<n^n+1$ for any $j$.
	We may write $\Delta= \sum_{i=1}^m a_i\Delta_i$, where $\Delta_i$ are distinct prime divisors. There exist real numbers $r_1,\dots,r_c$, and $s_1,\dots,s_m$ $\Qq$-linear functions: $\Rr^{c+1}\rightarrow\Rr$, such that $1,r_1,\dots,r_c$ are linearly independent over $\Qq$, and $a_i=s_i(1,r_1,\dots,r_c)$ for any $1\le i\le m$.
	
	Let 
	\[
	    \Delta(x_1,\ldots,x_c)\coloneqq\sum_{i=1}^m s_i(1,x_1,\dots,x_c)\Delta_i.
	\]
	
	Let $n=\dim X$, and $t_1,\ldots,t_l,\bm{a}_1,\ldots,\bm{a}_l$ constructed in Lemma \ref{lem:nv decomposable R-complements} which only depends on $n$ and $\bm{a}=(a_1,\ldots,a_m)$. Note that
	\begin{itemize}
	    \item 	$\{\left(s_1(1,x_1,\ldots,x_c),\ldots,s_m(1,x_1,\ldots,x_c)\right)\mid x_1,\ldots,x_c\in\Rr\}$ is the rational envelope of 
	$\bm{a}$,
	\item $\bm{a}_1,\ldots,\bm{a}_l$ lie in the rational envelope of 
	$\bm{a}$, and
	\item $\bm{a}$ lies in the interior of the convex hull of $\bm{a}_1,\ldots,\bm{a}_l$.
	\end{itemize}
Thus there exists a positive real number $\delta$, such that $\Supp \Delta=\Supp\Delta(x_1,\ldots,x_c)$,  $x\in(X,\Delta(x_1,\ldots,x_c))$ is klt, and $\{S_j\}_{j=1}^{\infty}$ is a sequence of Koll\'ar components of $x\in(X,\Delta(x_1,\ldots,x_c))$ for any $x_i$ satisfying $|r_i-x_i|<\delta$.
		
	Let $D_i$ be $\Qq$-divisors such that
	$K_X+\Delta=K_X+D_0+\sum_{i=1}^c r_iD_i$.
	By \cite[Lemma 5.3]{HLS19}, $K_X+D_0$ and $D_i$ are $\Qq$-Cartier $\Qq$-divisors for any $1\le i\le c$. Since $1,r_1,\dots,r_c$ are linearly independent over $\Qq$, we may write	$$K_X+\Delta(x_1,\ldots,x_c)=K_X+D_0+\sum_{i=1}^c x_iD_i.$$
	Write $m_iD_i=\mathrm{div}(f_i)-\mathrm{div}(g_i)$, for some $m_i\in\bZ_{>0}$ and $f_i,g_i\in\mathcal{O}_{X,x}$ for any $1\le i\le c$. Denote $m_i\overline{D}_i:= \mathrm{div}(f_i)+\mathrm{div}(g_i)$.
	Possibly replacing $\delta$ with a smaller positive real number, we may assume that 
	$$C_2\sum_{i=1}^c |r_i-x_i|\ord_x(\overline{D}_i)\le  \frac{1}{2},$$
	for any $x_i$ which satisfies that $|r_i-x_i|<\delta$, where $C_2=C_2(x\in (X,\Delta))$ is the Izumi constant given by Lemma \ref{lem:izumi}.
	
	Since $$A_{(X,\Delta(x_1,\ldots,x_c))}(S_j)=A_{(X,\Delta)}(S_j)+\sum_{i=1}^c (r_i-x_i)\ord_{S_j}(D_i)$$
	for any $j$, possibly passing to a subseqence of $\{S_j\}_{j=1}^{\infty}$, there exist $r_1',\ldots,r_c'\in\Qq$ such that $|r_i-r_i'|\le \delta$ for any $i$, and $A_{(X,\Delta')}(S_j)\le A_{(X,\Delta)}(S_j)$ for any $j$, where $\Delta'\coloneqq \Delta(r_1',\ldots,r_c')$. Thus
	\begin{align*}
		1 &\ge \frac{A_{(X,\Delta')}(S_j)}{A_{(X,\Delta)}(S_j)}=\frac{A_{(X,\Delta)}(S_j)+\ord_{S_j}(\Delta-\Delta')}{A_{(X,\Delta)}(S_j)}
		\ge 1-\frac{\sum_{i=1}^c |(r_i-r_i')\cdot\ord_{S_j}(D_i)|}{A_{(X,\Delta)}(S_j)}\\
		&\ge 1- \frac{\sum_{i=1}^c |r_i-r_i'|\cdot\ord_{S_j}(\overline{D}_i)}{A_{(X,\Delta)}(S_j)}
		\ge 1-C_2\sum_{i=1}^c |r_i-r_i'|\cdot \ord_x(\overline{D}_i)\ge \frac{1}{2}.
	\end{align*}
	Hence possibly passing to a subsequence of $\{S_j\}_{j=1}^{\infty}$, we may assume that there exists a positive real number $a\in[\frac{1}{2},1]$, such that $\lim_{j\to +\infty}\frac{A_{(X,\Delta')}(S_j)}{A_{(X,\Delta)}(S_j)}=a$. Then 
	\begin{equation*}
		\lim_{j\to+\infty}\nvol_{(X,\Delta'),x}(\ord_{S_j})=	\lim_{j\to+\infty}\left(\frac{A_{(X,\Delta')}(S_j)}{A_{(X,\Delta)}(S_j)}\right)^n\nvol_{(X,\Delta),x}(\ord_{S_j})\le (an)^n\le n^n.
	\end{equation*}
	Therefore, possibly passing to a subsequence of $\{S_j\}_{j=1}^{\infty}$, we have $\nvol_{(X,\Delta'),x}(\ord_{S_j})<n^n+1$
	for any $j$. The proof is finished.
\end{proof}
    
    Next we prove the existence and quasi-monomialness of a minimizer of $\nvol_{(X,\Delta),x}$.
    
	\begin{theorem}[{cf. \cite[Main Theorem]{Blu18a}, \cite[Theorem 1.2]{Xu19}}]\label{thm:quasi-monomial}
		Let $x\in(X,\Delta)$ be a klt singularity. Then 
		\begin{enumerate}
		    \item there exists a minimizer of the function 
		    \[
		    \nvol_{(X,\Delta),x}:\Val_{X,x}\to \bR_{>0}\bigcup \{+\infty\};
		    \]
    		\item any minimizer $v_*$ of the function $\nvol_{(X,\Delta),x}$ is quasi-monomial.
		\end{enumerate}

	\end{theorem}
	\begin{proof}
	We may assume that $\dim X\ge 2$.
	
	(1) By Theorems \ref{thm:nvol is bounded by n^n} and \ref{thm:lx-kc-minimizing}, there exists a sequence of Koll\'ar components $\{S_j\}_{j=1}^{\infty}$ of $x\in (X,\Delta)$, such that
	\begin{equation}\label{eqn: thmqm limit approach}
	  \lim_{j\to+\infty} \nvol_{(X,\Delta),x}(\ord_{S_j})= \inf_{v\in\Val_{X,x}}\nvol_{(X,\Delta),x}(v)\le n^n.
	\end{equation}

		By Lemma \ref{lem: peturb klt singularity with a seq of KC}, possibly passing to a subsequence of $\{S_j\}_{j=1}^{\infty}$, there exist a positive real number $a\in[\frac{1}{2},1]$ and a $\Qq$-divisor $\Delta'$ on $X$ which satisfy Lemma \ref{lem: peturb klt singularity with a seq of KC}(1--5). Let $v_j'\coloneqq\frac{1}{A_{(X,\Delta')}(S_j)}\ord_{S_j}$ for any $j$. Since 
		\[
		    \nvol_{(X,\Delta'),x}(v_j')=\nvol_{(X,\Delta'),x}(\ord_{S_j})<n^n+1
		\]
		for any $j$, by \cite[Lemma 3.4]{Xu19} and \cite[Proposition 3.9]{LX16}, possibly passing to a subsequence of $\{S_j\}_{j=1}^{\infty}$, we may assume that $v_{*}':=\lim_{j\to+\infty}v_j'$  exists.
		
		We finish the proof following arguments of \cite[Remark 3.8]{Xu19}. Since $\lim_{j\to+\infty}v_j'=v_{*}'$, by \cite[Proposition 3.5]{Xu19}, there exist a positive integer $N$ and a family of Cartier divisors $D \subset X \times V$ parametrized by a variety $V$ of finite type, such that for any closed point $u \in V$, $x\in (X, \Delta'+\frac{1}{N}D_{u})$ is lc but not klt, and for any $j$, $S_{j}$ is an lc place of $x\in (X, \Delta'+\frac{1}{N} D_{u_{j}})$ for some closed point $u_{j} \in V$. Replacing $V$ by an irreducible closed subset, we can further assume that the set $\{u_{j}\mid j\in \Zz_{\ge 1}\}$ forms a dense set of closed points on $V$. We may further resolve $V$ to be smooth. By \cite[2.13]{Xu19}, possibly shrinking $V$, passing to a subsuquence of $\{S_j\}_{j=1}^{\infty}$, and replacing $V$ by a finite \'{e}tale covering, we can assume that $(X \times V, \Delta' \times V+\frac{1}{N} D) \rightarrow V$ admits a fiberwise log resolution 
		$\mu: Y \rightarrow (X \times V, \Delta' \times V+\frac{1}{N} D)$ over $V$. 

    Let $E$ be the simple normal crossing exceptional divisor of $\mu$ given by the components which are the lc places of $V\in (X \times V, \Delta' \times V+\frac{1}{N} D)$. By construction, there is a sequence of prime toroidal divisors $\{T_{j}\}_{j=1}^{\infty}$ over $(Y,E)$, such that $S_j$ is given by the restriction of $T_j$ over $u_j$. Fix a closed point $u\in V$. Let $F_j$ be the restriction of $T_j$ over $u$ for any $j$. Recall that $\Supp\Delta'=\Supp \Delta$, so $\mu$ is also a fiberwise log resolution of $(X \times V, \Delta \times V)$. Since $A_{(X \times V, \Delta' \times V+(\frac{1}{N}-\epsilon)D)}(T_j)<1$, and
    $V\in (X \times V, \Delta' \times V+(\frac{1}{N}-\epsilon)D)$ is klt for some positive real number $\epsilon\ll 1$, by \cite[Theorem 2.18]{Xu19}, we have
    \begin{equation}\label{eqn: thmqm SjFj}
        \nvol_{(X,\Delta),x}(\ord_{S_j})=\nvol_{(X,\Delta),x}(\ord_{F_j})=\nvol_{(X,\Delta),x}(w_j),
    \end{equation}
    where $w_j:=\frac{1}{A_{(X,\Delta)}(F_j)}(\ord_{ F_j})$.

    Since $F_j$ is a prime toroidal divisor over $(Y_u,E_u)$, where $E_u$ and $Y_u$ are the restrictions of $E$ and $Y$ over $u$ respectively, the limit of $w_j$ is a quasi-monomial valuation $w$, and $A_{(X,\Delta)}(w)=1$. By \cite[Corolleary D]{BFJ14}, the function $\vol_{X,x}(v)$ of $v$ is continuous on any given dual complex, which implies that 
    \begin{equation}\label{eqn: thmqm vol continuous}
        \lim_{j\to+\infty}\nvol_{(X,\Delta),x}(w_j)=\lim_{j\to+\infty}\vol_{X,x}(w_j)=\vol_{X,x}(w)=\nvol_{(X,\Delta),x}(w).
    \end{equation}

    Combining \eqref{eqn: thmqm limit approach}, \eqref{eqn: thmqm SjFj} and \eqref{eqn: thmqm vol continuous}, we conclude that 
    \begin{align*}
        &\inf_{v\in\Val_{X,x}}\nvol_{(X,\Delta),x}(v)=\lim_{j\to+\infty}\nvol_{(X,\Delta),x}(\ord_{S_j})\\
        =&\lim_{j\to+\infty}\nvol_{(X,\Delta),x}(\ord_{F_j})=\lim_{j\to+\infty}\nvol_{(X,\Delta),x}(w_j)=\nvol_{(X,\Delta),x}(w),
    \end{align*}
    and we are done.

\medskip


(2) From the proof of part (1), we know that there exists a quasi-monomial minimizer $w$ of the function $\hvol_{(X,\Delta),x}$. By Theorem \ref{thm:uniqueness}, any minimizer $v_*$ is a rescaling of $w$ hence is quasi-monomial.
	\end{proof}
	
	The uniqueness of $\hvol$-minimizers up to rescaling was proved in \cite{XZ20} for $\bQ$-divisors $\Delta$. Their proof can be easily generalized to $\bR$-divisors since the lengths and multiplicities of ideal sequences are independent of the boundary  $\Delta$, and the summation formula of multiplier ideals also works for $\bR$-divisors (see \cite{Tak06}). Thus we omit the proof here.
	
	\begin{thm}[cf. {\cite[Theorem 1.1]{XZ20}}]\label{thm:uniqueness}
	Let $x\in (X,\Delta)$ be a klt singularity. Then up to rescaling, there exists a unique minimizer $v_*$ of the functional $\hvol_{(X,\Delta),x}$.
	\end{thm}
	
\subsection{Constructibility of local volumes in families}
	\begin{theorem}[{cf. \cite[Theorem 1.3]{Xu19}}]\label{thm:constructibility R-div}
		Let $n$ be a positive integer. Let $B\subset(\cX,\cD)\to B$ be an $\Rr$-Gorenstein family of klt singularities of dimension $n$ over a normal base. The local volume function $\nvol(b,\cX_b,\cD_b)$ of closed points $b\in B$ is constructible in the Zariski topology.
	\end{theorem}
	
	\begin{proof}
	We may assume that $n\ge 2$. By Theorem \ref{thm:nvol is bounded by n^n}, for any closed point $b$, $\nvol(b,\cX_b,\cD_b)<C\coloneqq n^n+1.$ We may write $\cD=\sum_{j=1}^m a_j\cD_j$, where $\cD_j$ are distinct prime divisors. Apply Lemma \ref{lem:nv decomposable R-complements} to $n$ and $\bm{a}\coloneqq(a_1,\ldots,a_m)$, and let $t_1,\ldots,t_l$ be positive real numbers, and $\bm{a}_i=(a_i^1,\ldots,a_i^m)\in\Qq^m$ rational points given by it. Let $\cD^{(i)}=\sum_{j=1}^m a_{i}^j\cD_j$ for any $1\le i\le l$.  
		Then $\cD=\sum_{i=1}^l t_i\cD^{(i)}$, and $B\subset(\cX,\cD^{(i)})\to B$ is a $\bQ$-Gorenstein family of klt singularities of dimension $n$ for any $1\le i\le l$. Apply \cite[Proposition 4.2]{Xu19} to $C$ and  $B\subset(\cX,\cD^{(i)})\to B$, there is a finite type $B$-scheme $V^{(i)}$, a family of effective Cartier divisors $(\cG^{(i)}\subset \cX\times_{B} V^{(i)})\to V^{(i)}$, and a positive integer $N_i$ such that the following statement hold: for any Koll\'ar component $S_b$ over $b\in (\cX_b,\cD_b^{(i)})$ with $\hvol_{b, (\cX_b,\cD_b^{(i)})}(\ord_{S_b})\leq C$, there exists a closed point $u\in V^{(i)}\times_B \{b\}$ such that if we base change $(\cX_b,\cD_b^{(i)})$ and $S_b$ to $u$, then $S_u$ is an lc place of the log canonical pair $(\cX_u, \cD_b^{(i)}+\frac{1}{N_i}\cG_u^{(i)})$. Possibly stratifying the base $V^{(i)}$ into a disjoint union of finitely many constructible subsets and taking finite \'{e}tale coverings, we may assume that there exists a decomposition $V^{(i)}=\bigsqcup_{\alpha} V^{(i)}_{\alpha}$ into irreducible smooth strata $V^{(i)}_{\alpha}$ such that for each $\alpha$, $\left(\cX\times_{B}V^{(i)}_{\alpha},\Supp(\cD^{(i)}\times_{B}V^{(i)}_{\alpha}+\frac{1}{N_i}\cG^{(i)})\right)$ admits a fiberwise log resolution $\mu^{(i)}_{\alpha}: \cY^{(i)}_{\alpha}\to \cX\times_{B} V^{(i)}_{\alpha}$ over $V^{(i)}_{\alpha}$.
		
	Let $\cE^{(i)}_{\alpha}$ be the simple normal crossing exceptional divisor of $\mu^{(i)}_{\alpha}$ given by the components $\cF$, such that  $A_{(\cX\times_{B}V^{(i)}_{\alpha},\cD^{(i)}\times_{B}V^{(i)}_{\alpha}+\frac{1}{N_i}\cG^{(i)})}(\cF)=0$, and the center of $\cF$ on $\cX\times_{B}V^{(i)}_{\alpha}$ is the section $V^{(i)}_{\alpha}$. By Noetherian induction, possibly shrinking $B$, we may assume that each $V^{(i)}_{\alpha}\to B$ is surjective.
		
			Since $\cD=\sum_{i=1}^l t_i\cD^{(i)}$, for any closed point $b$ and any Koll\'ar component $S_b$ over $b\in (\cX_b,\cD_b)$ with $\nvol_{b, (\cX_b,\cD_b)}(\ord_{S_b})\le C$, there exists $i$, such that $\nvol_{b, (\cX_b,\cD^{(i)}_{b})}(\ord_{S_b})\le C$. \cite[Proposition 4.2]{Xu19} implies that there is a closed point $u\in \{b\}\times_{B}V^{(i)}_{\alpha}$ such that $(\cX_{u},\cD^{(i)}_{u}+\frac{1}{N_i}\cG^{(i)}_u)$ is lc and $S_u$ is an lc place of the pair, where $(\cX_{u},\cD^{(i)}_{u})$ and $S_u$ are the base change of $(\cX_b,\cD^{(i)}_{b})$ and $S_b$ over $u$. By the construction of $\cY^{(i)}_{\alpha}$, there is a prime toroidal divisor $\cT^{(i)}_{\alpha}$ over $(\cY^{(i)}_{\alpha},\cE^{(i)}_{\alpha})$ for some $\alpha$, such that $S_u$ is given by the restriction of $\cT^{(i)}_{\alpha}$ over $u$.

		For any prime toroidal divisor $\cT^{(i)}$ over $(\cY^{(i)}_{\alpha},\cE^{(i)}_{\alpha})$, there exists a positive real number $\epsilon\ll 1$, such that $V^{(i)}_{\alpha}\in (\cX\times_{B}V^{(i)}_{\alpha},\cD^{(i)}\times_{B}V^{(i)}_{\alpha}+(\frac{1}{N_i}-\epsilon)\cG^{(i)})$ is klt, and $$A_{(\cX\times_{B}V^{(i)}_{\alpha},\cD^{(i)}\times_{B}V^{(i)}_{\alpha}+(\frac{1}{N_i}-\epsilon)\cG^{(i)})}(\cT^{(i)})<1.$$
		Thus by \cite[Theorem 2.18]{Xu19}, $\vol_{\cX_u,u}(\ord_{\cT^{(i)}_{u}})$ is a constant function for closed points $u\in V^{(i)}_{\alpha}$.
		Moreover, $\mu^{(i)}_{\alpha}: \cY^{(i)}_{\alpha}\to \cX\times_{B} V^{(i)}_{\alpha}$ is a fiberwise log resolution of $(\cX\times_{B} V^{(i)}_{\alpha},\Supp (\cD\times_{B}V^{(i)}_{\alpha}))$ as
		$\Supp \cD^{(i)}=\Supp \cD$. It follows that $A_{(\cX_u,\cD_u)}(\ord_{\cT_{u}^{(i)}})$ is also a constant function for closed points $u\in V^{(i)}_{\alpha}$. We conclude that
		$\nvol_{(\cX_u,\cD_u),u}(\ord_{\cT^{(i)}_{u}})$ is a constant function for closed points $u\in V^{(i)}_{\alpha}$.
		Hence for each $\alpha$ and $i$, there exists a positive real number $\nu^{(i)}_{\alpha}$, such that
		$$\nu^{(i)}_{\alpha}=\inf_{\cT^{(i)}}\{\nvol_{(\cX_u,\cD_{u}),u}(\ord_{\cT^{(i)}_{u}})\mid \cT^{(i)}\text{ is a prime toroidal divisor over } (\cY^{(i)}_{\alpha},\cE^{(i)}_{\alpha})\},$$
		for any closed point $u\in V^{(i)}_{\alpha}$.
		
		Recall that each $V^{(i)}_{\alpha}\to B$ is surjective. By Theorem \ref{thm:lx-kc-minimizing}, for any closed point $b\in B$, we have
		\begin{align*}
		\nvol(b,\cX_b,\cD_b) & =\inf_{S_b}\{
		\nvol_{(\cX_b,\cD_b),b}(\ord_{S_b})\le n^n+1 \mid S_b\text{ is a Koll\'ar component of }  b\in (\cX_b,\cD_b)\}\\
		&\ge  \min_{i,\alpha}\{\nu^{(i)}_{\alpha}\}\ge \nvol(b,\cX_b,\cD_b).
		\end{align*}
		Hence $\nvol(b,\cX_b,\cD_b)=\min_{i,\alpha}\{\nu^{(i)}_{\alpha}\}$ for any closed point $b\in B$, which implies that $\nvol(b,\cX_b,\cD_b)$ is a constructible function of $b\in B$ in the Zariski topology.
	\end{proof}

	\section{Log canonical thresholds and local volumes}
	
	In this section, we investigate the relation between $\lct(X,\Delta;\Delta)$ and $\hvol(x,X,\Delta)$ for a klt singularity $x\in (X,\Delta)$ where $\Delta$ is $\Rr$-Cartier.
	The main goal of this section is to prove the following theorem.
	
	\begin{theorem}\label{thm:estimate on lct}
	Let $n$ be a positive integer, and $B\subset\cX\to B$ a $\bQ$-Gorenstein family of $n$-dimensional klt singularities. Then there exists a positive real number $c$ depends only on $n$ and $B\subset\cX\to B$ satisfying the following. 
	
	Let $x\in (X,\Delta)$ be an $n$-dimensional klt singularity such that $(x\in X^{\rm an})\in (B\subset\cX^{\rm an}\to B)$. 
	Then
		\begin{equation*}
		c\cdot\lct(X,\Delta;\Delta)\geq  \nvol(x,X,\Delta). 
		\end{equation*}
	\end{theorem}
	We remark that Jiang studied lower bound of log canonical thresholds $\lct(X,\Delta;\Delta)$ in the setting of Fano fibrations \cite[Conjecture 1.13, Theorem 5.1]{Jia18}, see also \cite[Theorem 3.4, Conjecture 3.6]{CDHJS18}.
	
	\medskip
	
	The following proposition is crucial in the proof of Theorem \ref{thm:estimate on lct}.
	
	\begin{proposition}\label{lem:estimate on lct Cartier} Let $n\ge2$ be a positive integer, and $x\in(X,\Delta)$ an $n$-dimensional klt $\Qq$-Gorenstein singularity. Let $\fm\subset \cO_{X,x}$ be the maximal ideal. Let $C_2=C_2(x,X)$ be the Izumi constant of the klt singularity $x\in X$ (see Lemma \ref{lem:izumi}). 
	Then for any effective Cartier divisor $D$ passing through $x$, we have
		\[
		c\cdot\lct(X,\Delta;D)\geq \nvol(x,X,\Delta),
		\]
		where $c=\frac{n^{2n+1}}{(n-1)^{n-1}}\e(\fm)C_2\ord_x D$.
	\end{proposition}
	
	\begin{proof}
		Possibly shrinking $X$ near $x$, we may assume that $X=\Spec(R)$ and $D=\mathrm{div}(f)$ where $f\in\fm$. Let $c_0\coloneqq \lct(X,\Delta;D)$. Let $v\in \Val_{X,x}$ be the minimizing valuation of $\hvol_{X,x}$.
		Consider the $\fm$-primary ideal $\fa_{s,t}\coloneqq (f^s)+\fa_t(v)$, where $s,t\in\bZ_{>0}$. By the sub-additivity of log canonical thresholds (Proposition \ref{prop:sub-additivity for lct}), 
		we have
		\begin{equation}
		\lct(X,\Delta;\fa_{s,t})\leq \lct(X,\Delta;(f^s))+
		\lct(X,\Delta;\fa_t(v))\leq \frac{c_0}{s} + \frac{A_{(X,\Delta)}(v)}{v(\fa_t(v))}.
		\end{equation} 
		Moreover, we know that 
		\[
		\ell(R/\fa_{s,t})=\ell(R/\fa_t(v))-\ell((f^s)/(f^s)\cap\fa_t(v))
		=\ell(R/\fa_t(v))-\ell(R/(\fa_t(v):(f^s))).
		\]
		Since $v$ is a valuation, we have $(\fa_t(v):(f^s))=\fa_{t-v(f)s}(v)$ for any $t\ge v(f)s$.
		Hence
		\[
		\ell(R/\fa_{s,t})=\ell(R/\fa_t(v))-\ell(R/\fa_{t-v(f)s}(v)).
		\]
		Let $s:= \lfloor\frac{(n-1)c_0}{A_{(X,\Delta)}(v)}\cdot t\rfloor$ for $t\gg 1$.
		Then as $t\to \infty$ we have
		\begin{equation*}
		\begin{aligned}
		n!\cdot \ell(R/\fa_{s,t}) & = n!\cdot \ell(R/\fa_t(v))-n!\cdot \ell(R/\fa_{t-v(f)s}(v))\\
		& =\vol_{X,x}(v)\cdot \left( t^n - (\max\{ t-v(f)s,0\})^n\right)+O(t^{n-1})\\
		& \leq \vol_{X,x}(v)\cdot n v(f)s t^{n-1} + O(t^{n-1}).
		\end{aligned}
		\end{equation*}
		Thus by Lemma \ref{lem:nv=lcte}, we have
		\begin{align*}
		\nvol(x,X,\Delta)
		&\leq \liminf_{t\to+\infty}\e(\fa_{s,t})\cdot\lct(X,\Delta;\fa_{s,t})^n\\
		& \leq \liminf_{t\to+\infty}\e(\fm)\cdot  n!\ell(R/\fa_{s,t})\cdot\lct(X,\Delta;\fa_{s,t})^n\\
		& \leq \liminf_{t\to+\infty} \e(\fm)\cdot \vol_{X,x}(v) nv(f) st^{n-1}\cdot\left(\frac{c_0}{s}+\frac{A_{(X,\Delta)}(v)}{v(\fa_t(v))}\right)^n\\
		& =\e(\fm)\cdot \vol_{X,x}(v) nv(f)\lim_{t\to+\infty} st^{n-1}\cdot\left(\frac{c_0}{s}+\frac{A_{(X,\Delta)}(v)}{t}\right)^n\\
		& = \e(\fm)\cdot \vol_{X,x}(v) nv(f)\cdot \frac{n^n}{(n-1)^{n-1}} A_{(X,\Delta)}(v)^{n-1}\cdot c_0,
		\end{align*}
		where the second line follows from Lech's inequality \cite[Theorem 3]{Lec60}, and the fourth line follows from 
	    \cite[Lemma 3.5]{Blu18a}.
		
		By Izumi's inequality (Lemma \ref{lem:izumi}), there exists a positive real number $C_2$ independent on $f$ such that $v(f)\leq C_2 A_X(v)\ord_x(f)$. Hence 
		\[
		\hvol(x,X,\Delta)\leq \frac{n^{n+1}}{(n-1)^{n-1}}\e(\fm) C_2 \ord_x(f)\cdot \hvol_{X,x}(v)\cdot c_0\leq \frac{n^{2n+1}}{(n-1)^{n-1}}\e(\fm) C_2 \ord_x(f)\cdot c_0.
		\]
		Here the second inequality follows from $\hvol_{X,x}(v)=\hvol(x,X)\leq n^n$ by Theorem \ref{thm:nvol is bounded by n^n}. 
	\end{proof}
	
	We also need the following kind of approximation of $\bR$-divisors by $\bQ$-divisors.
	
	\begin{lemma}\label{lem: perturb coeff of R-div}
		Let $\epsilon$ be a positive real number, and $\Delta\ge0$ an $\Rr$-Cartier $\Rr$-divisor on a normal variety $X$. Then there exists a $\Qq$-Cartier $\Qq$-divisor $\Delta'\geq 0$, such that $(1+\epsilon)\Delta\ge \Delta'\ge (1-\epsilon)\Delta$.
	\end{lemma}
	\begin{proof}
		There exist positive real numbers $r_1,\ldots,r_c$, $\Qq$-linear functions $s_1,\dots,s_m$: $\Rr^{c+1}\rightarrow\Rr$, and distinct prime divisors $\Delta_i$, such that $1,r_1,\dots,r_c$ are linearly independent over $\Qq$, and $\Delta=\Delta(r_1,\ldots,r_c)$, where
		$$\Delta(x_1,\ldots,x_c)=\sum_{i=1}^m s_i(1,x_1,\ldots,x_c)\Delta_i.$$
		By \cite[Lemma 5.3]{HLS19}, $\Delta(x_1,\ldots,x_c)$ is $\Rr$-Cartier for any $x_1,\ldots,x_c\in\Rr$. It follows that there exist positive rational numbers $r_1',\ldots,r_c'$, such that $(1+\epsilon)\Delta\ge \Delta'\ge (1-\epsilon)\Delta$, where
		$\Delta'=\Delta(r_1',\ldots,r_c')$.
	\end{proof}

	\begin{proof}[Proof of Theorem \ref{thm:estimate on lct}] 
	    If $n=1$, then we may take $c=1$. Thus we may assume that $n\ge 2$. Fix any $\epsilon\in(0,1)$. By Lemma \ref{lem: perturb coeff of R-div}, there exists an effective $\Qq$-Cartier $\bQ$-divisor $\Delta'$, such that $(X,\Delta')$ is klt, and $\Delta'\ge (1-\epsilon)\Delta$. Let $N$ be a positive integer such that $N\Delta'$ is Cartier near $x$. By Proposition \ref{prop: order of a klt boundary}, $\ord_x(N\Delta')< Nn$. By Proposition \ref{lem:estimate on lct Cartier}, 
		\begin{align*}
		\hvol(x,X,\Delta) & \leq  \frac{n^{2n+1}}{(n-1)^{n-1}}\e(\fm) C_2 \ord_x(N\Delta')\cdot \lct(X,\Delta;N\Delta')\\
		& \leq  \frac{n^{2n+2}}{(n-1)^{n-1}}\e(\fm) C_2 \cdot\lct(X,\Delta;\Delta')\leq \frac{n^{2n+2}\e(\fm) C_2}{(1-\epsilon)(n-1)^{n-1}} \cdot\lct(X,\Delta;\Delta).
		\end{align*}
		Here we choose $C_2$ which depends on $B\subset \cX\to B$ as in Lemma \ref{lem:family Izumi}(2).  
	
		By the upper semi-continuity of 	Hilbert-Samuel function along a family of ideals (see for example \cite[Proposition 41]{BL18a}) and the fact that the completion preserves the multiplicity $\e(\fm)$, there exists a positive integer $M$ which only depends on $B\subset \cX\to B$ such that $\e(\fm)\leq M$. Let $\epsilon\to 0$, we see that the theorem holds with $c=\frac{n^{2n+2}}{(n-1)^{n-1}} M C_2$.
	\end{proof}

	\section{Lipschitz continuity of local volumes}
	We will prove some Lipschitz-type estimates for the normalized volume as a function of the coefficients in this section. 
    The main result is the following uniform Lipschitz-type estimate when the ambient space $x\in X$ analytically belongs to a $\Qq$-Gorenstein bounded family.
	
	\begin{theorem}\label{thm:uniform lipschitz}
		Let $n$ be a positive integer, and $\eta,\gamma$ positive real numbers. Let $B\subset \cX\to B$ be a $\bQ$-Gorenstein family of $n$-dimensional klt singularities. Then there exist positive real numbers $\iota,C$ depending only on $n,\eta,\gamma$ and the family $B\subset \cX\to B$, such that the following holds. 
		
		Let $x\in (X,\Delta=\sum_{i=1}^m a_i\Delta_i)$ be a klt singularity, such that
		\begin{enumerate}
			\item $(x\in X^{\rm an})\in (B\subset \cX^{\rm an}\to B)$,
			\item $a_i>\eta$ for any $i$,
			\item each $\Delta_i\ge 0$ is a $\bQ$-Cartier Weil divisor, and
			\item $\lct(X,\Delta;\Delta)>\gamma$.
		\end{enumerate}
		Then for any $-a_i\leq t_i\leq \iota$, $i=1,2,\ldots,m$,
		\begin{equation*}
		|\nvol(x,X,\Delta)-\nvol(x,X,\Delta(\bm{t}))|\le C\sum_{i=1}^m|t_i|,
		\end{equation*}
		where $\bm{t}\coloneqq(t_1,\ldots,t_m)$, and $\Delta(\bm{t})\coloneqq \sum_{i=1}^m (a_i+t_i)\Delta_i$.
	\end{theorem}

	\begin{lemma}\label{lem:uniform lipschitz}
		Let $n,\eta,\gamma,~\bxb, ~ x\in (X,\Delta)$ be as in Theorem \ref{thm:uniform lipschitz}. Let $V$ be a positive real number. Then there exists a positive real number $C$ depending only on $n,\eta,\gamma,V$ and the family $B\subset \cX\to B$ satisfying the following. 
		
		Let $v\in \Val_{X,x}$ be a valuation such that $\nvol_{(X,\Delta),x}(v)<V$.
	Then for any $-a_i\leq t_i\leq 0$, $i=1,2,\ldots,m$,
		\begin{equation*}
		0\leq\nvol_{(X,\Delta(\bm{t})),x}(v)-\nvol_{(X,\Delta),x}(v)\leq C\sum_{i=1}^m|t_i|,
		\end{equation*}
		where $\bm{t}\coloneqq(t_1,\ldots,t_m)$, and $\Delta(\bm{t})\coloneqq \sum_{i=1}^m (a_i+t_i)\Delta_i$.
	\end{lemma}

	\begin{proof}
		
		By \cite[18.22]{Fli92}, $m\le \frac{n}{\eta}$. By Lemma \ref{lem:kill boundary}, we have $A_{(X,\Delta)}(v)\geq\left(\frac{\gamma}{1+\gamma}\right)A_X(v)$.
		Let $C_2$ be the positive real number given by Lemma \ref{lem:family Izumi}, which depends only on $n$ and the family $B\subset \cX\to B$. By Proposition \ref{prop: order of a klt boundary}, we have
		\begin{equation}\label{eqn:Izumi}
		v(\Delta_i)\le C_2A_X(v)\ord_x\Delta_i\le \frac{nC_2}{\eta}A_X(v)\leq \frac{nC_2(1+\gamma)}{\eta\gamma}A_{(X,\Delta)}(v).
		\end{equation}
		for any $1\le i\le m$.
		By \eqref{eqn:Izumi}, we get
		\begin{align*}
		0&\leq \nvol_{(X,\Delta(\bm{t})),x}(v)-\nvol_{(X,\Delta),x}(v)		=\left(\left(\frac{A_{(X,\Delta(\bm{t}))}(v)}{A_{(X,\Delta)}(v)}\right)^n-1\right)\nvol_{(X,\Delta),x}(v)\\
		&<\left(\left(1+\frac{\sum_{i=1}^m|t_{i}|v(\Delta_{i})}{A_{(X,\Delta)}(v)}\right)^n-1\right)V		\leq \left(\left(1+\sum_{i=1}^m|t_{i}|\cdot \frac{nC_2(1+\gamma)}{\eta\gamma}\right)^n-1\right)V\\
		&\leq \sum_{i=1}^m|t_{i}|\cdot \frac{nC_2(1+\gamma)}{\eta\gamma}\cdot n\left(1+\frac{n^2C_2(1+\gamma)}{\eta^2\gamma}\right)^{n-1}V,
		\end{align*}
		where the last inequality follows from the inequalities $(1+xy)^n-1\le nxy(1+xy)^{n-1}\le nxy(1+\frac{n}{\eta}y)^{n-1}$ for any $\frac{n}{\eta}\ge x\ge 0$, $y\geq 0$, and $\sum_{i=1}^m |t_i|\le \sum_{i=1}^m a_i\le m\le \frac{n}{\eta}$. Now $C\coloneqq \frac{n^2C_2(1+\gamma)}{\eta\gamma}(1+\frac{n^2C_2(1+\gamma)}{\eta^2\gamma})^{n-1}V$ depends only on $n$, $\eta$, $\gamma$, $V$ and the family $B\subset \cX\to B$, hence we are done.
	\end{proof}
	

	\begin{proof}[Proof of Theorem \ref{thm:uniform lipschitz}]
		Possibly replacing $\gamma$ by $\min\{\gamma,1\}$, we may assume that $0<\gamma\le 1$. Since $\lct(X,\Delta;\Delta)> \gamma$, we have that $x\in (X,(1+\gamma)\Delta)$ is klt. This implies that 
		\[
		\lct(X,(1+\tfrac{\gamma}{2})\Delta;(1+\tfrac{\gamma}{2})\Delta)> \tfrac{\gamma}{3}
		\]
		because $(1+\frac{\gamma}{2})+\frac{\gamma}{3}(1+\frac{\gamma}{2})\le 1+\gamma$.
		Since $a_i+\frac{\gamma\eta}{2}\le (1+\frac{\gamma}{2})a_i$, we have $\Delta(\bm{\iota})\le (1+\frac{\gamma}{2})\Delta$ where $\iota\coloneqq \frac{\gamma\eta}{2}$ and $\bm{\iota}=(\iota,\ldots,\iota)$. 
		Let $t_{i}^{+}\coloneqq\max\{t_i,0\}$, $t_{i}^{-}\coloneqq\min\{t_i,0\}$ for any $1\le i\le m$, and 
		$\bm{t}^{+}\coloneqq(t_{1}^{+},\ldots,t_{m}^{+})$, $\bm{t}^{-}\coloneqq(t_{1}^{-},\ldots,t_{m}^{-})$. Let $v^{+}$ be a minimizer of $\nvol(x,X,\Delta(\bm{t}^{+}))$. Since $\Delta(\bm{t}^+)\leq \Delta(\bm{\iota})\leq (1+\frac{\gamma}{2})\Delta$, we have $\lct(X,\Delta(\bm{t}^+);\Delta(\bm{t}^+))>\frac{\gamma}{3} $.  By Lemma \ref{lem:uniform lipschitz},
		\begin{align*}
		&|\nvol(x,X,\Delta)-\nvol(x,X,\Delta(\bm{t}))|
		 \\	\leq ~&  |\nvol(x,X,\Delta)-\nvol(x,X,\Delta(\bm{t}^{+}))|+	|\nvol(x,X,\Delta(\bm{t}^{+}))-\nvol(x,X,\Delta(\bm{t}))|\\
		\le~ &(\nvol_{(X,\Delta),x}(v^{+})-\nvol_{(X,\Delta(\bm{t}^{+})),x}(v^{+}))+(\nvol_{(X,\Delta(\bm{t})),x}(v^{+})-\nvol_{(X,\Delta(\bm{t}^{+})),x}(v^{+}))\leq C\sum_{i=1}^m|t_i|,
		\end{align*}
		where $C$ is the positive real number given in Lemma \ref{lem:uniform lipschitz} which only depends on  $n,\eta,\frac{\gamma}{3},n^n$, and the family $B\subset \cX\to B$.
	\end{proof}

	The next result is a Lipschitz-type inequality for $\nvol(x,X,\Delta)$, when $x\in X$ is fixed and the boundary $\Delta$ varies in its rational envelope. Lemma \ref{lem: nv lipschitz in rational envelope} will be applied to prove Theorem \ref{thm:bddness for nvol}. We remark that we do not assume that $x\in X$ is $\Qq$-Gorenstein. 

	\begin{lemma}\label{lem: nv lipschitz in rational envelope}
		Let $x\in (X,\Delta\coloneqq \sum_{i=1}^m a_i\Delta_i)$ be a klt singularity of dimension $n$, where $\Delta_i$ are distinct prime divisors. Let $V\subseteq\Rr^m$ be the rational envelope of $\bm{a}=(a_1,\ldots,a_m)\in\Rr^m$. Then there exist a positive real number $C$ and a neighborhood $U\subseteq V$ of $\bm{a}$, such that $x\in (X,\Delta(\bm{a}'))$ is a klt singularity, and 
		$$|\nvol(x,X,\Delta)-\nvol(x,X,\Delta(\bm{a}'))|\le C\sum_{i=1}^m|a_i-a_i'|$$
		for any $\bm{a}'\coloneqq (a_1',\ldots,a_m')\in U$, where $\Delta(\bm{a}')\coloneqq\sum_{i=1}^m a_i'\Delta_i$. In particular, $\nvol(x,X,\Delta(\bm{a}'))$ is continuous at $\bm{a}'$ in $V$.
	\end{lemma}
	\begin{proof}
		There exist real numbers $r_1,\dots,r_c$, and  $\Qq$-linear functions $s_1,\dots,s_m$: $\Rr^{c+1}\rightarrow\Rr$, such that $1,r_1,\dots,r_c$ are linearly independent over $\Qq$, and $a_i=s_i(1,r_1,\dots,r_c)$ for any $1\le i\le m$.
				Let $D_i$ be $\Qq$-divisors such that
		$K_X+\Delta=K_X+D_0+\sum_{i=1}^c r_iD_i$.
		By \cite[Lemma 5.3]{HLS19}, $K_X+D_0$ and $D_i$ are $\Qq$-Cartier $\Qq$-divisors for any $1\le i\le c$. 
		
			There exists a positive integer $N$, such that $ND_i$ is Cartier for any $1\le i\le c$. For any $1\le i\le c$, possibly replacing $r_i$ with $\frac{r_i}{N}$ and $D_i$ with $ND_i$, we may assume that $D_i$ is Cartier. Write $D_i=\mathrm{div} (f_i)-\mathrm{div} (g_i)$, where $f_i,g_i\in \cO_{X,x}$ for any $1\le i\le c$. Let $\Delta'(\bm{t})\coloneqq D_0+\sum_{i=1}^c(r_i+t_i)D_i$, where $\bm{t}=(t_1,\ldots,t_c)\in\Rr^c$. There exists a positive real number $\iota\le 1$, such that $x\in (X,\Delta'(\bm{t}))$ is a klt singularity for any $\bm{t}=(t_1,\ldots,t_c)\in\Rr^c$ satisfying $\sum_{i=1}^c |t_i|\le \iota$.
		
		It suffices to show that there exist positive real numbers $C'$ and $\iota'\le \iota$, such that  
		$$|\nvol(x,X,\Delta)-\nvol(x,X,\Delta'(\bm{t}))|\le C'\sum_{i=1}^c |t_i|,$$
		for any $\bm{t}=(t_1,\ldots,t_c)\in\Rr^c$ which satisfies that $\sum_{i=1}^c |t_i|\le \iota'$,


		Let $C_2>0$ be the Izumi constant of the singularity $x\in (X,\Delta)$ given by Lemma \ref{lem:izumi}, and $M$ a positive real number such that $C_2\max\{\ord_{x}(f_i),\ord_{x}(g_i)\}\le M$ for any $1\le i\le c$. Then we have
		\begin{equation*}
		|v(D_i)|\le C_2A_{X,\Delta}(v)\max\{\ord_x(f_i),\ord_{x}(g_i)\}\le MA_{X,\Delta}(v).
		\end{equation*}
		for any $1\le i\le c$ and any $v\in \Val_{X,x}$.

		Let $v$ be a minimizer of $\nvol(x,X,\Delta)$. For any $\bm{t}=(t_1,\ldots,t_c)\in\Rr^c$ which satisfies that $\sum_{i=1}^c|t_i|\le\iota$, we have
		\begin{align*}
		&\nvol(x,X,\Delta'(\bm{t}))-\nvol(x,X,\Delta)
		\le\nvol_{(X,\Delta'(\bm{t})),x}(v)-\nvol_{(X,\Delta),x}(v)\\
		=&\left((\frac{A_{(X,\Delta'(\bm{t}))}(v)}{A_{(X,\Delta)}(v)})^n-1\right)\nvol_{(X,\Delta),x}(v)\leq\left((\frac{\sum_{i=1}^c|t_{i}|\cdot|v(D_{i})|}{A_{(X,\Delta)}(v)}+1)^n-1\right)n^n\\
		\leq&\left((M\sum_{i=1}^c|t_{i}|+1)^n-1\right)n^n\leq M(1+M)^{n-1}n^{n+1}\sum_{i=1}^c|t_i|,
		\end{align*}
		where the last inequality follows from the inequality $(xy+1)^n-1\le nxy(1+xy)^{n-1}\le nxy(1+y)^{n-1}$ for any $1\ge x\ge 0$ and any $y\ge 0$. 
		
		For any $\bm{t}=(t_1,\ldots,t_c)\in\Rr^c$ which satisfies that $\sum_{i=1}^c|t_i|\le \min\{\frac{1}{2nM},\iota\}$, let $v_{*}$ be a minimizer of $\nvol(x,X,\Delta'(\bm{t}))$. We have
		\begin{align*}
		&\nvol(x,X,\Delta'(\bm{t}))-\nvol(x,X,\Delta)\geq\nvol_{(X,\Delta'(\bm{t})),x}(v_{*})-\nvol_{(X,\Delta),x}(v_{*})\\		=&\left(1-(\frac{A_{(X,\Delta)}(v_{*})}{A_{(X,\Delta'(\bm{t})}(v_{*})})^n\right)\nvol_{(X,\Delta'(\bm{t})),x}(v_{*})\ge\left(1-(\frac{A_{(X,\Delta)}(v_{*})}{A_{(X,\Delta)}(v_{*})-\sum_{i=1}^c |t_i|\cdot|v_{*}(D_i)|})^n\right)n^n\\
		\ge &\left(1-(\frac{1}{1-M\sum_{i=1}^c |t_i|})^n\right)n^n
		\ge -2Mn^{n+1}\sum_{i=1}^c |t_i|,
		\end{align*}
		where the last inequality follows from inequalities $\frac{1}{(1-t)^n}\le \frac{1}{1-nt}$ and $(1-nt)(1+2nt)\ge1$ for any $0\le t\le \frac{1}{2n}$.
		Thus 	$\iota'\coloneqq\min\{\iota,\frac{1}{2nM}\}$ and $C'\coloneqq 2M(1+M)^{n-1}n^{n+1}$ have the required property. 
	\end{proof}
	

	\section{Local volumes of truncated singularities}
	

\subsection{Truncations preserve local volumes}
	In this section, we show that the local volume stays the same after taking a $k$-th truncation of the boundary divisor when $k$ is sufficiently large.
	In the general context of this paper, we often consider analytically bounded families. 
	Thus we make the following definition which we use throughout this section. 
	
	\begin{defn}\label{def:analytic-truncation} 
	Let $(x\in X)$ and $(x'\in X')$  be klt singularities which are analytically isomorphic to each other. Denote $(R,\fm):=(\cO_{X,x},\fm_{X,x})$, and $(R',\fm'):= (\cO_{X',x'},\fm_{X',x'})$. Let $\psi: \widehat{R}\xrightarrow{\cong} \widehat{R'}$ be the ring isomorphism. Let $k$ be a positive integer. Fix a $\bk$-linear basis $\bar{g}'_1,\cdots,\bar{g}'_d$ of $R'/\fm'^k$. Let $g_j'\in R'$ be a lifting of $\bar{g}_j'$. For an effective Cartier divisor $D=\mathrm{div}(f)$ on $X$, we define its \emph{$k$-th analytic truncation} $D'^k:=\mathrm{div}(f'_k)$ on $X'$ where $f'_k$ is the $\bk$-linear combination of $g_j'$ such that $\psi(f)-f_k'\in \widehat{\fm'}^k$.  If  $\Delta=\sum_i a_i \Delta_i$ is a non-negative $\bR$-linear combination of effective Cartier divisors $\Delta_i$, then we say that $x'\in (X', \Delta'^k:=\sum_i a_i \Delta'^k_i)$ is a \emph{$k$-th analytic truncation} of $x\in (X,\Delta)$.
	\end{defn}
	
	Note that in the above definition, a $k$-th analytic truncation depends on the choice of many data, such as the basis $\bar{g}_j'$, its lifting $g_j'$, and the expression $\Delta=\sum_i a_i \Delta_i$. Thus analytic truncations are highly non-unique. In this section, we aim to show that if $k\gg 1$ then any $k$-th analytic truncation of a given klt singularity has the same local volume and admits a $\delta$-plt blow-up for the same $\delta$.
	
	The main result of this section is Theorem \ref{thm:truncation} which will be applied to prove Theorem \ref{thm:discreteACC}. We will also need Proposition \ref{prop: truncation keep nv(v)} to prove Theorem \ref{thm: lct positive lower bound implies existence of delta-plt blow-up}, and thus Theorem \ref{thm:existence of delta-plt blow-up}.

		\begin{theorem}\label{thm:truncation}
		Let $n$ be a positive integer, $\eta,\gamma$ positive real numbers, and $B\subset\cX\to B$ a $\bQ$-Gorenstein family of $n$-dimensional klt singularities. Then there exists a positive integer $k_2$ depending only on $n,\eta,\gamma$ and $B\subset\cX\to B$ satisfying the following. 
		
		Let $x\in(X,\Delta=\sum_{i=1}^m a_i\Delta_i)$ be a klt singularity, such that	
		\begin{enumerate}
			\item $(x\in X^{\rm an})\in (\bxanb)$,
				\item $a_i\ge \eta$ for any $i$,
			\item each $\Delta_i\ge0$ is a Cartier divisor,
		 and 
			\item $\lct(X,\Delta;\Delta)>\gamma$.
		\end{enumerate}
		Then for any positive integer $k\ge k_2$, and any $k$-th analytic truncation $x'\in (X',\Delta'^k:=\sum_{i=1}^m a_i\Delta_i'^k)$ of $x\in (X,\Delta)$,
		\[
		\nvol(x,X,\Delta)=\nvol(x',X',\Delta'^k).
		\]
	
		Moreover, $v$ is a minimizer of $\nvol(x,X,\Delta)$ if and only if $v'=\phi(v)$ is a minimizer of $\nvol(x',X',\Delta'^k)$, where $\phi: \Val_{X,x}^\circ\to \Val_{X',x'}^\circ$ is defined as in Proposition \ref{prop:an-iso}.
	\end{theorem}
	
	We need some preparation to prove Theorem \ref{thm:truncation}. 

	\begin{lemma}\label{lem:lct of truncations}
	Let $n$ be a positive integer. Let $x\in X$ be an $n$-dimensional klt singularity. Let $\Delta=\sum_{i=1}^m a_i\Delta_i$ be a non-negative $\bR$-linear combination of effective Cartier divisors $\Delta_i$. Let	$x'\in (X',\Delta'^k:=\sum_{i=1}^m a_i\Delta_i'^k)$ be a $k$-th analytic truncation of $x\in (X,\Delta)$. Then 
	\begin{enumerate}
	    \item for any positive real number $\eta\leq \min\{a_i\mid 1\leq i\leq m\}$, and any positive integer $k$, we have    $|\lct(X';\Delta'^k) -  \lct(X;\Delta)|\leq \frac{n}{k\eta}$, and
	    \item if $I\subset[0,1]$ is a DCC set, $a_i\in I$ for any $i$, then there exists a positive integer $k_0$ depending only on $n$ and $I$ satisfying the following.
	    
	    If $x\in (X,\Delta)$ is lc, then $x'\in (X',\Delta'^k)$ is also an lc singularity for any $k\geq k_0$.
	\end{enumerate}
		\end{lemma}
	
		


	\begin{proof}
	(1)	Denote $\mathrm{div}(f_i)=\Delta_i$ and $\mathrm{div}(f_{i,k}')=\Delta_i '^k$. Let $\fb_i:=(f_i)+\fm^k$ and $\fb_{i}':=(f_{i,k}')+\fm'^k$, where $\fm,\fm'$ are the maximal ideals of $\cO_{X,x},\cO_{X',x'}$ respectively. By definition  $\psi(\widehat{\fb_i})=\widehat{\fb_i'}$ where $\psi$ is the isomorphism between complete local rings in Definition \ref{def:analytic-truncation}. 	By \cite[Lemma 2.6 and Proposition 2.19]{dFEM11}, we know that 
		\[
		0\leq \lct(X;\prod_{i=1}^m \fb_i^{a_i})- \lct(X;\Delta)\leq \frac{n}{k\eta}, \qquad
		0\leq \lct(X';\prod_{i=1}^m \fb_i'^{a_i})-\lct(X';\Delta'^k)\leq \frac{n}{k\eta}.
		\]
		Since $\lct(X;\prod_{i=1}^m \fb_i^{a_i})=\lct(X';\prod_{i=1}^m \fb_i'^{a_i})$ by \cite[Proposition 2.11]{dFEM11}, the above inequalities yield
		\[
		|\lct(X';\Delta'^k)-\lct(X;\Delta)|\leq \frac{n}{k\eta}. 
		\]
		
		(2) We may assume that $1\in I$. Set $\eta\coloneqq \min I\backslash\{0\}$. On the one hand, by the ACC of log canonical thresholds for analytically bounded singularities \cite[Theorem 1.1]{dFEM11} (see also \cite{HMX14}), there exists a positive integer $k_0=k_0(n, I)$ depending only on $n$ and $I$, such that for any positive integer $k\ge k_0$, if $\lct (X';\Delta'^k)\geq 1-\frac{n}{k_0\eta}$, then $x'\in(X',\Delta'^k)$ is lc. On the other hand, by (1), \[
		\lct(X';\Delta'^k)\geq \lct(X;\Delta)-\frac{n}{k\eta}\ge 1-\frac{n}{k_0\eta}
		\]
		for any positive integer $k\ge k_0$. Hence $x'\in (X',\Delta'^k)$ is lc for any positive integer $k\ge k_0$ by our choice of $k_0$. 
		\end{proof}

	\begin{proposition}\label{prop: truncation keep nv(v)}
	Let $n, \eta, \gamma,~ \bxb,~ x\in (X,\Delta)$ be as in Theorem \ref{thm:truncation}. Let $V$ be a positive real number. 

Then there exists a positive integer $k_1$ depending only on $n,\eta,\gamma,V$ and $B\subset\cX\to B$ satisfying the following.
		
		Let $v\in \Val_{X,x}^\circ$ be a valuation such that  $\nvol_{(X,\Delta),x}(v)\le V$.
		Then for any positive integer $k\ge k_1$,
		\begin{itemize}
		    \item 	$v(\Delta_i)<kv(\fm)$ for any $i$, where $\fm$ is the maximal ideal of $\cO_{X,x}$, and
 \item $v(\Delta_i)=v'(\Delta_i'^k)$, and $\nvol_{(X,\Delta),x}(v)=\nvol_{(X',\Delta'^k),x'}(v')$ for any $i$, and any $k$-th analytic truncation $x'\in (X',\Delta'^k:=\sum_{i=1}^m a_i\Delta_i'^k)$ of $x\in (X,\Delta)$, where $v'=\phi(v)$, and $\phi: \Val_{X,x}^\circ\to \Val_{X',x'}^\circ$ is defined as in Proposition \ref{prop:an-iso}.
		\end{itemize}
	Moreover, if $v'=\ord_{S'}$ is a divisorial valuation, and $S'$ is a $\delta$-Koll\'ar component of $x'\in (X',\Delta'^k)$ for some positive real number $\delta$, then $S$ is also a $\delta$-Koll\'ar component of $x\in (X,\Delta)$, where $v=\phi^{-1}(v')=\ord_S$.
	\end{proposition}
	\begin{proof}
	Let $C_1(B\subset\cX\to B)$ be the positive constant defined as in Lemma \ref{lem:family Izumi}. 
		Let $k_1\coloneqq\lceil \frac{V}{\eta C_1}(\frac{1+\gamma}{\gamma})^n\rceil$, and $k\ge k_1$ a positive integer. If there exists $i$ such that $v(\Delta_i)\geq kv(\fm)$, then by Lemmata \ref{lem:kill boundary} and \ref{lem:family Izumi}, we get
		
			\begin{align*}
 \nvol_{(X,\Delta),x}(v)&\geq  \left(\frac{\gamma}{1+\gamma}\right)^{n}\nvol_{X,x}(v)\geq \left(\frac{\gamma}{1+\gamma}\right)^{n} C_1\frac{A_X(v)}{v(\fm)}\geq \left(\frac{\gamma}{1+\gamma}\right)^{n} C_1\frac{kA_X(v)}{v(\Delta_i)}\\
		    &\geq \left(\frac{\gamma}{1+\gamma}\right)^{n}C_1\cdot k\lct(X;\Delta_i)\geq \left(\frac{\gamma}{1+\gamma}\right)^{n}C_1\cdot k\eta>V,
		\end{align*}
a contradiction. Thus $v(\Delta_i)<kv(\fm)$ for any $i$. 
	
	    Let $\Delta_i=\mathrm{div}(f_i)$ and $\Delta_i'^k=\mathrm{div}(f_{i,k}')$. Then by Definition \ref{def:analytic-truncation}, $h_i\coloneqq \psi(f_i)-f_{i,k}'\in\widehat{\fm'}^k$, where $\psi$ is the isomorphism between complete local rings. Since $\hat v(f_i)=v(f_i)<kv(\fm)=k\hat v(\hat\fm)$ for any $i$, we get 
	    \[
	        v'(f_{i,k}')=\hat v'(f_{i,k}')=\hat v(f_i-\psi^{-1}(h_i))=\hat v(f_i)=v(f_i),
	    \]
	    where $\hat{v}$ and $\hat{v'}$ are the unique extensions of $v$ and $v'$ in $\Spec\,\widehat{R}$ and $\Spec\,\widehat{R'}$ respectively (see \cite[Corollary 5.11]{JM12} and the proof of Proposition \ref{prop:an-iso}). By Lemma \ref{lem:A-pull-back} and Proposition \ref{prop:an-iso}(1),  $A_{(X,\Delta)}(v)=A_{(X',\Delta'^k)}(v')$. By Proposition \ref{prop:an-iso}(2), $\nvol_{(X,\Delta),x}(v)=\nvol_{(X',\Delta'^k),x'}(v')$.
		
		Suppose that $v'=\ord_{S'}$ is a divisorial valuation, and $S'$ is a $\delta$-Koll\'ar component of $x'\in (X',\Delta'^k)$ for some positive real number $\delta$. By Proposition \ref{prop:an-iso}(4), $S$ is also a Koll\'ar component of $x\in X$. Let $\mu':Y'\to X'$ and $\mu:Y\to X$ be the corresponding plt blow-ups with Koll\'ar components $S'$ and $S$ respectively. Let $\Gamma$ and $\Gamma'$ be the different divisors of $(Y, S)$ and $(Y', S')$ on $S$ and $S'$ respectively. Then by Proposition \ref{prop:an-iso}(4), we know that there is an isomorphism $\psi_S: S\to S'$ induced from taking graded algebra of $\psi: \widehat{R}\to \widehat{R'}$ such that $\Gamma'=(\psi_S)_*\Gamma$.
		
		Let $\Delta_S$ and $\Delta'^k_{S'}$ be the different divisors of $(Y, S+\mu_*^{-1}\Delta)$ and $(Y', S'+\mu'^{-1}_*\Delta'^k)$ on $S$ and $S'$ respectively. Let $m_i:=v(f_i)=v'(f'_{i,k})$. Let $\bar{f}_i$ and $\bar{f}'_{i,k}$ be the images of $f_i$ and $f'_{i,k}$ in $\fa_{m_i}(v)/\fa_{m_i+1}(v)$ and $\fa_{m_i}(v')/\fa_{m_i+1}(v')$ respectively. Then we know that $\bar{f}_i$ and $\bar{f}'_{i,k}$ define effective $\bQ$-Cartier $\bQ$-divisors $\bar{\Delta}_i$ and $\bar{\Delta}'^k_{i}$ on $S$ and $S'$ respectively, such that	$\bar{\Delta}_i= (\mu_*^{-1}\Delta_i)|_S$ and 
		$\bar{\Delta}'^k_i= (\mu'^{-1}_*\Delta'^k_{i})|_{S'}$.
		It is clear that $
		\Delta_S=\Gamma+\sum_{i=1}^m a_i \bar{\Delta}_i$ and $
		\Delta'^k_{S'}=\Gamma'+\sum_{i=1}^m a_i \bar{\Delta}'^k_i$. 
		Since $\hat{v'}(\psi(f_i)-f_{i,k})\geq \hat{v'}(\widehat{\fm'}^k) > m_i$, we know that $\mathrm{gr}_v\psi: \mathrm{gr}_v R\xrightarrow{\cong} \mathrm{gr}_{v'} R'$ maps $\bar{f}_i$ to $\bar{f}'_{i,k}$. In particular, we have $(\psi_S)_*\bar{\Delta}_i=\bar{\Delta}'^k_{i}$ and hence $(\psi_S)_*\Delta_S=\Delta'^k_{S'}$, i.e. $(S,\Delta_S)\cong (S',\Delta'^k_{S'})$. 
		Since $(Y',S'+\mu'^{-1}_*\Delta'^k)$ is $\delta$-plt near $S'$, we know that $(S',\Delta'^k_S)$ is $\delta$-klt and $\delta\leq 1$.  It follows that $(S,\Delta_{S})$ is also $\delta$-klt. By the inversion of adjunction \cite[Corollary 1.4.5]{BCHM10}, $(Y,S+{\mu_{*}}^{-1}\Delta)$ is $\delta$-plt near $S$. We conclude that $S$ is a $\delta$-Koll\'ar component of $x\in (X,\Delta)$.
	\end{proof}
	
	\begin{proof}[Proof of Theorem \ref{thm:truncation}]
	Let $k_0:=\lceil\frac{2n}{\eta \gamma}\rceil$. By Lemma \ref{lem:lct of truncations}, $$\lct(X';\Delta'^k)\ge\lct(X;\Delta)-\frac{n}{k\eta}>1+\gamma-\frac{n}{k\eta}\ge 1+\frac{\gamma}{2}$$
	for any positive integer $k\ge k_0$. Let $k_1$ be the positive integer given by Proposition \ref{prop: truncation keep nv(v)} depending only on $n,\eta,\frac{\gamma}{2},V:=n^n$ and $B\subset\cX\to B$, and $k_2\coloneqq\max\{k_0,k_1\}$. 
		
	For any positive integer $k\ge k_2$, if $v\in\Val_{X,x}^{\circ}$ satisfies that $\nvol_{(X,\Delta),x}(v)\leq n^n$, then by the construction of $k_2$, $
		\nvol_{(X,\Delta),x}(v)=\nvol_{(X',\Delta'^k),x'}(v')$, where $v'=\phi(v)$. Recall that $(x'\in X'^{\rm an})\subset(\bxanb)$, and 
	$x\in (X,\Delta)$ is a $k$-th analytic truncation of $x'\in (X',\Delta'^k)$. Similarly, if $v'\in\Val_{X',x'}^{\circ}$ satisfies that $\nvol_{(X',\Delta'^k),x'}(v')\leq n^n$, then 	$\nvol_{(X',\Delta'^k),x'}(v')=\nvol_{(X,\Delta),x}(v)$, where $v=\phi^{-1}(v')$.
		
	Now the theorem follows from Theorem \ref{thm:nvol is bounded by n^n} and \cite[Theorem A]{Blu18b}.
	\end{proof}
	
	\begin{proposition}\label{prop:truncation is bdd}
		Let $n, \gamma, B\subset\cX\to B$ be as in Theorem \ref{thm:truncation}. Let $I\subset [0,1]$ be a finite set. Let $\eta:= \min ((I\setminus\{0\})\cup \{\frac{1}{2}\})$. Let $k_2$ be the positive integer from Theorem \ref{thm:truncation} depending only on $n, \eta, \gamma$ and $B\subset\cX\to B$. Let $k\geq k_2$ be a positive integer.
		
		Then there is an $\Rr$-Gorenstein family of klt singularities over a (possibly disconnected) smooth base  $T\subset (\cY,\cE)\to T$ depending only on $n, I, \gamma, k$ and $B\subset \cX\to B$ satisfying the following.
		
		Let $x\in(X,\Delta=\sum_{i=1}^m a_i\Delta_i)$ be a klt singularity, such that	
		\begin{enumerate}
			\item $(x\in X^{\rm an})\in (\bxanb)$, 
			\item $a_i\in I$ for any $i$, and
			\item each $\Delta_i\ge0$ is a Cartier divisor.
			\item $\lct(X,\Delta;\Delta)>\gamma$.
		\end{enumerate}
		Then there exists a closed point $t\in T$ such that $t\in (\cY_t,\cE_t)$ is a $k$-th analytic truncation of $x\in (X,\Delta)$. 
	\end{proposition}
	\begin{proof}
		By \cite[18.22]{Fli92}, $m$ is bounded from above. It suffices to show the proposition for any fixed positive integer $m$. 
		
		By Noetherian induction and Grothendieck's generic freeness theorem, possibly shrinking $B$, we may assume that $B=\Spec(A)$ and $\cO_{\cX,B}/I_B^k$ is a free $A$-module with a basis $\bar g_1,\ldots,\bar g_d$ for some $g_i\in \cO_{\cX,B}$, where $I_B$ is the ideal sheaf of $B\subset \cX$. 
		
		Possibly replacing $I$ with $I\cup\{0\}$, we may assume that $0\in I$. Denote by $I^m\subset \bR^m$ the $m$-th Cartesian power of $I$. Let $L\coloneqq|I^m|<+\infty$, and $I^m=\{\bm{a}_1,\ldots,\bm{a}_L\}$, where $|I|$ is the cardinality of $I$. Set  $U\coloneqq\bA_{A}^{dm}\backslash\{\bm{0}\}$, where $\bA_{A}^{dm}=\Spec\, A[x_1,\ldots,x_{dm}]$. For each $1\le l\le L$, let $\cE\supl\to U$ be the space, such that for any closed point $u=(u_{11},\ldots,u_{1m},\ldots,u_{d1},\ldots,u_{dm})\in U$, the fiber $\cE\supl_u$ parametrizes the divisor $\sum_{j=1}^ma_{jl}E_{j,u}\subset (\cX\times_{B}U)_u$,
		where $\bm{a}_l=(a_{1l},\ldots,a_{ml})\in I^m$, and $E_{j,u}\coloneqq (\sum_{i=1}^d u_{ij}g_i=0)$ for any $1\le j\le m$. Thus we get a family $U\subset (\cX\times_{B}U,\cE\supl)\to U$. By construction, there exist a closed point $u\in U$ and a positive integer $l$, such that $u\in ((\cX\times_{B}U)_u,\cE_u\supl)$ is a $k$-th analytic truncation of $x\in (X,\Delta)$. By Theorem \ref{thm:truncation} we have $\hvol(u, (\cX\times_{B}U)_u,\cE_u\supl)=\hvol(x, X,\Delta)>0$, which implies that $u\in ((\cX\times_{B}U)_u,\cE_u\supl)$ is a klt singularity.
	
	By Lemma \ref{lem: klt locus open}, for each $l$, possibly stratifying the base $U$ into a disjoint union of finitely many constructible subsets, we can assume that there exists a decomposition $U=\sqcup_{\alpha\in J_{1,l}} U_{\alpha}\bigsqcup \sqcup_{\alpha\in J_{2,l}}U_{\alpha}$ into irreducible smooth strata $U_{\alpha}$, such that $U_{\alpha}\subset (\cX\times_{B}U_{\alpha},\cE\supl)\to U_{\alpha}$ is an $\Rr$-Gorenstein family of klt singularities over a smooth base $U_{\alpha}$ for any $\alpha\in J_{1,l}$ and $u'\in ((\cX\times_{B}U_{\alpha})_{u'},\cE\supl_{u'})$ is not klt for any $\alpha\in J_{2,l}$ and any closed point $u'\in U_{\alpha}$. Since $u\in ((\cX\times_{B}U)_u,\cE_u\supl)$ is klt, we know that $u\in U_\alpha$ for some $l$ and $\alpha\in J_{1,l}$. 
		Let $T:=\bigsqcup_{l,\alpha\in J_{1,l}} U_{\alpha}$, and $(\cY,\cE)\to T$ be the pullback of $\bigsqcup_{l} \left((\cX\times_{B}U,\cE^{(l)})\to U\right)$ by $T\to U^L$.
		Then $T\subset(\cY,\cE)\to T$ is an $\Rr$-Gorestein family of klt singularities over a smooth base. Let $t\in T$ be the unique preimage of $u$ under the injective map $T\to U$, then by construction $t\in (\cY_t,\cE_t) $ is isomorphic to $u\in ((\cX\times_{B}U)_u,\cE_u\supl)$. Thus $t\in (\cY_t,\cE_t)$ is a $k$-th analytic truncation of $x\in (X,\Delta)$. 
	\end{proof}
	
	\subsection{Singularities with analytic boundary}
	
	\begin{defn}
	Let $x\in X$ be a normal $\bQ$-Gorenstein singularity. Denote $(R,\fm):=(\cO_{X,x},\fm_{X,x})$. Let $\hat{x}\in \widehat{X}:=\Spec\,\widehat{R}$ be the completion of $x\in X$. Let $\fD:=\sum_{i=1}^m a_i \fD_i$ be a non-negative $\bR$-combination (i.e. $a_i\in\bR_{\geq 0}$) of effective Cartier divisors $\fD_i$ on $\hX$. We say that $\hx\in (\hX,\fD)$ is a \emph{$\bQ$-Gorenstein singularity with analytic $\bR$-boundary}. We use  \cite[Section 2]{dFEM11} to define klt and lc of such a singularity $\hx\in (\hX,\fD)$. 
	\end{defn}
	
	\begin{defn}\label{def:nv-analytic-bdry}
	    Let $\hx\in (\hX,\fD)$ be an $n$-dimensional $\bQ$-Gorenstein singularity with analytic $\bR$-boundary that is klt. Denote by $\iota: \hX\to X$ the completion morphism. 
	    \begin{enumerate}
	        \item For a valuation $v\in \Val_{X,x}^\circ$, we define the \emph{log discrepancy, volume, and normalized volume} of $\hat{v}$ (defined as in Proposition \ref{prop:an-iso}) with respect to $\hx\in (\hX,\fD)$ as
	        \[
	        A_{(\hX,\fD)}(\hat{v}) := A_X(v)- \hat{v}(\fD),
	        \qquad\vol_{\hX, \hx}(\hat{v}) :=\vol_{X,x}(v),
	        \]
	        \[
	        \hvol_{(\hX,\fD), \hx}(\hat{v}) :=  A_{(\hX,\fD)}(\hat{v})^n\cdot \vol_{\hX, \hx}(\hat{v}).
	        \]
	        We define the \emph{local volume} of $\hx\in (\hX,\fD)$ as 
	        \[
	        \hvol(\hx,\hX,\fD):=\inf_{v\in \Val_{X,x}^\circ}\hvol_{(\hX,\fD), \hx}(\hat{v}).
	        \]
	        \item We say that a projective birational map $\hmu: \hY\to \hX$ provides a \emph{Koll\'ar component} $\hS$ over $\hx\in (\hX,\fD)$ if there exists a plt blow-up $\mu: Y\to X$ over $x\in X$ and a Cartesian diagram
	    \begin{center}
	        \begin{tikzcd}
	         \hS\arrow[r, hook]\arrow[d, "\iota_S", "\cong" '] & \hY \arrow[r, "\hmu"] \arrow[d, "\iota_Y"] &  \hX \arrow[d, "\iota"]\\
	         S \arrow[r, hook] & Y \arrow[r, "\mu"] &  X
	        \end{tikzcd}
	    \end{center}
	    such that $(\hS, \hGamma +(\hmu_*^{-1}\fD)|_{\hS})$ is klt in the sense of \cite[Section 2]{dFEM11}, where $\Gamma$ is the different divisor of $S$ in $Y$ and $\hGamma:=\iota_S^{*} \Gamma$.
	    \end{enumerate}
	\end{defn}
	
	We note that the above definition only depends on the analytic isomorphism class of $x\in X$ due to the equivalence of valuations of finite log discrepancy and Koll\'ar components over analytic isomorphic singularities (see Proposition \ref{prop:an-iso}).
	
	\begin{defn}
	    Let $\hx\in (\hX,\fD)$ be a $\bQ$-Gorenstein singularity with analytic $\bR$-boundary. Assume that $\fD=\sum_{i=1}^m a_i\fD_i$ where $a_i\in \bR_{> 0}$ and $\fD_i=\mathrm{div}(h_i)$ with $h_i\in \hR$ for each $1\leq i\leq m$. Let $k$ be a positive integer. Fix a $\bk$-linear basis $\bar{g}_1,\cdots,\bar{g}_d$ of $R/\fm^k$. Let $g_j\in R$ be a lifting of $\bar{g}_j$. We define the \emph{$k$-th analytic truncation} $\fD^k$ of $\fD$ on $X$ as $\fD^k:=\sum_{i=1}^m a_i \fD_i^{k}$  where $\fD_i^k=\mathrm{div}(h_{i,k})$ and  $h_{i,k}\in R$ is the $\bk$-linear combination of $g_j$ such that $h_{i}-h_{i,k}\in \widehat{\fm}^k$. We also set $\fD^k=0$ when $\fD=0$. 
	\end{defn}

	\begin{theorem}\label{thm:truncation-an-bdry}
	    Let $\hx\in (\hX,\fD)$ be a $\bQ$-Gorenstein singularity with analytic $\bR$-boundary that is klt. Then we have 
	    \begin{enumerate}
	    	\item $\hvol(\hx,\hX,\fD)=\hvol(x,X,\fD^k)$ for $k\gg 1$ where $\fD^k$ is a $k$-th analytic truncation of $\fD$  on $X$.
	        \item $\hvol(\hx,\hX,\fD)=\inf_{\hS} \hvol_{(\hX,\fD),\hx}(\ord_{\hS})$ where $\hS$ runs over all Koll\'ar components over $\hx\in (\hX,\fD)$.
	    \end{enumerate}
	\end{theorem}
	
	\begin{proof}
	If $\fD=0$, then the statements follow from Theorem \ref{thm:lx-kc-minimizing} and Proposition \ref{prop:an-iso}. So we may assume that $\fD\neq 0$. Choose $\eta,\gamma\in \bR_{>0}$ such that $\lct(\hX,\fD;\fD)\geq \gamma$ and  $a_i\geq \eta$ for any $i$. Let $V:=n^n+1$. Then by similar arguments as the proof of Proposition \ref{prop: truncation keep nv(v)}, there exists $k_1=k_1(n, \eta,\gamma, V, x\in X) \in \bZ_{>0}$ such that for any positive integer $k\geq k_1$ and any valuation $v\in \Val_{X,x}^\circ$ satisfying $\hvol_{(\hX,\fD),\hx}(\hv)\leq V$, we have \begin{equation}\label{eq:an-bdry}
    \hv(\fD_i)< kv(\fm), \quad \hv(\fD_i)=v(\fD_i^k), \quad \textrm{and}\quad\hvol_{(\hX,\fD),\hx}(\hv)=\hvol_{(X,\fD^k),x}(v).
    \end{equation}
    By Definition \ref{def:nv-analytic-bdry} and Theorem \ref{thm:nvol is bounded by n^n}, we have that 
    \[
    \hvol(\hx, \hX,\fD)\leq \hvol(x, X) \leq n^n \quad \textrm{and}\quad \hvol(x,X,\fD^k)\leq n^n. 
    \]
    Since $V>n^n$, by \eqref{eq:an-bdry} we have that $\hvol(\hx,\hX,\fD)=\hvol(x, X, \fD^k)$ for any $k\gg 1$. This proves part (1). 
    
    Next we prove part (2). Fix an arbitrary $\epsilon\in (0,1)$. Let $k\geq k_1$ be a positive integer where $k_1$ is chosen as before. By Theorem \ref{thm:lx-kc-minimizing} and part (1), there exists a Koll\'ar component $S$ over $x\in (X,\fD^k)$ such that 
    \begin{equation}\label{eq:an-bdry2}
        \hvol_{(X,\fD^k),x} (\ord_S) \leq \hvol(x,X,\fD^k)+\epsilon = \hvol(\hx,\hX,\fD)+\epsilon < V.
    \end{equation}
    Let $\hS$ be the pull-back of $S$ under $\tau:\hX\to X$ as a Koll\'ar component over $\hx\in \hX$. Let $\mu:Y\to X$ (resp. $\hmu: \hY\to \hX$) be the plt blow-up providing $S$ (resp. $\hS$). By similar arguments to the proof of Propsition \ref{prop: truncation keep nv(v)} and \eqref{eq:an-bdry}, we know that
    \begin{equation}\label{eq:an-bdry3}
            \ord_{S}(\fD^k)=\ord_{\hS}(\fD) < k \ord_S(\fm)\quad \textrm{and}\quad (\hmu_*^{-1}\fD)|_{\hS} = (\hmu_*^{-1}\widehat{\fD^k})|_{\hS}.
    \end{equation}
    Thus we have $(\hS, \hGamma + (\hmu_*^{-1}\fD)|_{\hS})\cong (S, \Gamma + (\mu_*^{-1}\fD^k)|_S)$ is klt. This implies that $\hS$ is a Koll\'ar component over $\hx\in (\hX,\fD)$. Hence by \eqref{eq:an-bdry3} we have 
    \[
    A_{(\hX,\fD)}(\ord_{\hS}) =A_{\hX}(\ord_{\hS}) -\ord_{\hS}(\fD) = A_X(\ord_S)-\ord_S(\fD^k)=  A_{(X,\fD^k)}(\ord_{S}).
    \]
    Since the volumes of $\ord_S$ and $\ord_{\hS}$ are the same by Proposition \ref{prop:an-iso}, the inequality \eqref{eq:an-bdry2} implies that  
    \[
    \hvol_{(\hX,\fD),\hx}(\ord_{\hS}) = \hvol_{(X,\fD^k),x}(\ord_S)\leq \hvol(\hx, \hX, \fD)+\epsilon.
    \]
    Thus the proof of part (2) is finished as $\epsilon$ can be arbitrarily small.
	\end{proof}

	\section{Proofs of main results}
	\subsection{Existence of \texorpdfstring{$\delta$}{}-plt blow-ups} In this subsection, we will prove Theorems \ref{thm:existence of delta-plt blow-up} and \ref{thm:bddness for nvol}.
	\begin{theorem}\label{thm: lct positive lower bound implies existence of delta-plt blow-up}
		Let $n\ge 2$ be a positive integer, $\eta,\epsilon $ positive real numbers, and $B\subset \cX\to B$ a $\bQ$-Gorenstein family of $n$-dimensional klt singularities. Then there exists a positive real number $\delta$ depending only on $n,\eta,\gamma$ and $B\subset \cX\to B$ satisfying the following.
		
		If $x\in (X,\Delta=\sum_{i=1}^m a_i\Delta_i)$ is an $n$-dimensional klt singularity such that
		\begin{enumerate}
		\item  $(x\in X^{\rm an})\in (B\subset \cX^{\rm an}\to B)$,
			\item $a_i>\eta$ for any $i$, 
				\item each $\Delta_i\geq 0$ is a $\bQ$-Cartier Weil divisor, and
			\item $\lct(X,\Delta;\Delta)>\gamma$,
		\end{enumerate}
		then $x\in(X,\Delta)$ admits a $\delta$-plt blow-up.
	\end{theorem}
	\begin{proof}
	Let $l\coloneqq \lceil \frac{2+\gamma}{\gamma \eta}\rceil$, $\Delta^{+}\coloneqq\sum_{i=1}^m\frac{\lceil la_i\rceil}{l}\Delta_i$. Then $\Delta^{+}\ge\Delta$, $ (1+\frac{\gamma}{2})\cdot\frac{\lceil la_i\rceil }{l}\le (1+\gamma)a_i$ for any $i$, and $\lct(X,\Delta^{+};\Delta^{+})>\frac{\gamma}{2}$.
	
		Since for any positive real number $\delta$, any $\delta$-plt blow-up of $x\in (X,\Delta^{+})$ is also a $\delta$-plt blow-up of $x\in (X,\Delta)$, possibly replacing $(X,\Delta)$ by $(X,\Delta^+)$, and $\gamma$ by $\frac{\gamma}{2}$, we may assume that any coefficient $a_i$ of $\Delta$ belongs to the finite rational set $I=\frac{1}{l}\Zz\cap[0,1]$.
		
		By Theorem \ref{thm: family kc with nv bdd} and Proposition \ref{prop: truncation keep nv(v)}, there exists a positive real number $\delta_0$ which only depends on $n,\eta,\epsilon$ and $\bxb$, such that $x\in X$ admits a $\delta_0$-plt blow-up. By \cite[Theorem 1.2]{dFEM11}, there exists a positive real number $\epsilon_0$ which only depends on $\bxb$, such that $x\in X$ is $\epsilon_0$-lc. Thus by \cite[Theorem 1.6]{HLS19}, for each $i$, the Cartier index of $\Delta_i$ near $x$ is bounded from above by a positive integer $N$ which only depends on $n,\eta,\gamma$ and $\bxb$. Therefore, possibly replacing $\Delta_i$ with $N\Delta_i$ and $I$ with $\frac{1}{N}I$, we may assume that each $\Delta_i$ is Cartier.
		
		Let $k_1$ be the positive integer given in Proposition \ref{prop: truncation keep nv(v)} depending only on $n,\eta,\gamma,V:=n^n+1$ and $\bxb$. 
		Let $k_2$ be the positive integer given in Theorem \ref{thm:truncation} depending only on $n, \eta, \gamma$ and $\bxb$. Choose $k:=\max\{k_1,k_2\}$.
		It suffices to show that there exist a $k$-th analytic truncation $x'\in(X',\Delta'^k)$ of $x\in (X,\Delta)$, and a $\delta$-Koll\'ar component $S'$ of $x'\in(X',\Delta'^k)$ such that $\nvol_{(X',\Delta'^k),x'}(\ord_{S'})\le n^n+1$, for some positive real number $\delta$ depending only on $n,\eta,\gamma$ and $\bxb$.
		
		By Proposition \ref{prop:truncation is bdd}, there is an $\Rr$-Gorenstein family of klt singularities over a smooth base $T\subset (\cY,\cE)\to T$, such that
		$t\in (\cY_t,\cE_t)$ is a $k$-th analytic truncation of $x\in (X,\Delta)$ for some closed point $t\in T$. Now the theorem follows from Theorem \ref{thm: family kc with nv bdd}.
	\end{proof}
	
\begin{proof}[Proof of Theorem \ref{thm:existence of delta-plt blow-up}]
 This follows from Theorem \ref{thm:estimate on lct} and
 Theorem \ref{thm: lct positive lower bound implies existence of delta-plt blow-up}.
 \end{proof}

	If the coefficients of $\Delta$ belong to a finite set, then we may relax the assumption ``each $\Delta_i$ is a $\Qq$-Cartier Weil divisor'' in Theorem \ref{thm:existence of delta-plt blow-up} to ``each $\Delta_i$ is a Weil divisor'', as stated in Conjecture \ref{conj:existence of delta-plt blow-up}.
	\begin{theorem}\label{thm:existence of delta-plt blow-up finite coeff}
		Let $n\ge 2$ be a positive integer, $\epsilon$ a positive real number, $I$ a finite set, and $\bxb$ an $\bQ$-Gorenstein family of $n$-dimensional klt singularities. Then there exists a positive real number $\delta$ depending only on $n,\epsilon,I$ and $\bxb$ satisfying the following. 
		
		If $x\in (X,\Delta=\sum_{i=1}^m a_i\Delta_i)$ is an $n$-dimensional klt singularity such that
		\begin{enumerate}
		    \item  $(x\in X^{\rm an})\in(\bxanb)$, 
			\item $a_i\in I$ for any $i$, 
				\item each $\Delta_i\geq 0$ is a Weil divisor, and
			\item $\nvol(x,X,\Delta)>\epsilon$,
			\end{enumerate}
		then $x\in(X,\Delta)$ admits a $\delta$-plt blow-up.
	\end{theorem}
	
	\begin{proof}
	Suppose $I=\{c_1,\cdots, c_{|I|}\}$ where $c_i<c_j$ for any $i<j$. Since each $a_i$ is the same as $c_j$ for some $1\leq j\leq |I|$, we may write $\Delta=\sum_{j=1}^{|I|} c_j \Delta_j'$ where  $\Delta_j'\geq 0$ is a Weil divisor. 
	By Lemma \ref{lem:nv decomposable R-complements}, there exist positive real numbers $t_i$, rational points $\bm{a}_i=(a_i^1,\ldots,a_i^{|I|})\in\Qq^{|I|}$ for $1\leq i\leq l$ depending only on $n$ and $\bm{c}\coloneqq(c_1,\ldots,c_{|I|})$, such that  $\Delta=\sum_{i=1}^l t_i\Delta_{(i)}$, where $\Delta_{(i)}\coloneqq \sum_{j=1}^{|I|} a_i^j\Delta_j'$ is a $\Qq$-Cartier $\Qq$-divisor for any $i$. Let $N$ be a positive integer such that $N a_i^j$ is a positive integer for any $i,j$. Since $N\Delta_{(i)}$ is a $\bQ$-Cartier Weil divisor for any $i$, and $\Delta=\sum_{i=1}^l \frac{t_i}{N}(N\Delta_{(i)})$, Theorem \ref{thm:existence of delta-plt blow-up finite coeff} follows from Theorem \ref{thm:existence of delta-plt blow-up} as  $\{\frac{t_i}{N}\}_{1\leq i\leq l}$ has a positive lower bound.
	\end{proof}

	\begin{proof}[Proof of Corollary \ref{cor: accmld when nv bdd}]
		This follows from Theorem \ref{thm:existence of delta-plt blow-up} and \cite[Theorem 1.3]{HLS19}.
	\end{proof}
	
	The following example shows that both Theorems \ref{thm: lct positive lower bound implies existence of delta-plt blow-up} and \ref{thm:existence of delta-plt blow-up} no longer hold without assuming the positive lower bound of the nonzero coefficients of the boundary. 
	
	\begin{example}\label{eg:coeff no lower bound}
		Let $k>2$ be a positive integer and $\epsilon\in\bQ\cap[1/4,1/2)$.
		Consider the klt singularity $o\in(\bA^2, D_k:=(1-\epsilon)(\frac{1}{k-1}+\frac{1}{k})C_k)$, where $o$ is the origin and $C_k\coloneqq(x^{k-1}=y^{k})$.
		Let $E_k\subset Y_k\xrightarrow{\mu_k} \bA^2$ be the
		weighted blow-up of $o\in\bA^2$ with weight $(k,k-1)$. Let $\Delta_k\coloneqq \Diff_{E_k}((\mu_k)^{-1}_*D_k)$. Then $A_{(\bA^2,D_k)}(E_k)=(2k-1)\epsilon$, and by adjunction formula, we have
		\begin{equation*}
		\Delta_k=\left(1-\frac{1}{k-1}\right)p+\left(1-\frac{1}{k}\right)q
		+(1-\epsilon)\left(\frac{1}{k-1}+\frac{1}{k}\right)r,
		\end{equation*}
		where $p$ and $q$ are the singularities of $Y_k$ along $E_k$ and $r$ is a smooth point. So we get $\alpha(E_k,\Delta_k)=k^{-1}\epsilon^{-1}(\frac{1}{k}+\frac{1}{k-1})^{-1}=\frac{k-1}{\epsilon(2k-1)}\geq 1$
		for $k\gg 1$. Hence $o\in (\bA^2,D_k)$ is weakly exceptional, see for example, \cite[Theorem 4.3]{Pro00}. In particular, $E_k$ is the unique Koll\'{a}r component of $0\in (\bA^2,D_k)$. Thus for $k\gg 1$ we have
		\begin{align*}
		&\nvol(o,\bA^2, D_k)=\nvol_{(\bA^2,D_k),o}(\ord_{E_k})\\
		=&A_{(\bA^2,D_k)}(E_k)^2\cdot\vol_{\bA^2,o}(\ord_{E_k})=
		\frac{\epsilon^2(2k-1)^2}{k(k-1)}> \frac{1}{4}.
		\end{align*}
		However, for any given positive real number $\delta$, there exists a positive integer $k$, such that $0\in (\bA^2,D_k)$ does not admit a $\delta$-plt blow-up as the total log discrepancy of $(E_k,\Delta_k)$
		is $\frac{1}{k}\to 0$. 
	\end{example}
	
    The goal of the rest of this subsection is to prove Theorem \ref{thm:bddness for nvol}, that is, the converse direction of Conjecture \ref{conj:existence of delta-plt blow-up}. It is a consequence of Birkar--Borisov--Alexeev--Borisov Theorem, and an inequality involving the local volume and the $\delta$-invariant (see Proposition \ref{prop:control of nvol}). We will not need this result in the rest of this paper. 
	
    The $\delta$-invariant of a $\bQ$-Fano variety is introduced in \cite[Theorem 0.3]{FO18}, and is further studied by many people. We refer readers to \cite{Blu18b} for the definition of the $\delta$-invariant for log Fano pairs. Recall the following characterization of K-semistability in terms of the $\delta$-invariant.
    \begin{theorem}[{\cite[Theorem D]{Blu18b},\cite[Theorem 0.3]{FO18}, \cite[Theorem B]{BJ17}}]  \label{thm: kss equiv delta larger than or equal to 1}
        Let $(X, \Delta)$ be a log Fano pair, where $\Delta$ is a $\Qq$-divisor. Then $(X, \Delta)$ is K-semistable if and only if $\delta(X, \Delta) \geq 1$.
    \end{theorem} 



    \begin{proposition}\label{prop:control of nvol}
        	Let $x\in(X,\Delta)$ be an $n$-dimensional klt singularity, where $\Delta$ is a $\Qq$-divisor. Let $\mu:(Y,S)\to (X,x)$ be a plt blow-up of $(X,\Delta)$, and $S$ the corresponding Koll\'ar component of $x\in (X,\Delta)$. Let $\Delta_Y\coloneqq \mu_*^{-1}\Delta$, and $K_S+\Delta_S\coloneqq (K_Y+\Delta_Y+S)|_{S}$. Then
		\begin{equation*}
		\nvol(x,X,\Delta)\geq \nvol_{(X,\Delta),x}(\ord_S)\cdot\min\{1,\delta(S,\Delta_S)\}^n.
		\end{equation*}
	\end{proposition}
	\begin{proof}
	If $(S,\Delta_S)$ is $K$-semistable, then by Theorem \ref{thm:minimizer is K-ss}, $\ord_S$ is the minimizer of $\nvol(x,X,\Delta)$. Thus $	\nvol(x,X,\Delta)=\nvol_{(X,\Delta),x}(\ord_S)$.
	
	Otherwise, $(S,\Delta_S)$ is not $K$-semistable. By Theorem \ref{thm: kss equiv delta larger than or equal to 1}, $\delta(S,\Delta_S)<1$. It suffices to show that for any positive real number $\beta<\delta(S,\Delta_S)$,
	$$	\nvol(x,X,\Delta)\geq \nvol_{(X,\Delta),x}(\ord_S)\cdot\beta^n.$$
	
	By Theorem \cite[Theorem 7.2]{BL18b}, there exists an effective $\Qq$-divisor $D_S\sim_{\Qq}-(K_S+\Delta_S)$, such that $(S,\Delta_S+(1-\beta)D_S)$ is K-semistable and $(S,\Delta_S+D_S)$ is klt. By \cite[Lemma 7.1]{HLS19}, there exists an effective $\Qq$-divisor $D_Y\sim_{\Qq}-(K_Y+\Delta_Y+S)$, such that $D_Y|_{S}=D_S$. By inversion of adjunction \cite[Corollary 1.4.5]{BCHM10}, $(Y,\Delta_Y+D_Y+S)$ is plt near $S$. Let $D\coloneqq\mu_{*}D_Y$. Then $A_{(X,\Delta+D)}(\ord_S)=0$ which implies that $A_{(X,\Delta+(1-\beta)D)}(\ord_S)=\beta A_{(X,\Delta)}(\ord_S)$. By Theorem \ref{thm:minimizer is K-ss}, $\ord_S$ is the minimizer of $\nvol(x,X,\Delta+(1-\beta)D)$. Thus 
	\begin{align*}
	\nvol(x,X,\Delta)& \geq \nvol(x,X,\Delta+(1-\beta)D) =\beta^n A_{(X,\Delta)}(\ord_S)^n\vol_{X,x}(\ord_S)= \beta^n\nvol_{(X,\Delta),x}(\ord_S),
	\end{align*}
	and we are done.
	\end{proof}

	\begin{proof}[Proof of Theorem \ref{thm:bddness for nvol}]
		
		We first show the theorem for the case when $\Delta$ is a $\Qq$-divisor. 
		
		Let $\mu:(Y,S+\mu_*^{-1}\Delta)\to (X,\Delta)$ be the 
		$\delta$-plt blow-up. By the adjunction formula (see e.g. \cite[(17.2.2)]{Fli92}),
		$(S,\Delta_S)$ is $\delta$-klt, where $\Delta_S\coloneqq\Diff_{S}(\mu_*^{-1}\Delta)$.
		Then by Birkar--Borisov--Alexeev--Borisov Theorem \cite[Theorem 1.1]{Bir16}, $S$ belongs to a bounded family. 
		
		Let $L\coloneqq(-S)|_S$. By \cite[Proposition 4.4]{HLS19}, there exists a positive integer $M=M(\delta,\epsilon_1,n)$ which only depends on $\delta,\epsilon_1$ and $n$, such that $ML$ is a Cartier divisor on $S$.
		
		Now by Proposition \ref{prop:control of nvol}, we know that
		\[
		\nvol(x,X,\Delta)\geq \nvol_{(X,\Delta),x}(\ord_S)\cdot\min\{1,\delta(S,\Delta_S)\}^n.
		\]
		By \cite[Lemma 2.7]{LX16}, we have
		\[
		\nvol_{(X,\Delta),x}(\ord_S)=A_{(X,\Delta)}(\ord_S)^n\cdot L^{n-1}
		\geq \epsilon_1^n M^{1-n}(ML)^{n-1}\geq \epsilon_1^n M^{1-n}.
		\]
		Thus it is enough to give a positive lower bound of $\delta(S,\Delta_S)$. By \cite[Theorem C]{Blu18b} (see also \cite[Theorem A]{BJ17}),  $\delta(S,\Delta_S)\geq 
		\alpha(S,\Delta_S)$, where $\alpha(S,\Delta_S)$ is Tian's $\alpha$-invariant. Since $(S,\Delta_S)$
		is a $\delta$-klt log Fano pair, \cite[Theorem 1.4]{Bir16}
		implies that there exists a positive real number $t=t(\delta,n)$ which only depends on $\delta$ and $n$, such that $\alpha(S,\Delta_S)\geq t$. Therefore,
		\[
		\nvol(x,X,\Delta)\geq \epsilon_1^n M^{1-n}t^n.
		\]
		
		For the general case, let $M'\coloneqq M(\frac{\delta}{2},\frac{\epsilon_1}{2},n)$, $t'\coloneqq t(\frac{\delta}{2},n)$, and $\epsilon$ any positive real number such that $\epsilon<(\frac{\epsilon_1}{2})^n M'^{1-n}t'^n$. By Lemma \ref{lem: nv lipschitz in rational envelope}, there exists a $\Qq$-divisor $\Delta'$, such that $x\in (X,\Delta')$ admits a $\frac{\delta}{2}$-plt blow-up, $\mld(X,\Delta')\ge\frac{\epsilon_1}{2}$, and 
		\[
		\nvol(x,X,\Delta)\ge \nvol(x,X,\Delta')-\left(\left(\frac{\epsilon_1}{2}\right)^n M'^{1-n}t'^n-\epsilon\right)\ge \epsilon.
		\]
		Therefore the theorem is proved. 
	\end{proof}
	\subsection{Boundedness of Cartier indices in a family}
		In this section, we will show Theorem \ref{thm:boundedness of global index}. 

	\begin{theorem}\label{thm:boundedness of local index}
		Let $B\subset (\cX,\cD)\to B$ be an $\Rr$-Gorenstein family of klt singularities, then there exists a positive integer $N$ such that for any closed point $b\in B$, if $D$ is a $\bQ$-Cartier Weil divisor near $b\in\cX_b$, then $ND$ is Cartier near $b\in\cX_b$. 
	\end{theorem}
	
	\begin{proof}Possibly shrinking $B$ and replacing it by a finite \'{e}tale covering, we may assume that $B\subset (\cX,\cD)\to B$ admits a fiberwise log resolution. Thus there exists a positive real number $\epsilon_0$, such that $b\in (\cX_b,\cD_b)$ is $\epsilon_0$-lc. Now Theorem follows from Theorem \ref{thm: family kc with nv bdd} and \cite[Theorem 1.6]{HLS19}.
	\end{proof}	
	
	\begin{remark}
	Theorem \ref{thm:boundedness of local index} could also be proved by Theorem \ref{thm: finite degree formula}.
	\end{remark}
	
	\begin{proof}[Proof of Theorem \ref{thm:boundedness of global index}]
	Let $(X',\Delta')\to (X,\Delta)$ be a small $\Qq$-factorialization, $K_{X'}+{B'}$ the pullback of $K_X+B$, and $D'$ the pullback of $D$. By \cite[Theorem 1.2]{Bir18}, $(X',\Delta')$ belongs to a bounded family. Since $X'\to X$ is a blow-up, there exists a $\Qq$-divisor $H'\ge0$ on $X'$, such that $-H'$ is ample over $X$. In particular, $K_{X'}+\Delta'+H'$ is antiample over $X$. Possibly rescaling $H'$, we may assume that $(X',\Delta'+H')$ is klt. Then $X'\to X$ is a $(K_{X'}+\Delta'+H')$-negative contraction of an extremal face of the Mori-Kleiman cone of $X'$. Hence the Cartier index of $D'$ and $D$ are the same by the cone theorem. Thus possibly replacing $(X,\Delta)$ with $(X',0)$, we may assume that $X$ is $\Qq$-factorial, and $\Delta=0$. 
	
	Let $\cX\to B$ be the bounded family. Possibly shrinking $B$, using Noetherian induction and replacing $\cX$ with its normalization, by \cite[Proposition 2.4]{HX15} and generic flatness, we may assume that $\cX$ is normal, $\cX\to B$ is flat, and $({\cX},0)$ is klt. 
	
	Consider the following diagram:
	\begin{equation*}
	\xymatrix{
	\cX\ar[rdd]_\id\ar[rrd]^\id \ar[rd]^\sigma& &\\
	&\cX\times_B\cX\ar[r]\ar[d]_p &\cX\ar[d]_\pi\\
	&\cX\ar[r] &B
	}\label{derivedequiv}
	\end{equation*}
	where $\sigma(\cX)$ is a section of $p:\cX\times_B\cX\to\cX$. We remark that  $\sigma(\cX)\subset \cX\times_B\cX\to \cX$ is a $\Qq$-Gorestein family of klt singularities according to Remark \ref{rem: basechage q-gorestein family sing}.
	
	Since $(\cX\times_B\cX)_x\cong \cX_{\pi(x)}=X$ for any closed point $x\in X\subset \cX$, $D$ is a $\Qq$-Cartier Weil divisor on $(\cX\times_B\cX)_x$. 
By Theorem \ref{thm:boundedness of local index}, there exists a positive integer $N$ which only depends on $\cP$, such that $ND$ is Cartier near $x\in (\cX\times_B\cX)_x\cong X$ for any closed point $x\in X$. Hence $ND$ is Cartier.
	\end{proof}	
	
	\subsection{Discreteness and ACC for local volumes}
	
	\begin{proof}[Proof of Theorem \ref{thm:discreteACC}(1)]
		We may assume that $n\ge 2$. It suffices to prove that for any positive real number $\epsilon$, the set 
		\[
		V_\epsilon:=\left\{ \nvol(x,X,\Delta)\left| \begin{array}{l}
		(x\in X^{\rm an})\in (B\subset \cX^{\rm an}\to B),~ \Delta=\sum_{i=1}^m a_i\Delta_i,\\
		\text{ where $a_i\in I$,} \text{ each $\Delta_i\ge0$ is a $\bQ$-Cartier}\\
		\text{ Weil divisor,} \textrm{ and }\hvol(x,X,\Delta)>\epsilon.
		\end{array}\right.
		\right\}.
		\]
		is finite. 
		By Theorem \ref{thm:existence of delta-plt blow-up} and \cite[Theorem 1.6]{HLS19}, there exists a positive integer $N$ which only depends on $n,I,\epsilon$, and $\bxb$, such that $N\Delta_i$ is Cartier near $x$ for any $i$. Possibly replacing $I$ with $\frac{1}{N}I$, we may assume that each $\Delta_i$ is Cartier. By Theorems \ref{thm:estimate on lct} and  \ref{thm:truncation}, there exists a positive integer $k$ depending only on $n,I,\epsilon$ and $\bxb$, such that if $\nvol(x,X,\Delta)\in V_\epsilon$, then 
		\[
		\nvol(x,X,\Delta)=\nvol(x',X',\Delta'^k),
		\]
		for any $k$-th analytic truncation of $x\in (X,\Delta)$. By Proposition \ref{prop:truncation is bdd} and Theorem \ref{thm:constructibility R-div}, $\nvol(x,X,\Delta^k)$ belongs to a finite set.
	\end{proof}
	
	\begin{proof}[Proof of Theorem \ref{thm:discreteACC}(2)]
		We may assume that $n\ge 2$. Assume to the contrary that there exists a sequence of klt singularities $x_j\in(X_j,\Delta\supj=\sum_{i=1}^{m_j} a_i\supj \Delta_i\supj)$, such that $a_i\supj\in I$, $(x_j\in X_j^{\rm an})\in (B\subset \cX^{\rm an}\to B)$, and the sequence of normalized volumes $\{V\supj\coloneqq \nvol(x_j,X_j,\Delta\supj)\}_{j=1}^{\infty}$ is strictly increasing. In particular, there exists $\epsilon>0$ such that $\nvol(x_j,X_j,\Delta\supj)>\epsilon$ for all $j$. 
		
		
		By Theorem \ref{thm:estimate on lct}, possibly shrinking $X_j$ near $x_j$, there exists a positive real number $\gamma$ which only depends on $n,\bxb$ and $\frac{\epsilon}{2}$ such that $\lct(X_j,\Delta\supj;\Delta\supj)>\gamma$. 
		By \cite[18.22]{Fli92}, possibly passing to a subsequence, we may assume that $m_j=m$ for any $j$, and $\{a_i\supj\}_{j=1}^{\infty}$ is increasing for each $1\leq i\leq m$. Set $a_i\coloneqq \lim_{j\to+\infty} a_i\supj\le 1$ for each $1\leq i\leq m$. By Theorem \ref{thm:uniform lipschitz}, possibly passing to a subsequence, we may assume that $\nvol(x_j,X_j,\Delta'^{(j)})>\epsilon/2$, where $\Delta'^{(j)}\coloneqq \sum_{i=1}^m a_i\Delta_i\supj$. In particular, $x_j\in (X_j,\Delta'^{(j)})$ is klt. Since $I'\coloneqq \{a_1,\ldots,a_m\}$ is a finite set, by Theorem \ref{thm:discreteACC}(1), $\nvol(x_j,X_j,\Delta'^{(j)})$ belongs to a finite set. Possibly passing to a subsequence, we may assume that there exists a positive real number $V>\frac{\epsilon}{2}$, such that $\nvol(x_j,X_j,\Delta'^{(j)})=V$ for any $j$. 
		
		By Theorem \ref{thm:estimate on lct}, we have $\lct(X_j,\Delta'^{(j)};\Delta'^{(j)})>\gamma$.
		By Theorem \ref{thm:uniform lipschitz} again, there exists a positive real number $C$ which only depends on $n,I',\gamma$ and $\bxb$, such that
		\begin{equation*}
		V\leq V\supj \leq V+C\sum_{i=1}^m|a_i\supj-a_i|.
		\end{equation*}
		Let $j\to +\infty$, we derive a contradiction as we assume that $\{V\supj\}_{j=1}^{\infty}$ is strictly increasing.
	\end{proof}
	
	If the coefficients of $\Delta$ belong to a finite set, then we may relax the assumption ``each $\Delta_i\ge0$ is a $\Qq$-Cartier Weil divisor'' in Theorem \ref{thm:discreteACC}(1) to ``each $\Delta_i\ge0$ is a Weil divisor'', as stated in Conjecture \ref{conj:ACC for nvol}.
	\begin{theorem}\label{thm:discrete nv}
		Let $n$ be a positive integer and let $I\subset [0,1]$ be a finite set. Let $B\subset \cX\to B$ be a $\bQ$-Gorenstein family of $n$-dimensional klt singularities. The set of local volumes
		\[  
		    \Vol_{B\subset\cX\to B, I}:=  \left\{ \nvol(x,X,\Delta)\left| \begin{array}{l}
		    (x\in X^{\rm an})\in (B\subset \cX^{\rm an}\to B),~ \Delta=\sum_{i=1}^m a_i\Delta_i,\\
		    \text{ where $a_i\in I$,} \text{ each $\Delta_i\ge0$ is a Weil divisor,}\\
		    \textrm{ and }x\in(X,\Delta) \text{ is klt.} 
		    \end{array}\right.
		\right\}\]
		has no non-zero accumulation point.
	\end{theorem}
	\begin{proof}
	We may assume that $n\ge 2$. By Lemma \ref{lem:nv decomposable R-complements}, there exist positive real numbers $t_i$, rational points $\bm{a}_i=(a_i^1,\ldots,a_i^m)\in\Qq^m$ depending only on $n$ and $\bm{a}\coloneqq(a_1,\ldots,a_m)$, such that  $\Delta=\sum_{i=1}^l t_i\Delta_{(i)}$, where $\Delta_{(i)}\coloneqq \sum_{j=1}^m a_i^j\Delta_j$ is a $\Qq$-Cartier $\Qq$-divisor for any $i$. Let $N$ be a positive integer such that $N a_i^j$ is a positive integer for any $i,j$. Since $N\Delta_{(i)}$ is a Weil divisor for any $i$, and $\Delta=\sum_{i=1}^l \frac{t_i}{N}(N\Delta_{(i)})$, Theorem \ref{thm:discrete nv} follows from Theorem \ref{thm:discreteACC}(1).
	\end{proof}

	The following result is a direct consequence of  Theorem \ref{thm:discreteACC}.

		\begin{cor}\label{cor:smooth}
		Let $n$ be a positive integer and let $I\subset [0,1]$ be a subset. Consider the set of local volumes
		\[
		\Vol_{n,I}^{\rm sm}:=\left\{\hvol(x,X,\Delta)\left| \begin{array}{l}
		\textrm{$x\in X$ is $n$-dimensional smooth}, ~x\in (X,\Delta)\textrm{ is klt}, \\
		\text{where }\Delta=\sum_{i=1}^m a_i\Delta_i,~\text{$a_i\in I$ for any $i$,} \\
		\text{and each $\Delta_i\ge 0$ is a Weil divisor}
		\end{array}\right.
		\right\}.
		\]
		\begin{enumerate}
			\item If $I$ is finite, then $\Vol_{n,I}^{\rm sm}$ has no non-zero accumulation point.
			\item If $I$ satisfies the DCC, then $\Vol_{n,I}^{\rm sm}$ satisfies the ACC.
		\end{enumerate}
	\end{cor}

	
	
	\subsection{Case of surfaces}
	In this subsection, we will prove Theorem \ref{thm:surface ACC} and Theorem \ref{thm:delta-plt-blowup-lowdim}(1). 
	\begin{lem}\label{lem: surface klt sing nv an bdd}
	Let $\epsilon$ be a positive real number. There exists a finite set of surface klt singularities $\{(x_{i}\in X_{i})\}_{i}$ depending only on $\epsilon$ satisfying the following. 
	
	For any  klt surface singularity $x\in X$ such that
	$\nvol(x,X,\Delta)>\epsilon$, $(x\in X)$ is analytically isomorphic to $(x_{i}\in X_{i})$ for some $i$.
	\end{lem}
	\begin{proof}
	It is well-known that $(x\in X)$ is analytically isomorphic to a klt surface singularity $(x'\in X')$ which is a quotient of $0\in \bA^2$ by a finite group $G$ containing no pseudo-reflections, see for example, \cite[Proposition 4.18]{KM98}. By Proposition \ref{prop:an-iso}(3), $\nvol(x',X')=\nvol(x,X)>\epsilon$.  
	
	Let $(y\in Y)\coloneqq (0\in \bA^2)$. There exists a finite Galois morphism $f:(y\in Y)\to (x'\in X')$, such that $K_{Y}=f^{*}K_{X'}$, and $\deg f=|G|$. By Theorem \ref{thm: finite degree formula} and Theorem \ref{thm:nvol is bounded by n^n}, $$\epsilon<\nvol(x',X')=\frac{1}{|G|}\nvol(y,Y)\le \frac{4}{|G|},$$
	which implies that $|G|< \frac{4}{\epsilon}$.
	
	It is well known that any finite subgroups of $\mathrm{PGL}_{2}(\bk)$ is isomorphic to $\mathbb{Z}/r, D_{r}$ (the dihedral group), $\mathfrak{A}_{4}$, $\mathfrak{S}_{4}$ or $\mathfrak{A}_{5}$, and there is only one conjugacy class for each of these groups (see e.g. \cite{Kle93}). As $G\in \mathrm{GL}_{2}(\bk)$ and $|G|<\frac{4}{\epsilon}$, $G$ is isomorphic to $\mathbb{Z}/r, D_{r}$, $\mathfrak{A}_{4}$, $\mathfrak{S}_{4}$ or $\mathfrak{A}_{5}$ up to a scaling of a $\lfloor\frac{4}{\epsilon}\rfloor!$-th unit root, and there is only one conjugacy class for each of these groups. Thus the isomorphism class of $(x'\in X')$ belongs to a finite set 
    only depending on $\epsilon$, and we are done.
	\end{proof}
	
	\begin{proof}[Proof of Theorem \ref{thm:surface ACC}] 
	This follows from Lemma \ref{lem: surface klt sing nv an bdd} and Theorem \ref{thm:discreteACC}.
	\end{proof}
	
	\begin{proof}[Proof of Theorem \ref{thm:delta-plt-blowup-lowdim}(1)]
		This follows from Lemma \ref{lem: surface klt sing nv an bdd} and Theorem \ref{thm:existence of delta-plt blow-up}.
	\end{proof}
	
	\subsection{\texorpdfstring{$3$}{}-dimensional terminal singularities}
	
	In this subsection, we prove Theorems \ref{thm:3fold-terminal-vol} and \ref{thm:delta-plt-blowup-lowdim}(2). First of all, by \cite[Proposition 14]{BL18a} and Corollary \ref{cor:delta-plt-fields} since local volumes and the existence of $\delta$-plt blow-ups are preserved under algebraically closed field extension and restriction, we may assume that the base field $\bk=\bC$ in this subsection. Our approach here is largely based on the classification of $3$-dimensional terminal singularities by Mori \cite{Mor85} as explained by Reid \cite{Rei87} (see also \cite[Chapter 5.3]{KM98}).
	
	We first give the following useful lemma on $3$-dimensional Gorenstein terminal singularities.
	
	\begin{lemma}[{\cite{Mor85}}]\label{lem:3-fold involution}
	Let $z\in Z$ be a $3$-dimensional Gorenstein terminal singularity. Then there exists a formal power series $f(z_2,z_3,z_4)$ such that $\widehat{\cO_{Z,z}}\cong \bC\llbracket z_1,z_2,z_3,z_4\rrbracket/(z_1^2+f(z_2,z_3,z_4))$.
	\end{lemma}
	
	The following theorem and table summarize the classification of $3$-dimensional terminal singularities from \cite{Mor85} and \cite[Section 6]{Rei87}. Here  $\bmu_r$ is the multiplicative group of $r$-th roots of unity.
	
	\begin{theorem} [\cite{Mor85}]\label{thm:terminal-classify}
	Let $x\in X$ be a $3$-dimensional terminal singularity. Let $r$ be the Gorenstein index of $x\in X$. Assume that $r\geq 2$. Then $x\in X$ is isomorphic to the $\bmu_r$-quotient of a $3$-dimensional Gorenstein terminal singularity $z\in Z$ as the index $1$ cover of $x\in X$. Moreover, there exist local analytic coordinates $(x_1,x_2,x_3, x_4)$ with a diagonal $\bmu_r$-action and a $\bmu_r$-semi-invariant formal power series $\phi(x_1,x_2,x_3,x_4)$ such that $\widehat{\cO_{Z,z}}$ is $\bmu_r$-equivariantly isomorphic to $\bC\llbracket x_1,x_2,x_3,x_4\rrbracket/(\phi)$. For a list of types of the $\bmu_r$-action and $\phi$, see Table \ref{table:terminal}.
	\end{theorem}

	\begin{table}[htbp!]\renewcommand{\arraystretch}{1.5}
\caption{$3$-dimensional terminal singularities (cf. \cite[p. 391]{Rei87})}\label{table:terminal}
	\begin{tabular}{|c|c|l|l|l|}
	\hline
	 Type &  $r$  &  $\bmu_r$-action & $\phi$  \\ \hline\hline
	(I) &  any  & $\frac{1}{r}(a,-a,1,0;0)$ &  $x_1 x_2+g(x_3^r,x_4)$ \\
	(II) & $4$ & $\frac{1}{4}(3,1,1,2;2)$ & $x_1^2+g(x_2,x_3,x_4)$ \\
	(III) & $3$ & $\frac{1}{3}(0,2,1,1;0)$ & $x_1^2 + x_2^3 +x_2 g(x_3,x_4) + h(x_3,x_4)$ \\
	(IV) & $2$ & $\frac{1}{2}(1,0,1,1;0)$ & $x_1^2+g(x_2,x_3,x_4)$ \\\hline
	\end{tabular}
	\end{table}
	
	In Table \ref{table:terminal}, $\frac{1}{r}(a_1,a_2,a_3,a_4;b)$ means that the generator $\zeta=e^{2\pi i/r}$ of $\bmu_r$ acts on the coordinates $(x_1,x_2,x_3,x_4)$ and the formal power series $\phi$ as $(x_i;\phi)\mapsto (\zeta^{a_i}x_i;\zeta^b\phi)$.
	\medskip
	
	Let $z\in Z$ be a $3$-dimensional Gorenstein terminal singularity. Let $R:=\cO_{Z,z}$, and $\fm$ the maximal ideal of $R$. 
	We denote by $\hat{z}\in \widehat{Z}=\Spec\, \widehat{\cO_{Z,z}}$ the completion of $z\in Z$. Then by Lemma \ref{lem:3-fold involution} we know that $\widehat{R}\cong \bC\llbracket z_1,z_2,z_3,z_4\rrbracket/(z_1^2+f(z_2,z_3,z_4))$.  Denote by $\tau: \widehat{Z}\to \widehat{Z}$ the involution given by $(z_1,z_2,z_3,z_4)\mapsto (-z_1,z_2,z_3,z_4)$. A valuation $v\in \Val_{Z,z}^{\circ}$ is called \emph{$\tau$-invariant} if $v=\phi_{\tau}(v)$ where $\phi_{\tau}:\Val_{Z,z}^{\circ}\to \Val_{Z,z}^{\circ}$ is the involution induced by $\tau$ according to Proposition \ref{prop:an-iso}. Let $\fD:=(f(z_2,z_3,z_4)=0)$ be an effective Cartier divisor on $\widehat{\bA^3}=\Spec \bC\llbracket z_2,z_3,z_4 \rrbracket$. Hence by taking quotient of the $\bmu_2$-action on $\hZ$ induced by $\tau$,  we obtain a finite crepant Galois morphism $\pi: \hZ\to (\widehat{\bA^3}, \frac{1}{2}\fD)$ given by $(z_1,z_2,z_3,z_4)\mapsto (z_2, z_3,z_4)$.
	
	\begin{lem}\label{lem:kc-involution} With the above notation, there is a $1$-to-$1$ correspondence between $\tau$-invariant Koll\'ar components $S$ over $z\in Z$ and Koll\'ar components $\bar{S}$ over $0\in (\widehat{\bA^3}, \frac{1}{2}\fD)$. Moreover, we have $\hvol_{Z,z}(\ord_S)=2\hvol_{(\widehat{\bA^3}, \frac{1}{2}\fD),0}(\ord_{\bar{S}})$.
	\end{lem}
	
	\begin{proof}
	This is a straightforward consequence of Proposition \ref{prop:finite deg formula}. To be more precise, if $S\subset Y\to Z$ is a $\tau$-invariant Koll\'ar component over $z\in Z$, we may take a formal neighborhood  $\hS\subset \hY$ of $S\subset Y$ where $\bmu_2$ acts. Then taking the $\bmu_2$-quotient of $(\hY, \hS)$ provides a Koll\'ar component $\bar{S}$ over $0\in (\widehat{\bA^3}, \frac{1}{2}\fD)$. Conversely, if $\bar{S}\subset \bar{Y}\to \widehat{\bA^3}$ is a Koll\'ar component over $0\in (\widehat{\bA^3}, \frac{1}{2}\fD)$, by taking Cartesian product we obtain $\hS\subset \hY:=\bar{Y}\times_{\widehat{\bA^3}}\hZ$, and $\hS$ is a Koll\'ar component over $\hat{z}\in \hZ$ by the Koll\'ar-Shokurov connectedness theorem as in \cite[Proof of Lemma 2.13]{LX16}. The equality on normalized volumes follows from similar arguments as \cite[Proof of Lemma 2.13]{LX16}.
	\end{proof}
	
	\begin{lem}\label{lem:xz-analytic} With the above notation, we have $\hvol(z,Z)=\inf_{S} \hvol_{Z,z}(\ord_S)$ where $S$ runs over all $\tau$-invariant Koll\'ar components over $z\in Z$.  
	\end{lem}
	
	\begin{proof}
	By Theorem \ref{thm:uniqueness}, there exists a unique valuation $v_*\in \Val_{Z,z}$ up to scaling that minimizes $\hvol_{Z,z}$. Since $A_Z(\phi_\tau(v_*))=A_Z(v_*)$ and $\hvol_{Z,z}(\phi_\tau(v_*))=\hvol_{Z,z}(v_*)$ by Proposition \ref{prop:an-iso}, we have $\phi_\tau(v_*)=v_*$ by Theorem \ref{thm:uniqueness}. Let $\fa_m:=\fa_m(v_*)$ be valuation ideals of $v_*$ for $m\in \bZ_{>0}$. By \cite[Proof of Theorem 27]{Liu18}, we know that 
	\[
	\lim_{m\to \infty} \lct(Z;\fa_m)^3\cdot \e(\fa_m)=\hvol_{Z,z}(v_*)=\hvol(z,Z).
	\]
	Let $\widehat{\fa_m}$ be the completion of $\fa_m$ in $\hR$. Hence $\widehat{\fa_m}$ is $\tau$-invariant since $\phi_\tau(v_*)=v_*$. By the $\bmu_2$-equivariant version of Lemma \ref{lem:kc-compute-lct} (see e.g. \cite[Lemma 4.8]{Zhu20}), 
	there exists a $\tau$-invariant Koll\'ar component $S_m$ over $z\in Z$ computing $\lct(Z;\fa_m)$. Thus the proof of Theorem \ref{thm:lx-kc-minimizing} implies that $	\lim_{m\to \infty} \hvol_{Z,z}(\ord_{S_m})=\hvol(z,Z)$.
	The proof is finished.
	\end{proof}

	\begin{proof}[Proof of Theorem \ref{thm:3fold-terminal-vol}]
	Let $\epsilon>0$ be a positive real number. Then it suffices to show that $\Vol_{3}^{\mathrm{term}}\cap (\epsilon, 27]$ is a finite set. For any $3$-dimensional terminal singularity $(x\in X)$ of Gorenstein index $r$, we may take its index $1$ cover $(z\in Z)$, and Theorem \ref{thm: finite degree formula} implies that 
	\[
	\hvol(z, Z)= r\cdot \hvol(x, X).
	\]
	If $\hvol(x,X)\geq \epsilon$, then we know that $r\leq \frac{27}{\epsilon}$. Hence it suffices to show that the following set 
	\[
		V_\epsilon:=\left\{ \nvol(z,Z)\mid 
		z\in Z\textrm{ is $3$-dimensional Gorenstein terminal}
		\textrm{ and }\hvol(z,Z)>\epsilon.
		\right\}
	\]
    is finite. In the rest of the proof, we will denote $z\in Z$ a $3$-dimensional Gorenstein terminal singularity satisfying $\hvol(z,Z)>\epsilon$.  Denote $(R,\fm):=(\cO_{Z,z},\fm_{Z,z})$. Let $\hz\in \hZ=\Spec\,\hR$ be the completion of $z\in Z$.
    By Lemma \ref{lem:3-fold involution} we know that $\hR\cong \bC\llbracket z_1,z_2,z_3,z_4\rrbracket/(z_1^2+f(z_2,z_3,z_4))$. Thus there exists a crepant double cover
    \[
    \pi: (\hz\in \hZ)\to (0 \in (\widehat{\bA^3}, \tfrac{1}{2} \fD)) \quad\textrm{ where } \fD=(f(z_2, z_3, z_4)=0).
    \]
    By Lemmata \ref{lem:kc-involution}, \ref{lem:xz-analytic}, and Theorem \ref{thm:truncation-an-bdry}(2) we know that 
    \begin{equation}\label{eq:xz-3fold}
    \hvol(z,Z)=\inf_{S} \hvol_{Z,z}(\ord_S) 
    = 2\inf_{\bar{S}}\hvol_{(\widehat{\bA^3},\frac{1}{2}\fD),0}(\ord_{\bar{S}})=2\hvol(0, \widehat{\bA^3},\tfrac{1}{2}\fD ),
    \end{equation}
    where $S$ runs through all $\tau$-invariant Koll\'ar components over $z\in Z$, and $\bar{S}$ runs through all Koll\'ar components over $0 \in (\widehat{\bA^3}, \tfrac{1}{2} \fD)$. By Theorem \ref{thm:truncation-an-bdry}(1), for $k\gg 1$ we have 
    $\hvol(0,\bA^3,\frac{1}{2}\fD^k)=\hvol(0,\widehat{\bA^3},\frac{1}{2}\fD)$ where $\fD^k$ is a $k$-th analytic truncation of $\fD$ on $\bA^3$. Hence for $k\gg 1$ we have 
    \[
    \epsilon<\hvol(z,Z)= 2\hvol(0,\bA^3,\tfrac{1}{2}\fD^k),
    \]
    and the right-hand-side belongs to a finite set by Corollary  \ref{cor:smooth}. Thus $V_\epsilon$ is a finite set.
	\end{proof}
	
	\begin{proof}[Proof of Theorem \ref{thm:delta-plt-blowup-lowdim}(2)]
	Fix a positive number $\epsilon>0$.
	Let $x\in X$ be a $3$-dimensional terminal singularity satisfying $\hvol(x,X)\geq \epsilon$. For simplicity, we assume that $\epsilon\in (0,1)$.  We will show that there exists a $\delta$-plt blow up of $x\in X$ where $\delta>0$ only depends on $\epsilon$. Let $\rho: (z\in Z)\to (x\in X)$ be the index $1$ cover of $K_X$. Denote by $r$ the Gorenstein index of $x\in X$. By 
	Theorems  \ref{thm:nvol is bounded by n^n} and \ref{thm: finite degree formula}, we have 
	\[
	27\geq \hvol(z,Z)=r\cdot \hvol(x,X)\geq r\epsilon.
	\]
	Thus we have $r\leq r_{\max}:= \lfloor\frac{27}{\epsilon}\rfloor$. 
	
	Let $\tau: \hZ\to \hZ$ be the analytic involution as before. 
	Let $\pi: (\hz\in \hZ) \to (0\in ( \widehat{\bA^3}, \frac{1}{2}\fD))$ be the double cover where $\fD= (f=0)$. By \eqref{eq:xz-3fold} we have 
	\[
	\hvol(0, \widehat{\bA^3}, \tfrac{1}{2}\fD)= \frac{1}{2}\hvol(z,Z)\geq \frac{r\epsilon}{2}\geq \frac{\epsilon}{2}.
	\]
	Let $f_k$ be a polynomial in $z_2, z_3, z_4$ of degree less than $k$ such that $f_k$ is a $k$-th analytic truncation of $f$. Denote by $\fD^k:= (f_k=0)\subset \bA^3$. 
	Then by Theorem \ref{thm:truncation-an-bdry}, we have that $\hvol(0, \bA^3, \frac{1}{2}\fD^k)=\hvol(0, \widehat{\bA^3}, \frac{1}{2}\fD)$ for any $k\gg 1$. By Theorem \ref{thm:estimate on lct}, there exists a positive constant $c$ such that 
	\[
	\lct(\bA^3,\tfrac{1}{2}\fD^k;\tfrac{1}{2}\fD^k)\geq c^{-1}\hvol(0, \bA^3, \tfrac{1}{2}\fD^k)\geq \frac{\epsilon}{2c}.
	\]
	By Lemma \ref{lem:lct of truncations} it is clear that $\lct(\bA^3,\tfrac{1}{2}\fD^k;\tfrac{1}{2}\fD^k)$ converges to $\lct(\widehat{\bA^3},\frac{1}{2}\fD;\frac{1}{2}\fD)$ as $k\to \infty$. Thus we have   $\lct(\widehat{\bA^3},\tfrac{1}{2}\fD;\tfrac{1}{2}\fD)\geq \frac{\epsilon}{2c}$. 
	Then by Theorem \ref{thm:truncation}, Proposition \ref{prop: truncation keep nv(v)}, and Theorem \ref{thm:truncation-an-bdry}, there exist $k_1, k_2\in \bZ_{>0}$ depending only on $\epsilon$ such that for any $k\geq k_3:=\max\{k_1,k_2\}$ and any $\delta\in \bR_{\geq 0}$ we have
	\begin{enumerate}[label=(\alph*)]
	    \item $\hvol(0, \bA^3, \tfrac{1}{2}\fD^{k})=\hvol(0, \widehat{\bA^3}, \tfrac{1}{2}\fD)\geq \frac{\epsilon}{2}$, and
	    \item any $\delta$-Koll\'ar component $S$ of $0\in (\bA^3, \frac{1}{2}\fD^{k})$ satisfying $\hvol_{(\bA^3, \frac{1}{2}\fD^{k}),0}(\ord_S)\leq 28r_{\max}$ corresponds to  a $\delta$-Koll\'ar component $\hS$ of $0\in (\widehat{\bA^3}, \frac{1}{2}\fD)$ by taking completion.
	\end{enumerate}

	Next, we show the existence of a  $\delta$-plt blow-up of $x\in X$ for $\delta=\delta(\epsilon)>0$. Firstly, we consider the case where $r=1$, i.e. $x\in X$ is Gorenstein. Since $k_3$ only depends on $\epsilon$, we know that $0\in (\bA^3, \frac{1}{2}\fD^{k_3})$ belongs to a bounded $\bQ$-Gorenstein family of klt singularities. By Theorem \ref{thm: family kc with nv bdd}, there exists $\delta_1=\delta_1(\epsilon)>0$ such that $0\in (\bA^3, \frac{1}{2}\fD^{k_3})$ admits a $\delta_1$-Koll\'ar component $S$ satisfying $\hvol_{(\bA^3, \frac{1}{2}\fD^{k}),0}(\ord_S)\leq 28$. Hence by (b), we have a $\delta_1$-Koll\'ar component $\hS$ over $0\in (\widehat{\bA^3}, \frac{1}{2}\fD)$. Hence by pulling back $\hS$ under $\pi$ then push-forward under the completion map $\iota: \hX\to X$, we obtain a $\delta_1$-Koll\'ar component $\tau_*\pi^*\hS$ over $x\in X$ by \cite[Proposition 5.20]{KM98}.
	
	Next, we consider the case where $r\geq 2$ and the covering morphism $\rho:(z\in Z)\to (x\in X)$  has type (II), (III), or (IV) in Table \ref{table:terminal}. We may assume that the  coordinates $(z_i)$ from Lemma \ref{lem:3-fold involution} coincide with the coordinates $(x_i)$ from Theorem \ref{thm:terminal-classify} and Table \ref{table:terminal}. Denote by $G:= \bmu_r$. Then by restricting to the last three coordinates, the $G$-action on $\hZ$ induces a $G$-action on $\bA^3$ such that $\fD$ is $G$-invariant and $\pi$ is $G$-equivariant. In particular, $\fD^{k_3}$ is also $G$-invariant. Let $w\in (W, \Delta_W)$ be the crepant $G$-quotient of $0\in (\bA^3, \frac{1}{2}\fD^{k_3})$. Since $0\in (\bA^3, \frac{1}{2}\fD^{k_3})$ belongs to a bounded $\bQ$-Gorenstein family of klt singularities and the $G$-action on $\bA^3$ has finitely many choices, we know that $w\in (W,\Delta_W)$ also belongs to a bounded $\bQ$-Gorenstein family of klt singularities. Then by Theorem \ref{thm: family kc with nv bdd}, there exists $\delta_2=\delta_2(\epsilon)>0$ such that $w\in (W,\Delta_W)$ admits a $\delta_2$-Koll\'ar component $S_W$ satisfying $\hvol_{(W,\Delta_W),w}(\ord_{S_W})\leq 28$. Denote by $S$ the pull-back of $S_W$ under the $G$-quotient morphism $\bA^3\to W$. Hence by \cite[Proposition 5.20]{KM98} and Proposition \ref{prop:finite deg formula}, $S$ is a $G$-equivariant $\delta_2$-Koll\'ar component over $0\in (\bA^3, \frac{1}{2}\fD^{k_3})$ satisfying $\hvol_{(\bA^3, \frac{1}{2}\fD^{k}),0}(\ord_S)\leq 28r_{\max}$. Hence by (b), we have a $G$-equivariant $\delta_2$-Koll\'ar component $\hS$ over $0\in (\widehat{\bA^3},\frac{1}{2}\fD)$. By pulling back $\hS$ under $\pi$ then push-forward under the completion map $\iota: \hZ\to Z$, we obtain a $G$-equivariant $\delta_2$-Koll\'ar component $\iota_*\pi^*\hS$ over $z\in Z$. Then taking the $G$-quotient of $\iota_*\pi^*\hS$ and applying \cite[Proposition 5.20]{KM98} again, we obtain a $(\delta_2/r_{\max})$-Koll\'ar component over $x\in X$. 
	
	Finally, we consider the case where $r\geq 2$ and the covering morphism $\rho$ has type (I) in Table \ref{table:terminal}. Let $z_1=\frac{x_1+x_2}{2}$, $z_2=\frac{x_1-x_2}{2}$, $z_3=x_3$, and $z_4=x_4$. Then the local analytic equation of $\hz\in \hZ$ is given by $(z_1^2 - z_2^2 + h(z_3, z_4)=0)$ where $h$ is some formal power series in $(z_3,z_4)$. Let $h_k$ be the $k$-th analytic truncation of $h$ as a polynomial in $(z_3,z_4)$ of degree less than $k$.
	Let $0\in Z^k$ be the hypersurface singularity $(z_1^2 -z_2^2 + h_k (z_3, z_4)=0)$ in $\bA^4$. Then clearly the $G$-action on $\hZ$ induces a $G$-action on $Z^k$. Let $\tau_k: Z^k\to Z^k$ be the involution given by $\tau_k (z_1,z_2,z_3,z_4)= (-z_1,z_2,z_3,z_4)$. Then $G$ and $\{\id, \tau_k\}$ generate a finite subgroup $H<\Aut(0, Z^k)$ of size $|H|=2r$. 
	Let $w\in (W,\Delta_W)$ be the crepant $H$-quotient of $0\in Z^{k_3}$. Since $0\in Z^{k_3}$ belongs to a bounded $\bQ$-Gorenstein family of klt singularities and the $H$-action on $\bA^4$ has finitely many choices, we know that $w\in (W,\Delta_W)$ also belongs to a bounded $\bQ$-Gorenstein family of klt singularities. Thus Theorem \ref{thm: family kc with nv bdd} implies that there exists $\delta_3=\delta_3(\epsilon)>0$ such that $w\in (W,\Delta_W)$ admits a $\delta_3$-Koll\'ar component $S_W$ satisfying $\hvol_{(W,\Delta_W),w}(\ord_{S_W})\leq 28$. Denote by $\tS^{k_3}$ the pull-back of $S_W$ under the $H$-quotient morphism $Z^{k_3}\to W$. So \cite[Proposition 5.20]{KM98} and Proposition \ref{prop:finite deg formula} imply that $\tS^{k_3}$ is an $H$-equivariant $\delta_3$-Koll\'ar component over $0\in Z^{k_3}$ such that $\hvol_{Z^{k_3}, 0}(\ord_{\tS^{k_3}})\leq 56r_{\max}$. 
	Taking quotient of the involution $\tau_{k_3}$, we obtain a crepant covering morphism  $\pi_{k_3}: (0\in Z^{k_3})\to (0\in(\bA^3, \frac{1}{2}\fD^{k_3}))$. Hence \cite[Proposition 5.20]{KM98} and Proposition \ref{prop:finite deg formula} imply that $S:=(\pi_{k_3})_* \tS^{k_3}$ is a $(\delta_3/2)$-Koll\'ar component over $0\in(\bA^3, \frac{1}{2}\fD^{k_3})$ satisfying $\hvol_{(\bA^3, \frac{1}{2}\fD^{k_3}),0}(\ord_S)\leq 28r_{\max}$. By (b), we have a $(\delta_3/2)$-Koll\'ar component $\hS$ over $0\in (\widehat{\bA^3}, \frac{1}{2}\fD)$. Pulling back $\hS$ under $\pi$ yields an $H$-equivariant (hence $G$-equivariant) $(\delta_3/2)$-Koll\'ar component $\pi^*\hS$ over $\hz\in \hZ$. Thus by taking push-forward $\pi^*\hS$ under $\iota$ and then quotient out by $G$, we obtain a $(\delta_3/(2r_{\max}))$-Koll\'ar component over $x\in X$. Therefore, the  proof is finished by taking $\delta:=\min\{\delta_1, \frac{\delta_2}{r_{\max}}, \frac{\delta_3}{2r_{\max}}\}$. 
	\end{proof}

	\section{Discussions}
	In this section, we discuss some topics related to our main results, and ask several questions.

	\subsection{Relations among three classes of singularities}
	
    When $(x\in X^{\rm an})\in (B\subset\cX^{\rm an}\to B)$ (``bounded'') and under some mild assumptions, we showed that the three classes of singularities \textcircled{a}, \textcircled{b}, \textcircled{c} in Figure \ref{abc classes of singularities} are equivalent to each other, see Theorems \ref{thm:bddness for nvol} (\textcircled{b} $\Rightarrow$ \textcircled{a}), \ref{thm:estimate on lct} (\textcircled{a} $\Rightarrow$ \textcircled{c}), and \ref{thm: lct positive lower bound implies existence of delta-plt blow-up} (\textcircled{c} $\Rightarrow$ \textcircled{a}). In this subsection, we will discuss the relations among these three classes of singularities without the assumption ``$(x\in X^{\rm an})\in (B\subset\cX^{\rm an}\to B)$''. Note that in this general setting, Theorem \ref{thm:bddness for nvol} (\textcircled{b} $\Rightarrow$ \textcircled{a}) holds, and Conjecture \ref{conj:existence of delta-plt blow-up} is about the implication ``\textcircled{a} $\Rightarrow$ \textcircled{b}''. 
    

 
     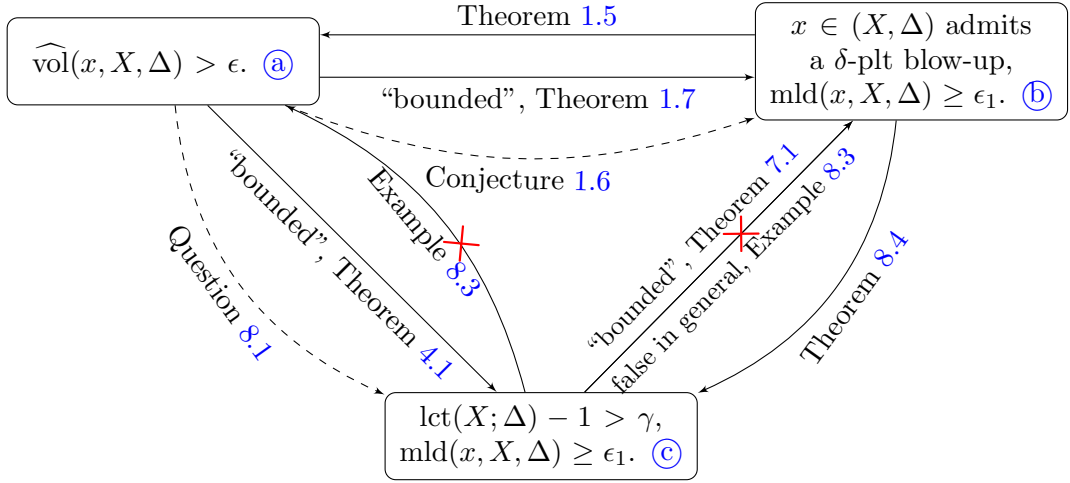
\begin{figure}[ht]     
        \centering
    \tikzstyle{block} = [rectangle, draw, text width=10em, text centered, rounded corners, minimum height=3em]
	\tikzstyle{line} = [draw, -latex']
	
	\begin{tikzpicture}[auto]
		\node [block] (nv) {$\widehat{\mathrm{vol}}(x,X,\Delta)>\epsilon$. \textcolor{blue}{\textcircled{a}}};
		
		\node [block, below right of=nv, node distance=200pt] (lct) {$\mathrm{lct}(X;\Delta)-1>\gamma$, $\mld(x,X,\Delta)\ge\epsilon_1$. \textcolor{blue}{\textcircled{c}}};
		
		\node [block, right of=nv, node distance=280pt] (dplt) {$x\in(X,\Delta)$ admits a $\delta$-plt blow-up, $\mld(x,X,\Delta)\ge\epsilon_1$. \textcolor{blue}{\textcircled{b}}};
		
		\draw [line, transform canvas={shift={(0,0)}}] (nv) --node[below, sloped]{``bounded'', Theorem \ref{thm:estimate on lct}} (lct);
		
			\path[line] (lct) edge [bend right=25] node[below,sloped] {Example \ref{ex: lct does not imply nv bdd}}node[right, red, sloped] {\huge $\times$}  (nv);

		
		\path [line, transform canvas={shift={(0,10pt)}}] (dplt) --node[above, sloped]{Theorem \ref{thm:bddness for nvol}}(nv); 
		
		\path [line, transform canvas={shift={(0,0)}}] (lct) --node[above, xshift=-5pt,sloped]{\small ``bounded'', Theorem \ref{thm: lct positive lower bound implies existence of delta-plt blow-up}} (dplt);
		
	    \path [line, transform canvas={shift={(0,0)}}] (lct) --node[below, sloped]{\small false in general, Example \ref{ex: lct does not imply nv bdd}}node[right, red, sloped] {\huge $\times$} (dplt) ;

		
		\path[line] (dplt) edge [bend left=30] node[below, sloped] {Theorem \ref{thm: delta plt imply lct lower bound}}  (lct);
		
		\path[line, dashed] (nv) edge [bend right=30] node[below,sloped] {Question \ref{conj: nv bdd below implies lct bdd below}}  (lct);
		

	    \path[line, transform canvas={shift={(0,-6pt)}}] (nv) --node[below, sloped]{``bounded'', Theorem \ref{thm:existence of delta-plt blow-up}} (dplt);

	\path[line, dashed] (nv) edge [bend right=20]node[below,sloped] {Conjecture \ref{conj:existence of delta-plt blow-up}}  (dplt);
	
		
	\end{tikzpicture}
        \caption{Three classes of singularities \textcircled{a}, \textcircled{b}, and \textcircled{c}.}
       \label{abc classes of singularities}
    \end{figure}

	We expect that Theorem \ref{thm:estimate on lct} holds (\textcircled{a} $\Rightarrow$ \textcircled{c}) in this general setting. 
	
	\begin{question}\label{conj: nv bdd below implies lct bdd below}
	Let $n$ be a positive integer. Then does there exist a positive real number $c(n)$ depends only on $n$ satisfying the following statement?
	
	Let $x\in (X,\Delta)$ be an $n$-dimensional klt $\Qq$-Gorenstein singularity. Then
		\begin{equation*}
		c(n)\cdot\lct(X,\Delta;\Delta)\geq  \nvol(x,X,\Delta). 
		\end{equation*}
	\end{question}

    \begin{remark}\label{rmk:sharp_constant}
        One might also ask for a sharp value $c_{\min}(n)$ in Question \ref{conj: nv bdd below implies lct bdd below}. We guess that $c_{\min}(n)=n^n$. When $n=2$, it is not hard to show that $c_{\min}(2)\le 8$, see Theorem \ref{thm:surface lct}. 
	\end{remark}
	

Example \ref{ex: lct does not imply nv bdd} shows that the implication ``\textcircled{c} $\Rightarrow$ \textcircled{a}'' does not hold when $x\in X$ is not analytically bounded. Thus the implication ``\textcircled{c} $\Rightarrow$ \textcircled{b}''  does not hold either by Theorem \ref{thm:bddness for nvol}. 
	
	\begin{example}\label{ex: lct does not imply nv bdd} 
	
	Let $m$ be a positive integer. Consider the surface klt singularity $(x\in (X,\Delta=\frac{1}{2}(L_1+L_2)))$, where  $(x\in X\cong (o\in\bA^2)/\bmu_{m+1})$ is the $A_m$-singularity, and $L_1$, $L_2$ are the images of two coordinate lines in $\bA^2$. Then $\lct(X,\Delta;\Delta)=1$, $\mld(x,X,\Delta)=\frac{1}{2}$, and $\nvol(x,X,\Delta)<\nvol(x,X,0)=\frac{4}{m+1}\to 0$ when $m\to +\infty$. 
	\end{example}
	
	If we assume that all the (non-zero) coefficients of $\Delta$ have a positive lower bound, then Theorem \ref{thm: delta plt imply lct lower bound} together with Conjecture \ref{conj:existence of delta-plt blow-up} would imply that Question \ref{conj: nv bdd below implies lct bdd below} has an affirmative answer immediately. Theorem \ref{thm: delta plt imply lct lower bound} is a consequence of Birkar's proof of Birkar--Borisov--Alexeev--Borisov Theorem \cite[Theorems 1.1 and 1.6]{Bir16}, and it is embedded in \cite{HLS19}.
	
	\begin{theorem}\label{thm: delta plt imply lct lower bound}
		Let $n\ge 2$ be a positive integer and $\eta,\delta,\epsilon_1$ positive real numbers. Then there exists a positive real number $\gamma$ which only depends on $n,\eta,\delta$, and $\epsilon_1$ satisfying the following.
		
		Let $x\in (X,\Delta=\sum_{i=1}^m a_i\Delta_i)$ be an $n$-dimensional klt $\Qq$-Gorenstein singularity, such that
		\begin{enumerate}
		\item $a_i>\eta$,
		\item each $\Delta_i\ge0$ is a Weil divisor, 
			\item $x\in(X,\Delta)$ admits a $\delta$-plt blow-up, and
			\item $\mld(x,X,\Delta)\geq \epsilon_1$,
		\end{enumerate}
		then $\lct(X,\Delta;\Delta)\ge \gamma$. 
	\end{theorem}
	\begin{proof}
	
	By \cite[Lemma 3.13]{HLS19}, there exists a $\Qq$-factorial weak $\delta$-plt blow-up $f:Y\to X$ of $x\in(X,\Delta)$, that is, $f$ is a birational morphism with the exceptional prime divisor $E$ such that
		\begin{itemize}
		\item $(Y,f_{*}^{-1}\Delta+E)$ is $\Qq$-factorial $\delta$-plt near $E$, 
		\item $-E$ is nef over $X$, 
		\item $-(K_Y+f_{*}^{-1}\Delta+E)|_{E}$ is big, and
		\item $f^{-1}(x)=\Supp E$. 
	\end{itemize}
	 By \cite[Proposition 4.3]{HLS19}, there exists a positive real number $M$ which only depends on $n$ and $\delta$, such that $A_{X}(E)\le M$. By assumption, $A_{(X,\Delta)}(E)\ge \epsilon_1$. Thus there exists a positive real number $\gamma_1$ which only depends on $n$, $\delta$ and $\epsilon_1$, such that $A_{(X,(1+\gamma_1)\Delta)}(E)\ge 0$. 
	 
	 By \cite[Proposition 7.7]{HLS19}, there exists a positive real number $\gamma_2$ which only depends on $n$, $\eta$ and $\delta$, such that $(Y,(1+\gamma_2)f_{*}^{-1}\Delta+E)$ is plt near $E$.  
	
	Let $\gamma\coloneqq\min\{\gamma_1,\gamma_2\}$. Then
	$x\in (X,(1+\gamma)\Delta)$ is lc, and $\lct(X,\Delta;\Delta)\ge \gamma$.
	\end{proof}
	\begin{remark}
	Example \ref{eg:coeff no lower bound} also indicates that the assumption ``$a_i>\eta$'' in Theorem \ref{thm: delta plt imply lct lower bound} is necessary. 
	\end{remark}

\subsection{Boundedness of singularities up to a special degeneration}\label{sec:spdeg}

\begin{defn}
Let $(X,\Delta)$ be a klt pair with $x\in X$ a closed point. A \emph{special test configuration}  of $x\in (X,\Delta)$ consists of the following data:
\begin{itemize}
    \item a normal variety $\cX$ and an effective $\bR$-divisor $\Delta_{\tc}$ on $\cX$ such that $K_{\cX}+\Delta_{\tc}$ is $\bR$-Cartier;
    \item a flat morphism $\pi: (\cX,\Supp(\Delta_{\tc}))\to \bA^1$ and a section $\sigma: \bA^1\to \cX$ of $\pi$;
    \item a $\bG_m$-action on $(\cX,\Delta_{\tc})$ such that both $\pi$ and $\sigma$ are $\bG_m$-equivariant with respect to the standard $\bG_m$-action on $\bA^1$ by scalar multiplication;
    \item $\sigma(\bA^1\setminus 0)\subset (\cX\setminus \cX_0, \Delta_{\tc}|_{\cX\setminus\cX_0})$ is $\bG_m$-equivariantly isomorphic to $(x\in (X,\Delta))\times (\bA^1\setminus 0)$;
    \item $(\cX,\cX_0+\Delta_{\tc})$ is plt.
\end{itemize}
We call the central fiber $(\sigma(0)\in (\cX_0, \Delta_{\tc,0}))$ of the  special test configuration a \emph{special degeneration} of $(x\in (X,\Delta))$. By adjunction, $(\sigma(0)\in (\cX_0, \Delta_{\tc,0}))$ is also a klt singularity. 
\end{defn}

\begin{defn}
A set of klt singularities $\cP$ is said to be \emph{log bounded up to special degeneration} if there is a log bounded set $\cC$ of projective pairs, such that the following holds.

For any klt singularity $x\in (X,\Delta)$ in $\cP$, there exist  a special degeneration $x_0\in (X_0, \Delta_0)$ of $x\in (X,\Delta)$, a pair $(Y, B)\in \cC$ together with a closed point $y\in Y$, and open neighborhoods $U$ and $V$ of $x_0\in X_0$ and $y\in Y$ respectively, such that $(x_0\in (U, \Supp(\Delta_0)|_{U}))\cong (y\in (V, \Supp(B)|_{V}))$. 
\end{defn}

The following theorem from \cite{HLM19} shows that $\epsilon$-lc singularities admitting $\delta$-plt blow-ups with positive lower bounds on boundary coefficients are log bounded up to special degeneration. We expect that the $\bQ$-Gorenstein assumption from Theorem \ref{thm:hlm} can be dropped.

\begin{theorem}[{\cite[Theorem 4.1 and its proof]{HLM19}}]\label{thm:hlm}
Let $n$ be a positive integer, and $\epsilon_1,\delta,\eta$ three positive numbers. Then the set of $n$-dimensional $\epsilon_1$-lc $\bQ$-Gorenstein singularities $x\in (X,\Delta)$ admitting a $\delta$-plt blow-up, and  coefficients of $\Delta$ are at least $\eta$ is log bounded up to special degeneration.
\end{theorem}

We ask the following conjecture.

\begin{conj}\label{conj:bddspecial}
Let $n$ be a positive integer, and $\epsilon,\eta$ two positive numbers. Then the set of $n$-dimensional klt singularities $x\in (X,\Delta)$ satisfying that $\hvol(x,X,\Delta)\geq \epsilon$ and coefficients of $\Delta$ are at least $\eta$ is log bounded up to special degeneration.
\end{conj}

By Theorem \ref{thm:hlm}, Conjecture \ref{conj:bddspecial} follows from Conjecture \ref{conj:existence of delta-plt blow-up} for  $\bQ$-Gorenstein singularities.

	One can ask about the converses of Theorem \ref{thm:hlm} and Conjecture \ref{conj:bddspecial}. Assume that  coefficients of $\Delta$ belong to a finite set $I$. We expect that the converse of Theorem \ref{thm:hlm} holds under this assumption, although we do not have a proof at the moment. Meanwhile, using the lower semi-continuity of local volumes \cite{BL18a} and the constructibility \cite{Xu19} (see also Theorem \ref{thm:constructibility R-div}), we can show that the converse of Conjecture \ref{conj:bddspecial} is true under this assumption.

	The following example provides one more prototype for Conjecture \ref{conj:existence of delta-plt blow-up} and Conjecture \ref{conj:ACC for nvol}.
	
	\begin{example}\label{eg:Cartier cone sing}
		Let $n$ be a positive integer and $I\subset \bQ\cap [0,1]$ a finite set. Let $(V,\Delta)$ be an $n$-dimensional $K$-semistable log Fano pair, such that all the coefficients of $\Delta$ belong to $I$. 
		Consider the affine cone $X\coloneqq \Spec \oplus_{m=0}^\infty H^0(V,-mr(K_V+\Delta))$ over $(V,\Delta)$ with the ample Cartier polarization $-r(K_V+\Delta)$, where $r\in\bQ_{> 0}$. Let $D\coloneqq C(\Delta)$ be the cone divisor, and $o\in X$ the vertex. Then by \cite[Theorem 4.5]{LX16}, the canonical valuation $\ord_S$ obtained by blowing up the vertex $o\in X$, $\mu:Y\to X$, minimizes $\nvol_{(X,D),o}$ on $\Val_{X,o}$. Let $D_S$ be the different divisor of $(Y,S+\mu_{*}^{-1}D)$, then $(S,D_S)\cong (V,\Delta)$. We claim if $\nvol(o,X,D)>\epsilon$ for some $\epsilon>0$, then there exists an integer $N$ which only depends on $n,I$ and $\epsilon$, such that $N(K_S+D_S)$ is Cartier and $\nvol(o,X,D)$ belongs to a finite set. In particular, $(S,D_S)$ is $\frac{1}{N}$-lc, and $o\in X$ admits a $\frac{1}{N}$-plt blow up. Hence Conjectures \ref{conj:ACC for nvol} and  \ref{conj:existence of delta-plt blow-up} hold for those cone singularities.
		
		Now we show the claim. By \cite[Theorem 4.5]{LX16} and \cite[Proposition 3.14(4)]{Kol13}, we have 
		\begin{equation}\label{eqn:nvol_Cartier_cone}
		\nvol(o,X,D)=\nvol_{(X,D),o}(\ord_S)=\frac{(-(K_V+\Delta))^n}{r}>\epsilon.
		\end{equation}
		Since $-r(K_V+\Delta)$ is an ample Cartier divisor, by the length of extremal rays, we know that $r\geq\frac{1}{2n}$, so $(-(K_V+\Delta))^n>\frac{\epsilon}{2n}$. On the other hand, by Theorem \ref{thm: kss equiv delta larger than or equal to 1} and \cite[Theorem C]{Blu18b}, we know that $\alpha(V,\Delta)\geq \frac{1}{n+1}$, so $\alpha(V,\Delta)^n(-(K_V+\Delta))^n>\frac{\epsilon}{2n(n+1)^{n}}$. 
		Thus by \cite[Corollary 6.14]{LLX18}, $(V,\Supp\Delta)$ is log bounded. Now the existence of $N$ follows from \cite[Lemma 2.24]{Bir19}. It follows that $r$ belongs to a finite set. By \cite[Lemma 3.26]{HLS19}, $\nvol(o,X,D)=\frac{(-(K_V+\Delta))^n}{r}$ belongs to a finite set.
	\end{example}
	
    If $(V,\Delta)$ is as in Example \ref{eg:Cartier cone sing} and $X\coloneqq \Spec\oplus_{m=0}^\infty H^0(X,mL)$ is an orbifold cone over $(V,\Delta)$, where $o\in X$ is the vertex and the polarization $L\sim_\bQ -r(K_V+\Delta)$ is an integral Weil divisor, then \eqref{eqn:nvol_Cartier_cone} is true. So it is not difficult to see that the discreteness of $\{\nvol(o,X,D)\}$ away from $0$ follows from an affirmative answer to the following question (see also \cite[Example 4.4]{LX19}).     Indeed, Conjecture \ref{conj:ACC for nvol} for general klt singularities is not far from the case of  orbifold cone singularities if we assume the Stable Degeneration Conjecture \cite[Conjecture 1.2]{LX18}.
	 
	\begin{question}\label{que:weil-index}
	    Let $n$ be a positive integer and $I\subset \bQ\cap [0,1]$ a finite set. Let $(V,\Delta)$ be an $n$-dimensional $K$-semistable log Fano pair, such that all the coefficients of $\Delta$ belong to $I$. Consider the Fano-Weil index $q(X,\Delta)$ of $(X,\Delta)$ defined as 
	    \[
	        q(X,\Delta)\coloneqq \max\{q\in\bQ_{>0}\mid \textrm{there exists an integral Weil divisor } L\sim_\bQ q^{-1}(-K_V-\Delta)\},
        \]
	    then does there exist a positive real number $M$ depending only on $n$ and $I$ such that $q(X,\Delta)\le M$?
    \end{question}

	In view of Example \ref{eg:Cartier cone sing}, we also recall the following folklore question.
	
	\begin{question}\label{q:bdd alpha}
		Let $n$ be a positive integer. For any $n$-dimensional klt singularity $x\in(X,\Delta)$, is there a sequence of Koll\'ar components $\{S_k\}$ of $x\in(X,\Delta)$ with 
		$\lim\limits_{k\to+\infty}\nvol_{(X,\Delta),x}(\ord_{S_k})=\nvol(x,X,\Delta)$, such that
		\[
		    \limsup_{k\to\infty}\alpha(S_k,\Delta_{S_k})\geq \frac{1}{n},\quad
		\text{ or }	\quad	    \limsup_{k\to\infty}\delta(S_k,\Delta_{S_k})\geq 1\textrm{?}
		\]
	\end{question}

		\appendix
	
	\section{An elementary approach toward the case of surface singularities}
	
	We include in this appendix an elementary proof for Conjecture \ref{conj:ACC for nvol} in the $2$-dimensional case and an estimate of $c_{\min}(2)$ in Question \ref{conj: nv bdd below implies lct bdd below}. \footnote{This appendix will not be contained in the published version of this article.}
	
	We will need the following explicit criterion for log K-semistability when $X=\bP^1$.
	
	\begin{proposition}[{\cite[Theorem 3]{Li15}, \cite[Example 6.6]{Fuj19a}}]\label{prop:K-ss condition}
		Consider the log Fano pair $(X,\Delta)=(\bP^1,\sum_{i=1}^{m}d_ip_i)$, where $1\geq d_1\geq\ldots\geq d_m>0$ are rational numbers and $p_i$ are distinct closed points. Then $(X,\Delta)$ is $K$-semistable if and only if $\sum_{i=2}^md_i\geq d_1$.
	\end{proposition}

	The following result seems to be well known for experts, for reader's convenience we include its proof here.
	
	\begin{proposition}\label{prop:weighted blow-up}
		Let $x\in (X,\Delta)$ be a klt singularity, where $X$ is a smooth surface. Then any Koll\'ar component of $x\in (X,\Delta)$ can be obtained by a weighted blow-up under suitable coordinates.
	\end{proposition}
	
	\begin{proof}
		Suppose that $\mu:(Y,S)\to(X,x)$ is a Koll\'ar component of $(X,\Delta)$. By \cite[Theorem 1]{Kaw17}, it suffices to prove that $S$ computes the mld of some log canonical pair $(X,\Delta')$.
		
		By assumption, $-(K_Y+\mu^{-1}_*\Delta+S) \sim_{X,\bR} -A_{(X,\Delta)}(S)\cdot S$ is $\mu$-ample. There exists an $\Rr$-Cartier $\Rr$-divisor $G\ge0$, such that $(Y,\mu^{-1}_*\Delta+S+G)$ is lc, and $K_{Y}+\mu^{-1}_*\Delta+S+G\sim_{\Rr,X}0$. Then $S$ computes the mld of $(X,\Delta'\coloneqq \Delta+\mu_*G)$. 
	\end{proof}
	
	For simplicity, we only state the results for $o\in\bA^2$ with a $\bQ$ boundary. 
	First, we prove the following version of Theorem \ref{thm:surface ACC}.
	
	\begin{theorem}\label{thm:ACC for Q coeff}
		Let $I\subset\bQ\cap [0,1]$ be a set which satisfies the DCC, then the set
		\[
		\{\nvol(o,\bA^2,\Delta=\sum_{i=1}^ma_i\Delta_i) \mid o\in(\bA^2,\Delta) \text{ is klt, } a_i\in I \text{ for any } i\}
		\] 
		satisfies the ACC, where $o\in \bA^2$ is the origin. 
	\end{theorem}
	
	\begin{proof}		
		Suppose to the contrary that there exists a sequence of boundaries $\Delta\supn=\sum_{i=1}^{m_n}a_i\supn \Delta_i\supn$ on $\bA^2$, where $a_i^{(n)}\in I$ for any $i$ and $n$, such that $\nvol(o,\bA^2,\Delta\supn)$ is strictly increasing. In particular we may assume there exists some $\epsilon>0$ such that \begin{equation}\label{eqn:nvol bdd}
			\nvol(o,\bA^2,\Delta\supn)>\epsilon.
		\end{equation}
		Moreover since $I$ satisfies DCC we may also assume $
			a_i^{(n)}>\epsilon$.
%
		By \cite[Theorme 4.14]{LLX18}, for every $n$, there is a plt blow-up $\mu_n:(Y_n,S_n)\to (\bA^2,o)$ of $o\in(\bA^2,\Delta\supn)$ such that $\nvol_{o,(\bA^2,\Delta\supn)}(\ord_{S_n})=\nvol(o,\bA^2,\Delta^{(n)})$. 
		By Proposition \ref{prop:weighted blow-up}, $\mu_n$ is obtained by a weighted blow-up of $o$ with weight $(m_1\supn,m_2\supn)$, where $m_1\supn$ and $m_2\supn$ are coprime. In particular $S_n\cong \bP^1$ for every $n$.
		
		Define $\Delta_{S_n}\coloneqq \Diff_{S_n}(\Delta\supn)$, then by Theorem \ref{thm:minimizer is K-ss} $(S_n,\Delta_{S_n})$ is a $K$-semistable log Fano pair. By adjunction we have
		\begin{equation}
			\begin{aligned}
				& (K_{Y_n}+(\mu_n^{-1})_*\Delta\supn+S_n)|_{S_n}=K_{S_n}+\Delta_{S_n}\\
				&=K_{\bP^1}+
				\frac{m_1\supn-1+\sum_i r_{1,i}\supn a_i\supn}{m_1\supn}p_1\supn+
				\frac{m_2\supn-1+\sum_i r_{2,i}\supn a_i\supn}{m_2\supn}p_2\supn +\sum_{j>2}\sum_i r_{j,i}\supn a_i\supn p_j\supn,
			\end{aligned}
		\end{equation}
		where $r_{j,i}$ are nonnegative integers. 
		Let $r_i\supn\coloneqq \sum_{j>2}r_{j,i}\supn$ for $j>2$. Denote $\Delta_{S_n}=\sum_j d_j p_j\supn$ with distinct $p_j\supn\in S_n$, then $d_1=\frac{m_1\supn-1+\sum_i r_{1,i}\supn a_i\supn}{m_1\supn}$, $d_2=\frac{m_2\supn-1+\sum_i r_{2,i}\supn a_i\supn}{m_2\supn}$ and $\sum r_i\supn a_i\supn=\sum_{j>2}d_j$. Note that by \cite[Lemma 2.7]{LX16},
		\begin{equation}\label{eqn:formula for nvol}
			\begin{aligned}
				\nvol(o,\bA^2,\Delta\supn)=\nvol(\ord_{S_n})=
				m_1\supn m_2\supn\cdot \deg(-(K_{S_n}+\Delta_{S_n}))^2
				= m_1\supn m_2\supn (2-\sum d_i\supn)^2.
			\end{aligned}
		\end{equation} 
		
		We have the following cases.
		
		\noindent	{\bf Case a)}. $\sum r_i\supn a_i\supn\neq 0$ for infinitey many $n$. Then by the assumption that $a_i^{(n)}>\epsilon$, after passing to a sbusequence, we may assume that for all $n$,
		\begin{equation}\label{eqn:correction bdd}
			\sum r_i\supn a_i\supn>\epsilon.
		\end{equation}
		
		{\bf a1)} If there exists $M>0$ such that $m_1\supn m_2\supn \leq M$, in particular $\{m_1\supn m_2\supn\}$ is a finite set, so after passing to a subsequence we may assume that it is constant. Then by \eqref{eqn:formula for nvol} and our assumption, $\{2-\sum d_j\supn\}$ is strictly increasing. But by \cite[Proposition 3.4.1]{HMX14}, 
		$\{\sum d_j\supn\}$ satisfies the DCC, which is a contradiction.
		
		{\bf a2)} If there exists $M>0$ such that $m_1\supn\leq M$ but $m_2\supn\to+\infty$ as $n\to+\infty$. Since $(S_n,\Delta_{S_n})$ is $K$-semistable, by Proposition \ref{prop:K-ss condition} we have $
		\sum_{j\neq 2} d_j\supn\geq
		d_2\supn$, 
		whence 
		\begin{equation*}
			\frac{1-\sum_i r_{1,i}\supn a_i\supn}{m_1\supn}-
			\sum_i r_i\supn a_i\supn\leq
			\frac{1-\sum_i r_{2,i}\supn a_i\supn}{m_2\supn}\leq \frac{1}{m_2\supn}.
		\end{equation*}
		Thus by \eqref{eqn:formula for nvol} 
		\begin{equation*}
			\nvol(o,\bA^2,D\supn)=\nvol(\ord_{S_n})
			\leq Mm_2\supn\cdot\big(\frac{2}{m_2\supn}\big)^2
			=\frac{4M}{m_2\supn}\to 0 \text{ as } n\to+\infty,
		\end{equation*}
		which contradicts \eqref{eqn:nvol bdd}.
		
		{\bf a3)} If $m_1\supn\to+\infty$, $m_2\supn\to+\infty$ as $n\to+\infty$, then by \eqref{eqn:correction bdd} we get $\sum d_j\supn>2$ for $n\gg 0$, which contradicts the condition that $(S_n,\Delta_{S_n})$ is log Fano.
		
		\noindent	{\bf Case b)}. $\sum r_i\supn a_i\supn=0$ for $n\gg 0$. In this case, since $(S_n,\Delta_{S_n})$ is $K$-semistable, by Proposition \ref{prop:K-ss condition} we must have $d_1=d_2$, that is
		\begin{equation}\label{eqn:ss condition for 2 pts}
			\frac{m_1\supn-1+\sum_i r_{1,i}\supn a_i\supn}{m_1\supn}=
			\frac{m_2\supn-1+\sum_i r_{2,i}\supn a_i\supn}{m_2\supn}
		\end{equation}
		It follows from \eqref{eqn:formula for nvol} that
		\begin{equation}\label{eqn:nvol for 2 pts}
			\nvol(o,\bA^2,\Delta\supn)=4(1-\sum r_{1,i}\supn a_i\supn)(1-\sum r_{2,i}\supn a_i\supn),
		\end{equation}
		which cannot be strictly increasing by \cite[Proposition 3.4.1]{HMX14} again. This contradiction finishes the proof.
	\end{proof}
	
	\begin{remark}
	    One can also prove Theorem \ref{thm:existence of delta-plt blow-up} for $o\in(\bA^2,\Delta)$ using essentially the same calculation as above, we leave this to the reader.
	\end{remark}
	
Note that we can get $c_{\min}(2)\le 64$ by the proof of Theorem \ref{thm:estimate on lct}. Using the following result, one can improve this estimate with $c_{\min}(2)\le 8$, as stated in Remark \ref{rmk:sharp_constant}.
	
	\begin{theorem}\label{thm:surface lct}
		Let $o\in(\bA^2,\Delta)$ be a klt singularity, where $\Delta\ge 0$ is a $\bQ$-divisor. 
		Then we have
		\begin{equation*}
			8\cdot \lct(\bA^2,\Delta;\Delta)\ge \nvol(o,\bA^2,\Delta).
		\end{equation*}
	\end{theorem}
	
	\begin{proof}
		We may assume that $\nvol(o,\bA^2,\Delta)=\epsilon< 4$. 
		By \cite[Theorem 4.14]{LLX18} and Theorem \ref{thm:minimizer is K-ss}, 
		we may assume that there is a plt blow-up $\mu:(Y,S)\to(\bA^2,o)$ of $o\in(\bA^2,\Delta)$ such that $(S,\Delta_S\coloneqq\Diff_S \Delta)$ is a $K$-semistable log Fano pair.
		
		By Proposition \ref{prop:weighted blow-up} we know that $\Delta_S$ is of the form
		\begin{equation}\label{eqn:expression for Delta_S}
			\Delta_S=\sum_{j=1}^l\frac{m_j-1+\sum_i r_{j,i}a_i}{m_j}p_j=\sum_j d_jp_j,
		\end{equation}
		where $d_j\neq 0$, $m_1$ and $m_2$ are coprime, and $m_j=1$ for $j>2$. 
		
		We first assume that $l\geq 3$. As $(S,\Delta_S)$ is klt, we know $\sum_ir_{j,i}a_i<1$ for any $j$. Set $d_{j_0}\coloneqq\max\{d_j\}$, 
		then by Proposition \ref{prop:K-ss condition} we have $d_{j_0}\leq \sum_{j\neq j_0}d_j$, so for any $1\leq j\leq l$,
		\begin{equation*}
		    \frac{1-\sum_i r_{j,i}a_i}{m_j}=1-d_j\geq
			1-d_{j_0}
			\geq 1-\sum_{j\neq j_0}d_j,
		\end{equation*}
		so by \eqref{eqn:formula for nvol} we get
		\begin{equation}\label{eqn:first control for coefficients}
			4m_1m_2\left(\frac{1-\sum_i r_{j,i}a_i}{m_j}\right)^2\geq
			m_1m_2(2-\sum_j d_j)^2=\nvol(o,\bA^2,\Delta)= \epsilon
		\end{equation}
		Denote by $m\coloneqq\min\{m_1,m_2\}$ and $M\coloneqq\max\{m_1,m_2\}$.
		Then by \eqref{eqn:first control for coefficients} we have $\frac{m}{M}\geq\frac{\epsilon}{4}$, so
		\begin{equation}\label{eqn:control for first coefficients}
			1-\sum_ir_{j,i}a_i\geq
			\frac{\epsilon}{4} \text{ for } j\leq 2,
		\end{equation}
		and
		\begin{equation}\label{eqn:log Fano}
			2-\sum_j d_j= \sqrt{\frac{\epsilon}{mM}}
			\geq \frac{\epsilon}{2m}.
		\end{equation}
		To estimate $d_j$ for $j>2$, we treat two subcases seperately. If $m\geq 3$, then from \eqref{eqn:log Fano} we can see $\sum_{j>2}d_j\leq 2-d_1-d_2\leq \frac{2}{m}\leq \frac{2}{3}$, so
		\begin{equation}\label{eqn:control for last coefficients_large m}
			1-d_j\geq \frac{1}{3}.
		\end{equation}
		If $m\leq 2$, then by \eqref{eqn:first control for coefficients} we have 
		$	(1-d_j)^2= \frac{\epsilon}{4mM}\geq\frac{\epsilon}{4m}\cdot\frac{\epsilon}{4m}\geq\frac{\epsilon^2}{64}$,
		so 
		\begin{equation}\label{eqn:control for last coefficients_small m}
			1-d_j\geq\frac{\epsilon}{8}.
		\end{equation}
		
		Now set $\Delta'\coloneqq (1+\frac{\epsilon}{8})\Delta$, then combining \eqref{eqn:control for first coefficients}, \eqref{eqn:control for last coefficients_large m} and \eqref{eqn:control for last coefficients_small m}, we know that the pair $(S,\Delta_S'=\sum d'_jp_j)$ is lc, where $\Delta_S'\coloneqq\Diff_S \Delta'$. So by inversion of adjunction $(Y,S+\mu^{-1}\Delta')$ is lc.
		
		On the other hand, we have
		\begin{equation*}
			\begin{aligned}
				2-\sum_jd'_j=&\frac{1-(1+\epsilon/8)\sum_i r_{1,i}a_i}{m_1}+\frac{1-(1+\epsilon/8)\sum_i r_{2,i}a_i}{m_2}-\sum_{j>2}(1+\epsilon/8)d'_j\\
				\geq&2-\sum_jd_j-\frac{\epsilon}{8}(\frac{1}{m_1}+\frac{1}{m_2}+\sum_{j>2}d_j)
				\geq \frac{\epsilon}{2m}-\frac{\epsilon}{8}\cdot\frac{4}{m}=0,
			\end{aligned}
		\end{equation*}
		where the second inequality follows from \eqref{eqn:log Fano}. So $A_{(\bA^2,\Delta')}(S)\ge 0$, and we know that $o\in(\bA^2,\Delta')$ is lc. Hence $8\cdot \lct(\bA^2,\Delta;\Delta\ge\nvol(o,\bA^2,\Delta)$.
		
		If $l=2$, then we are in case b) of the proof of Theorem \ref{thm:ACC for Q coeff}. By \eqref{eqn:nvol for 2 pts} we have 
		\begin{equation*}
			\epsilon=\nvol(o,\bA^2,\Delta)
			=\frac{4m_2(1-\sum r_{1,i} a_i)^2}{m_1}
			=\frac{4m_1(1-\sum r_{2,i} a_i)^2}{m_2},
		\end{equation*}
		so we get $1-\sum_i r_{j,i}a_i>\frac{\epsilon}{4}$ for $j=1,2$. Moreover, if we set $\Delta'=(1+\frac{\epsilon}{4})\Delta$ and $\Delta'_S=\Diff_S \Delta'$, then we have 
		\begin{equation*}
			\begin{aligned}
				A_{(\bA^2,\Delta')}(S)=m_1 m_2 \cdot\deg(-(K_S+\Delta_S'))
				=2(m_2)^2\left(1-(1+\frac{\epsilon}{4})\sum r_{1,i} a_i\right)>0.
			\end{aligned}
		\end{equation*}
		So $(\bA^2,\Delta')$ is klt. Thus in this case we get $4\cdot \lct(\bA^2,\Delta;\Delta)>\nvol(o,\bA^2,\Delta)$.
	\end{proof}

\end{document}